\documentclass[12pt]{article}
\usepackage{hyperref}
\usepackage{fullpage}
\usepackage{amsmath,bm}
\usepackage{amsfonts}
\usepackage{graphicx}
\newtheorem{theorem}{Theorem}

\newtheorem{proposition}[theorem]{Proposition}

\newenvironment{proof}[1][Proof]{\noindent\textbf{#1.} }{\ \rule{0.5em}{0.5em}}

\begin{document}

\title{Simplified normal forms\\ near a degenerate elliptic fixed point\\
in  two-parametric families of area-preserving maps
}
\author{Vassili Gelfreich${}^1$\footnote{This work was partially
supported by a grant from EPSRC.}{}~{}
and
Natalia Gelfreikh${}^2$
\\[24pt]
${}^1$\small Mathematics Institute, University of Warwick,\\
\small Coventry, CV4 7AL, UK\\
\small E-mail: v.gelfreich@warwick.ac.uk\\[12pt]
${}^2$\small Department of Higher Mathematics and Mathematical Physics,\\
\small Faculty of Physics, Faculty of Liberal Arts and Sciences\\ 
\small St.Petersburg State University, Russia\\
\small E-mail: gelfreikh@mail.ru}
\date{December 6, 2013}

\maketitle

\begin{abstract}
We derive simplified normal forms for an area-preserving map in a neighbourhood
of  a degenerate resonant elliptic fixed point. Such  fixed points appear in generic
two-parameter families of area-preserving maps. We also derive a simplified
normal form for a generic two-parametric unfolding. The normal forms
are used to analyse bifurcations of $n$-periodic orbits.
In particular, for $n\ge6$ we find regions of parameters where the
normal form has ``meandering'' invariant curves.
\end{abstract}

\section{Introduction}

In a Hamiltonian system small oscillations around a periodic orbit 
are often described using the normal form theory~\cite{SM,AKN}. In the case of two degrees of freedom
the Poincar\'e section is used to reduce the problem 
to studying a family of area-preserving maps in a neighbourhood of a fixed point. The Poincar\'e map depends 
on the energy level and possibly on other parameters involved in  the problem. 
A sequence of coordinate changes is used 
to transform the map to a normal form. In the absence of resonances 
the normal form is a rotation of the plane, and the angle
of the rotation depends on the amplitude. In the presence of resonances
the normal form is more complicated.

Let us describe the normal form
theory in more details.
Let $F_0:\mathbb{R}^{2}\to \mathbb{R}^{2}$ be an area-preserving map. We assume that $F_0$
 also preserves orientation. Let the origin be a fixed point:
\begin{equation*}
F_0(0)=0.
\end{equation*}%
Since $F_0$ is area-preserving $\det DF_0(0)=1$. Therefore
the two eigenvalues of the Jacobian matrix  
$DF_0(0)$ are $\lambda_0 $ and $\lambda_0 ^{-1}$.
These eigenvalues are often called {\em multipliers\/}
of the fixed point.  We will consider an elliptic fixed point when $\lambda_0$ is not real.
As the map is real-analytic the multipliers are 
complex conjugate, i.e. $\lambda_0^{-1}=\lambda_0^*$,
and consequently belong to the unit circle, i.e. $|\lambda_0|=1$. 
Note that in our case $\lambda_0 \ne\pm 1$ as it is assumed non-real.

There is a linear area-preserving change of variables
such that the Jacobian of $F_0$ takes the form of
a rotation:
\begin{equation}\label{Eq:rotation}
DF_0(0)=R_{\alpha_0 }\qquad\mbox{where }R_{\alpha_0}=\left(
\begin{array}{cc}
\cos \alpha_0 & -\sin \alpha_0 \\
\sin \alpha_0 & \cos \alpha_0%
\end{array}%
\right).
\end{equation}
where the rotation angle $\alpha_0$ is related to the multiplier: $\lambda_0=e^{i\alpha_0}$. 

The classical normal form theory   \cite{AKN} states that there 
is a formal area-preserving change of coordinates which transforms $F_0$ into the resonant normal form
$N_0$ such that the formal series $N_0$ commutes with the rotation: $N_0R_{\alpha_0}=R_{\alpha_0} N_0$.
Following  the
method suggested in \cite{Tak74} (see e.g. \cite{BGSTT1993,GG2009}), we 
consider a formal series $H_0$ such that
\begin{equation}
N_0=R_{\alpha_0}\Phi^1_{H_0}
\end{equation}
where $\Phi_{H_0}^t$ is a flow generated by the Hamiltonian $H_0$. The Hamiltonian is invariant
with respect to the rotation $H_0\circ R_{\alpha_0}=H_0$. It follows that the normal form
preserves the Hamiltonian: $H_0\circ N_0=H_0$ since the Hamiltonian flow also preserves $H_0$.
Returning to the original coordinates we conclude that the original map has a non-trivial formal integral
$\hat H_0$.
This integral provides a powerful tool for analysis of the local dynamics
including the stability of the fixed point (see e.g. \cite{Arnold1961}).

It is natural to describe the normal forms using
 the symplectic polar coordinates
$(I,\varphi)$ defined by the equations
\begin{equation}\label{Eq:polar}
x=\sqrt{2I}\cos\varphi\,,\quad
y=\sqrt{2I}\sin\varphi\,.
\end{equation}
A fixed point is called {\em resonant} if there exists $n\in \mathbb{N}$ such that
$\lambda_0^{n}=1$. The least positive $n$ is called the {\em order of the resonance}.
In the resonant case  the formal Hamiltonian takes the form \cite{AKN}:
\begin{equation}\label{Eq:inter}
H_0(I,\varphi)=I^2\sum_{k\ge 0} a_kI^k+\sum_{k\ge 1}\sum_{j\ge0} b_{kj}I^{kn/2+j}\cos (k n\varphi +\beta_{kj})
\end{equation}
where $a_k$, $b_{kj}$ and $\beta_{kn}$ are real coefficients. 
The normal form is not unique.
In the paper \cite{GG2009} we proved that in the non-degenerate case, namely if 
$a_0 b_{10}\ne 0$ for $n \ge 4$ or $b_{10}\ne 0$ for $n=3$, the Hamiltonian
can be simplified substantially: there is a canonical formal change of variables
which transforms the Hamiltonian to the form
\begin{equation}\label{Eq:inter0}
H_0(I,\varphi)=I^2 A(I)+I^{n/2} B(I)\cos (n\varphi)\,,
\end{equation}
where $A$ and $B$ are formal series in powers of $I$.
We refer the reader to paper \cite{GG2009} for the detail discussions of possible
simplifications for the formal series $A$ and $B$ which depend on the order 
of the resonance~$n$. In particular, for $n\ge4$ the series $B(I)$ contains even powers only.
We note that (\ref{Eq:inter0}) involves formal series in a single variable $I$
only and therefore is substantially simpler than  (\ref{Eq:inter}) which involves series in two variables.

In this paper we consider the case when the fixed point of $F_0$ has a degeneracy of the lowest 
co-dimension, i.e.,
in  (\ref{Eq:inter}),
\begin{equation}\label{Eq:nondegen1}
a_0=0,\qquad a_1\ne0,\qquad b_{10}\ne0\,.
\end{equation}
We prove that 
under these assumptions the normal form Hamiltonian (\ref{Eq:inter}) can be transformed to
a simplified form which looks similar to (\ref{Eq:inter0}). The detailed description
of the simplified normal form is given in Section~\ref{Se:normalforms}.
We also prove that the simplified normal form is unique and therefore in general there is no room
for further simplifications.

It is natural to consider a generic unfolding by considering
the map $F_0$ as a member of a generic two-parameter family. In Section~\ref{Se:normalforms}
we provide simplified normal forms for such families and discuss their uniqueness.

Of course, in general the series of the normal form theory are expected to diverge. 
Nevertheless they provide rather accurate information about the dynamics
of the original map. For example, it follows that 
 for any $m\in\mathbb N$ 
the partial sum $\hat H_0^m$ (a polynomial which includes all
terms of $\hat H_0$ up to the order $m$) satisfies
$$
\hat H_0^m\circ F_0-\hat H_0^m=O_{m+1}\,,
$$
where $O_{m+1}$ stands for an analytic function 
which has a Taylor series without terms of order less than $m+1$.
In particular, the implicit function theorem can be used
to show that the map $F_0$ has $n$-periodic points near
critical points of the Hamiltonian $\hat H_0^m$,
and KAM theory can be used to establish existence
of invariant curves \cite{Arnold1961,AKN}.

In the case of a generic one-parameter unfolding of a non-degenerate
resonant elliptic fixed point, the normal form 
provides a description for a chain of islands which is born from the
origin when the unfolding parameter changes its sign~\cite{AKN,Meyer1970,SM,SV2009}.
A typical picture is shown on Figures~\ref{Fig:islands7}(a) and~(b).
This analysis distinguishes the cases of week ($n\ge5$) and strong ($n\le4$) resonances.

\begin{figure}[t]
\begin{center}
\includegraphics[width=3.8cm]{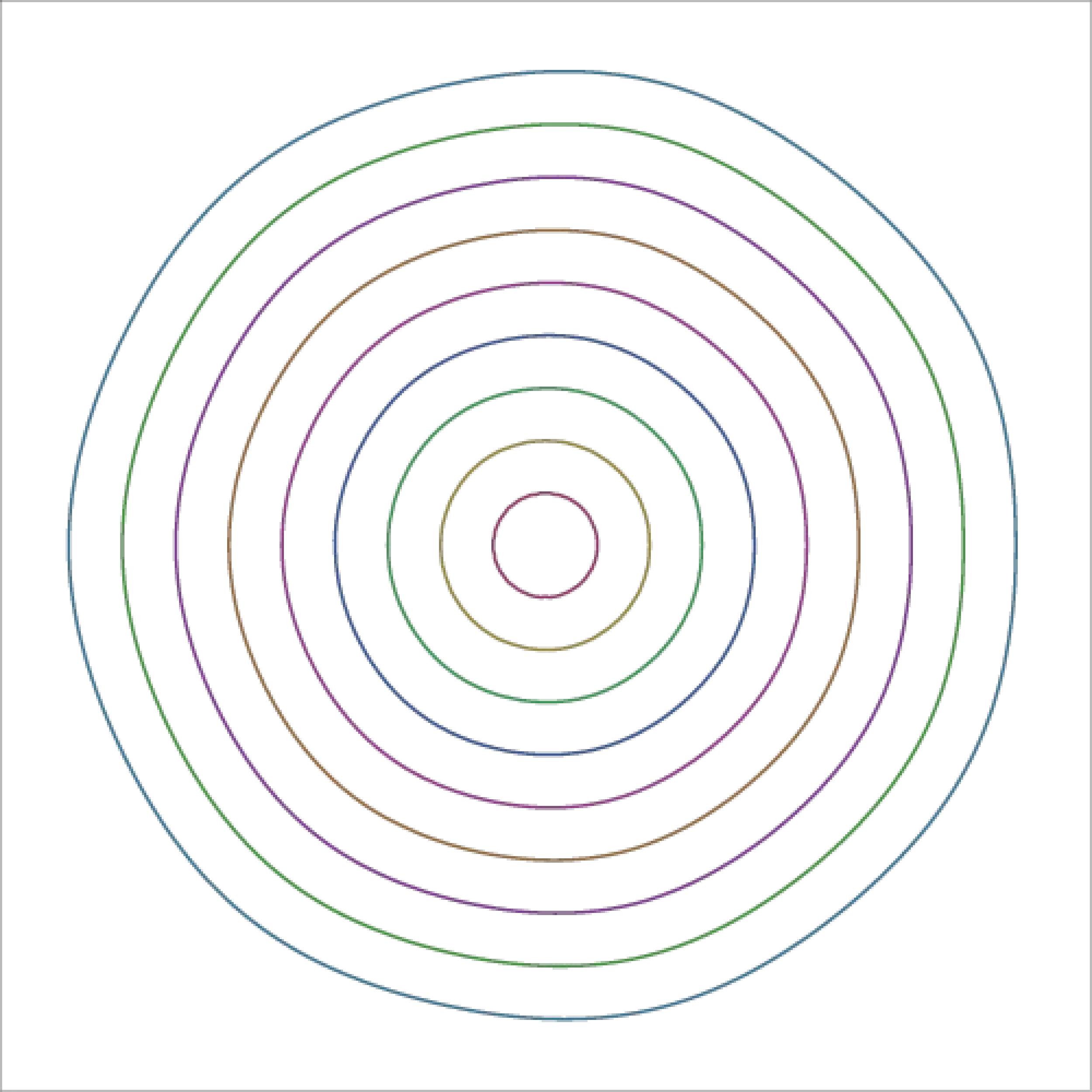}\quad
\includegraphics[width=3.8cm]{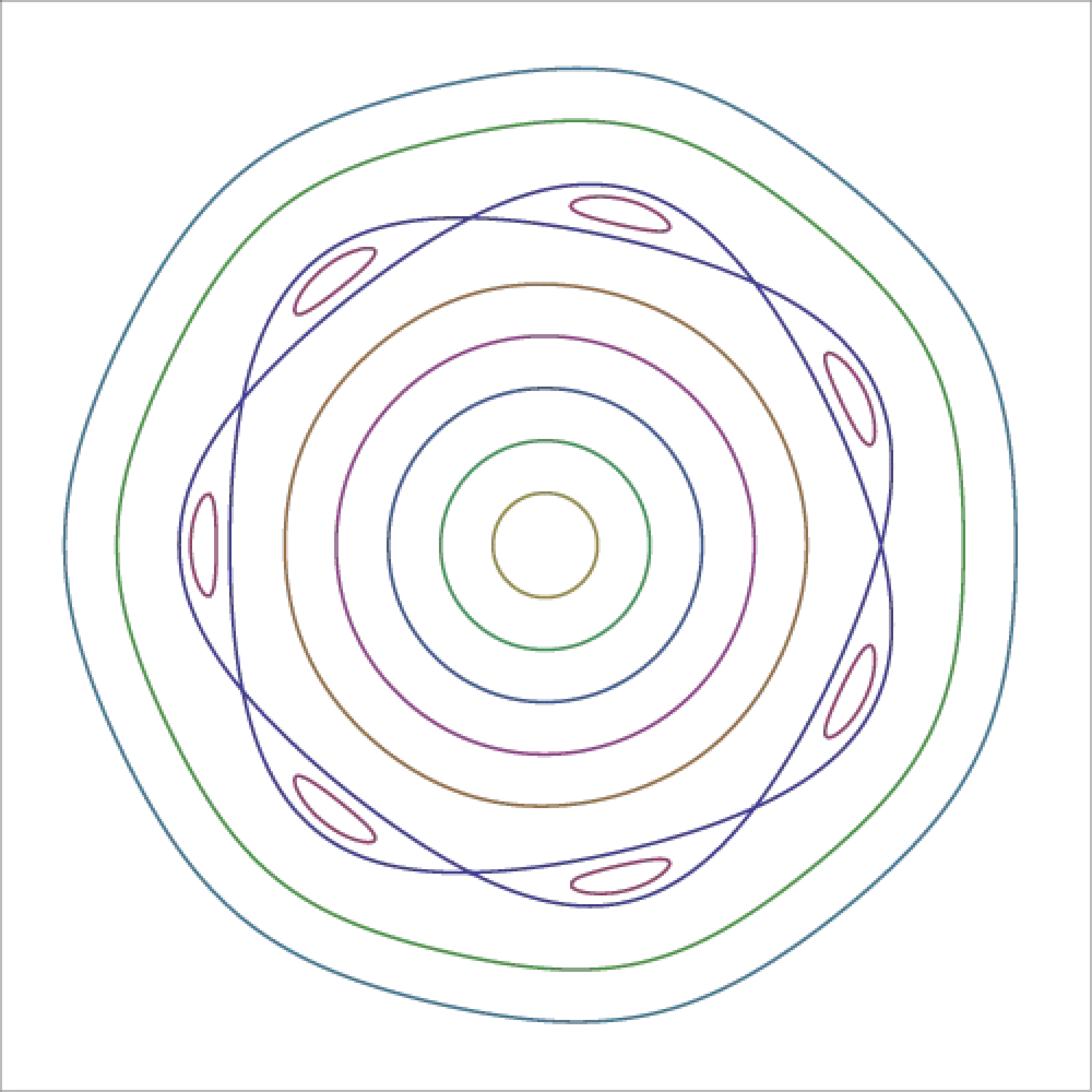}\\
(a)\kern3.8cm(b)
\end{center}
\caption{Level lines of the normal form Hamiltonian: typical pictures for 
$n=7$. 
\label{Fig:islands7}}
\end{figure}

In this paper we use the normal forms to describe  bifurcations of $n$-periodic orbits which appear
in a generic two-parameter unfolding of a degenerate map $F_0$.
In Section~\ref{Se:models} we analyse bifurcations of $n$-periodic points using 
model normal form Hamiltonians for $n\ge7$, $n=6$ and $n=5$ separately.
It is interesting to note that  bifurcation diagrams contain  sectors where
the leading part of the normal form does not satisfy the twist condition required by a standard KAM theory
and invariant curves of ``meandering type" are observed similar to ones
described in the papers~\cite{Simo1998,DL2000,Haro2002,FC2003}.

Finally we note that Theorem~\ref{Thm:h0_n34} of the next section imply that in the case $n=3,4$
the lowest degeneracy has co-dimension three.

\section{Unique normal forms\label{Se:normalforms}}

Suppose that $F_0$ is a real-analytic  area-preserving map with an elliptic fixed point at the origin.
We note the theory repeats almost literally in the $C^\infty$ category.
Let $\lambda_0$ be the multiplier of the fixed point and suppose that 
$\lambda_0^n=1$ for some $n>2$ and $\lambda_0^k\ne1$ for all integer values of $k$, $1\le k<n$.
Then there is an analytic area-preserving change of variables which eliminates
all non-resonant terms from the Taylor series of $F_0$ up to the order $n+1$ \cite{SM,AKN}.
So we assume that $F_0$ is already transformed to this form.

It is convenient to rewrite $F_0$ identifying the plane with coordinates $(x,y)$ and  the complex plane of 
the variable $z=x+iy$. Equation (\ref{Eq:polar}) implies  $z=\sqrt{2I}e^{i\varphi}$.
In the complex notations the transformed map $F_0$ takes the form
\begin{equation}\label{Eq:prenorm0}
 f_0=a(z\bar z)z + b(z\bar z)\bar z^{n-1}+O(|z|^{n+1})
\end{equation}
where $a,b$ are polynomial in $z\bar z$. The coefficients of the polynomials are
complex, in particular, $a(0)=\lambda_0$. In \cite{GG2009} we studied the
normal forms of the map in the generic case when $a'(0)\ne0$
and $b(0)\ne0$. A resonant fixed point with the degeneracy of the twist term
corresponds to 
$a(0)^n=1$ and $a'(0)=0$.
We note that in the space of real area-preserving maps 
such maps are of co-dimension two. Indeed, 
the area-preserving property of the map implies that $a(0)$ is on the unit
circle and $a(0)^{-1}a'(0)$ is purely imaginary. Consequently, in a generic
situation two real parameters are required to satisfy these two condition
simultaneously. Note that another obvious degeneracy of a resonant
fixed point corresponds to  $b(0)=0$. 
As the coefficient $b(0)$ is complex,
this 
degeneracy  has co-dimension three and is not considered in this paper.

\begin{theorem}\label{Thm:h22eq0}
If  $n\ge5$, $a'(0)=0$, $a''(0)\ne0$ and $b(0)\ne0$,
then there is 
a formal canonical change of variables which conjugates $F_0$ with $R_{\alpha_0}\circ\Phi^1_{H_0}$,
where the formal Hamiltonian $H_0$ has the form
\[
H_0(I,\varphi )=I^{3} A_0(I)+I^{n/2} B_0(I)\cos n\varphi,
\]
where $A_0(I)=\sum_{k\ge0} a_k {I}^k$ and $B_0(I)=\sum_{k\ge0} b_k {I}^k$
are formal series in $I$ with real coefficients such that 
\begin{itemize}
\item
if $n=5$, then $a_k=0$ for $k=1 \pmod 5$ and $b_k=0$ for $k=4 \pmod5$;

\item
if $n\ge 6$, then $b_k=0$ for $k=2 \pmod 3$.
\end{itemize}
The coefficients of the series $A_0$ and $B_0$ are defined uniquely by the map $F_0$
provided the leading order is normalised to ensure $b_0>0$. 
\end{theorem}

This theorem follows  from Proposition~\ref{Pro:h22is0} of Section~\ref{Se:h22is0}.
The proofs of this and other theorems from this section 
are based on a development of the Lee series method~\cite{Deprit1969,Rink2001,GG2009,GG2010}. 
We note that Theorem~\ref{Thm:h22eq0} provides a complete set of formal
invariants for the map $F_0$.

\medskip

The following theorem covers the cases of $n=3$ and $n=4$.

\begin{theorem}\label{Thm:h0_n34}
If  $n\in\{\,3,4\,\}$ and  $b(0)\ne0$,
 then there is 
a formal canonical change of variables which conjugates $F_0$ with $R_{\alpha_0}\circ\Phi^1_{H_0}$,
where the formal Hamiltonian $H_0$ has the form
\[
H_0(I,\varphi )=I^{2} A_0(I)+I^{n/2} B_0(I)\cos n\varphi,
\]
where $A_0(I)=\sum_{k\ge0} a_k {I}^k$ and $B_0(I)=\sum_{k\ge0} b_k {I}^k$
are formal series in $I$ with real coefficients such that 
\begin{itemize}
\item
if $n=3$, then $a_k=0$ for $k=0\pmod3$ and $b_k=0$ for $k=2 \pmod 3$;

\item
if $n=4$, then $a_k=0$ for $k=1 \pmod 4$ and $b_k=0$ for $k=3 \pmod 4$;

\end{itemize}
The coefficients of the series $A_0$ and $B_0$ are defined uniquely by the map $F_0$
provided the leading order is normalised to ensure $b_0>0$. 
\end{theorem}

The case $n=3$ was discussed in \cite{GG2009}. The proof for $n=4$ is in Section~\ref{Se:h22is0}.
In this case, Theorem~\ref{Thm:h0_n34} has an advantage over the normal form of the paper
 \cite{GG2009} as the present version eliminates assumptions on the coefficients of 
 the twist part of the map which allows a unified treatment for a wider class of maps
 without increasing the complexity of the simplified normal form.

\medskip

Let $F_p$ be a two-parametric unfolding of $F_0$, where $p=(p_1,p_2)$.
We assume that 
\begin{enumerate}
\item[(1)] the family $F_p$  has an elliptic fixed point
at the origin, i.e., $F_p(0)=0$ for all small $p$; 
\item[(2)] $F'_p(0)=R_{\alpha_p}$ where the angle $\alpha_p$ is an analytic function of $p$ and
$\alpha_0=2\pi\frac{m}{n}$, i.e., $p=0$ corresponds to a resonance of order $n$.
\end{enumerate}
Suppose that $p$ is sufficiently small to ensure that $|\alpha_p-\alpha_0|<\frac{\pi}{2n^2}$. Then
 $\alpha_p$ does not satisfy any resonant condition of order less than $2n$ excepting the original one
 which corresponds to $\alpha_p=\alpha_0$.
According to the normal form theory,  there is an analytic area-preserving change of variables
which eliminates all non-resonant terms from the Taylor series of $F_p$
up to the order $n+1$.
We assume that $F_p$ is already transformed to this form.
Then in the complex coordinates the map $ F_p$ has the form
\begin{equation}\label{Eq:prenorm}
f_p=a_p(z\bar z)z + b_p(z\bar z)\bar z^{n-1}+O(z^{n+1})
\end{equation}
where $a_p$ and $b_p$ are polynomial in $z\bar z$, which depend analytically on  $p$,
and $a_p(0)=\lambda_p$.
We introduce new parameters $\delta$ and $\nu$ by
\begin{equation}\label{Eq:deltanu}
\delta=i\log\frac{\lambda_p}{\lambda_0}=\alpha_0-\alpha_p
\qquad\mbox{and}\qquad \nu=\frac{ia_p'(0)}{\lambda_p}
.
\end{equation}
The preservation of the area by the map $ f_p$ implies that
 $\delta$, $\nu$ are real. Moreover, they are defined uniquely,
i.e., they are independent of the choice of the coordinate change $C_p$
which transforms the original map $F_p$ to the form (\ref{Eq:prenorm}).

If  the Jacobian of the map $(p_1,p_2)\mapsto (\delta,\nu)$ does not vanish
at $p=0$, the family $F_p$ is  a generic unfolding of $F_0$. 
Then the inverse function theorem implies that $p$ can be expressed in terms of $(\delta,\nu)$.
 From now on we use $p=(\delta,\nu)$ instead of the original parameters.

The normal form theory states that there is a formal change of variables which transforms $F_p$
to its normal form $N_p=R_{\alpha_0}\circ \Phi^1_{H_p}$ where the formal sum $H_p$ includes only 
resonant monomials, i.e.,  monomials of the form $h_{j,k,l,m}z^j\bar z^k\delta^l\nu^m$ with $k=j\pmod n$.
The transformation to the normal form and the normal form itself are  formal power series in $z,\bar z$ and $\delta,\nu$.

We note that the normal form is  not unique and can be simplified using a formal tangent-to-identity
change of variables.

\begin{theorem} \label{Thm:nf4families}
Let $n\ge5$.
If $F_0$ has $a_0'(0)=0$, $a_0''(0)\ne0$ and $b_0(0)\ne 0$, 
and $F_p$ is a generic unfolding of $F_0$
with parameter  $p=(\delta,\nu)$, 
 then there is 
a formal canonical change of variables which conjugates $F_p$ with $R_{\alpha_0}\circ\Phi^1_{H_p}$,
where the formal Hamiltonian $H_p$
has the form
\begin{equation}\label{Eq:nfsimpl}
H_{\delta , \nu}(I, \varphi  )=\delta I +\nu I^2 +I^3 A(I; \delta , \nu)+ I^{n/2} B(I ;\delta , \nu)\cos  n\varphi,
\end{equation}
where 
$$
A(I; \delta , \nu)=\sum_{k,m,j \ge 0} a_{kmj} I^k \delta^m \nu^j, 
\qquad
B(I;\delta , \nu) = \sum_{k,m,j \ge 0} b_{kmj} I^k \delta^m  \nu^j   ,
$$
and
\begin{itemize}
\item
if $n=5$, then $a_{kmj}=0$ for $k=1 \pmod 5$ and $b_{kmj}=0$ for $k=4 \pmod5$;
\item
if $n\ge 6$, then $b_{kmj}=0$ for $k=2 \pmod 3$.
\end{itemize}
Moreover, the coefficients of the series $A(I; \delta , \nu)$ and
$B(I ;\delta , \nu)$ are defined uniquely.

\end{theorem}

The proof of this  theorem is in Section~\ref{Se:Familiesh22}.

\medskip

Theorem~\ref{Thm:nf4families} implies that for $n\ge5$ the qualitative properties of the normal form
can be studied using the following model Hamiltonian:
\begin{equation}\label{Eq:modelN}
h(I,\varphi)=\delta I+\nu I^2+I^3+I^{n/2}\cos n\varphi\,.
\end{equation}
This Hamiltonian keeps only the leading terms of the formal series $A$ and $B$ from the
simplified normal form (\ref{Eq:nfsimpl}).
Note that the normalisation of the coefficients in front of $I^3$ and $I^{n/2}$ 
does not restrict generality of the model, as under the assumptions
of the theorem the normalisation used in 
(\ref{Eq:modelN}) can be achieved by rescaling the variable $I$, the parameters $\delta$ and~$\nu$, 
and the Hamiltonian function $h$.

\medskip

We note that in the case $n=3$ and~$4$, the degeneracy
in the twist term does not affect the reduction to the normal
form.
 Let $\alpha_0=\frac{2\pi}n$ and $\delta=\alpha_0-\alpha$.
Then (\ref{Eq:prenorm}) implies that 
 the map can be represented in the form
$$
f_{\delta,\nu}(z,\bar z)=e^{i(\alpha_0-\delta)}z+a_{\nu,\delta} z^2 \bar z+b_{\nu,\delta} \bar z^{n-1}+O(z^4)
$$
where $a_{\nu,\delta}$ and $b_{\nu,\delta}$ are analytic functions of the parameters.

\begin{theorem} 
Let $\alpha_0=\frac{2\pi}{n}$ with $n=3$ or $4$, and 
 $b_{0,0}\ne 0$. Then for all sufficiently small values of $\nu$
there is 
a formal canonical change of variables which conjugates $F_{\delta,\nu}$ with
$R_{\alpha_0}\circ\Phi^1_{H_{\delta,\nu}}$
where
a formal Hamiltonian $H_{\delta,\nu}$ has the form
\begin{equation}\label{Eq:nfsimpl34}
H_{\delta , \nu}(I, \varphi  )=\delta I +I^2 A(I,\delta ,\nu)+ I^{n/2} B(I ,\delta ,\nu)\cos  n\varphi,
\end{equation}
where 
$$
A(I, \delta ,\nu)=\sum_{k,m,j \ge 0} a_{km}(\nu) I^k \delta^m , 
\qquad
B(I,\delta,\nu) = \sum_{k,m,j \ge 0} b_{km}(\nu) I^k \delta^m   ,
$$
where the coefficients are real-analytic functions of $\nu$
and
\begin{itemize}
\item
if $n=3$, then $a_{km}=b_{km}=0$ for $k=2 \pmod 3$;
\item
if $n=4$, then $a_{km}=0$ for $k=1 \pmod 4$ and $b_{km}=0$ for $k=3 \pmod 4$.
\end{itemize}
Moreover, the coefficients  $a_{km}( \nu)$
and $b_{km}(\nu)$ are defined uniquely.

\end{theorem}

We note that the theorem implies that in the cases of $n=3$ and $n=4$ the bifurcations
of $n$-periodic points in a generic two-parameter family
 are similar to the generic one-parameter case. This statement
can be checked by introducing properly scaled variables:
if $n=3$, then $J=\delta^{-2}I$ and $\bar h=\delta^{-3}H_{\delta\nu}$. After
an additional change of variables, which depends on the value of $b_0(0)$,
the leading part of the Hamiltonian takes the form $\bar h_0=J+J^{3/2}\cos3\varphi$.
If $n=4$, then  $J=\delta^{-1}I$ and $\bar h=\delta^{-2}H_{\delta\nu}$.
After an additional change of variables, which depends on the value of $b_0(0)$,
the leading part of 
the Hamiltonian takes the form  $\bar h_0=J+J^{2}(a(\nu)+\cos3\varphi)$.
In both cases, the scaling leads to the same Hamiltonian as in the one-parameter case.
We note that for $n=4$ there is a bifurcation which corresponds
to the transition between the stable and unstable sub-cases of the non-degenerate case (i.e. 
the coefficient $a$ crosses $\pm1$).

%%%%%%%%%%%%%%%%%%%%%%%%%%%%%%%%%%%%%%%%%%%%%%%
\section{Bifurcations of $n$-periodic points\label{Se:models}}

In order to study bifurcations of critical points for the normal form 
we use model Hamiltonian functions which involve the least
possible number of terms while preserving the qualitative properties of 
the general case described by the normal form (\ref{Eq:nfsimpl}). 
 In particular, we provide
bifurcation diagrams for saddle critical points located in
a small neighbourhood of the origin 
and illustrate the structure of corresponding critical level sets.

\subsection{$n\ge7$}

For $n\ge7$ the model Hamiltonian function takes the form
\begin{equation}\label{Eq;hn}
h=\delta I+ \nu I^2+I^3+I^{n/2}\cos n\varphi
\end{equation}
where $I,\varphi$ are symplectic polar coordinates~(\ref{Eq:polar}).

If $\delta=\nu=0$, the Hamiltonian $h$ has a minimum at the origin. 
All level lines of $h$  are closed  curves around the origin (similar to Fig.~\ref{Fig:islands7}(a)). 
Therefore the origin is a stable equilibrium of the normal form. Moreover,
since $I^3\gg I^{n/2}$,
it remains stable for all sufficiently small values of $\delta$ and $\nu$.

In order to analyse the Hamiltonian for small $\delta$ and $\nu$, we recall that
critical points of $h$ are defined 
 from the equations $\partial_Ih=\partial_\varphi h=0$.
After computing the derivatives, we conclude that 
all equilibria have either $\varphi=0$ or $\frac{\pi}n\pmod{\frac{2\pi}{n}}$,
and
\begin{equation}\label{Eq:ficpoint}
\delta +2\nu I+3I^2+\sigma_\varphi  \frac{n}{2}I^{ n/2-1}=0
\end{equation}
where $\sigma_\varphi=\cos n\varphi\in\{+1,-1\}$.
The equation (\ref{Eq:ficpoint}) can be written in the form 
\begin{equation}\label{Eq:fd}
\delta=f_\sigma(I,\nu)
\end{equation}
where
$$
f_\sigma=-2\nu I-3I^2-\sigma \frac{n}{2}I^{ n/2-1}\,.
$$ 
A typical plot of the functions $f_\sigma$ is shown on Fig.~\ref{Fig:fsigma}(a) and~(b)
for $\nu<0$ and $\nu>0$ respectively.
%%%%%%%%%%%%%%%%%%%%%%%%%%%%%%%%%%%%%%%%%%%%%
\begin{figure}
\begin{center}
\includegraphics[width=4.5cm]{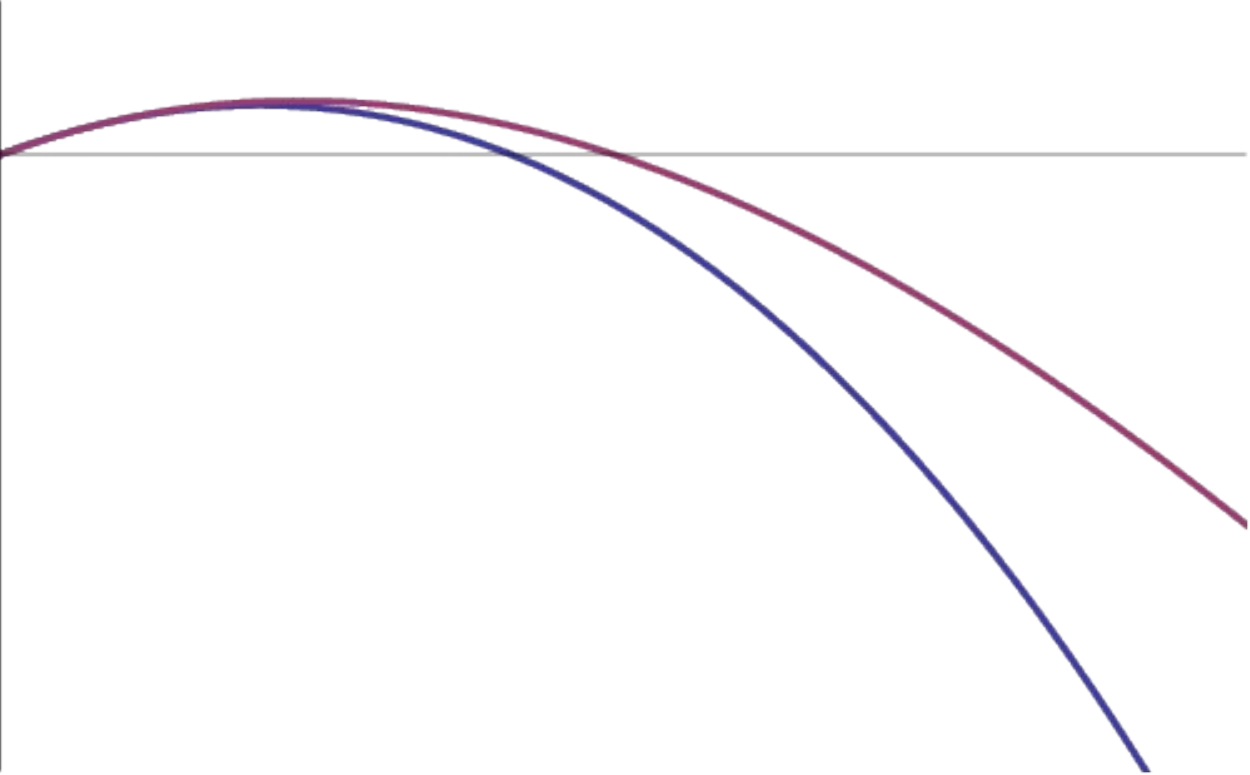}
\hskip1cm
\includegraphics[width=4.5cm]{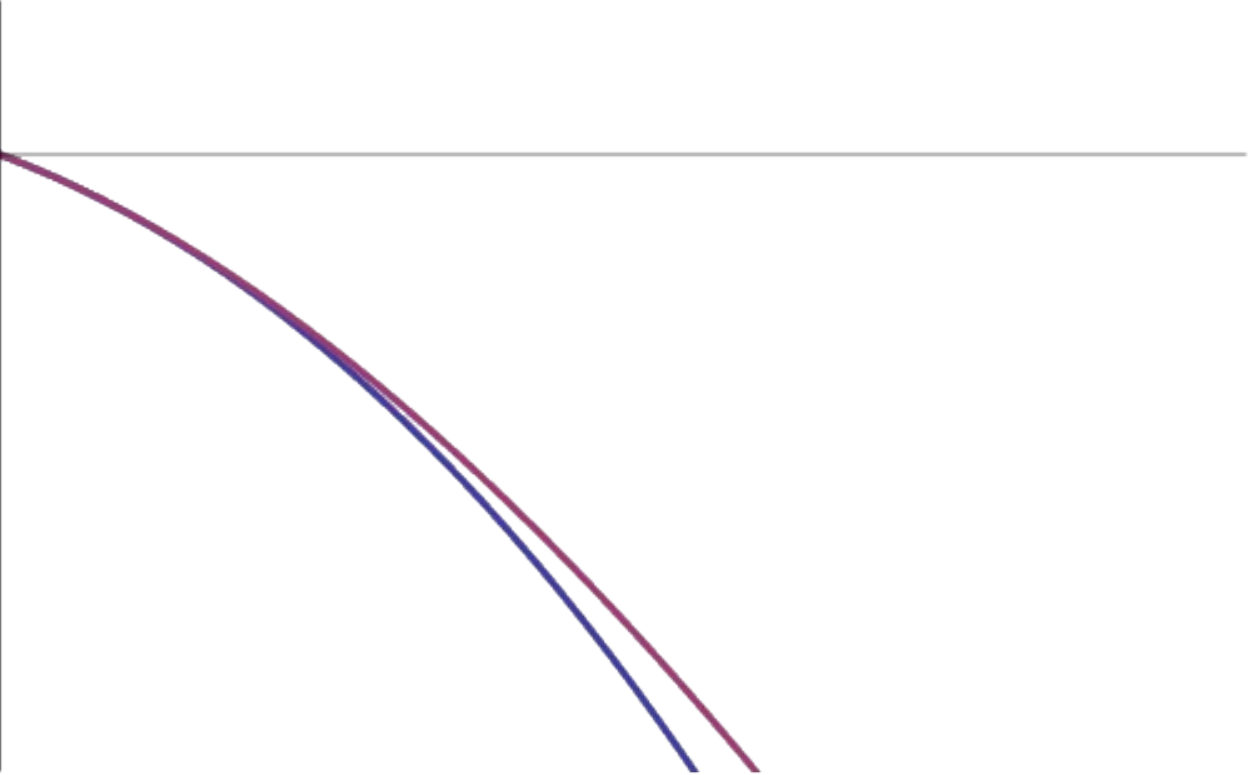}\\
(a)\hskip5cm(b)
\end{center}
\caption{Plots of the functions $f_+$ and $f_-$: (a) for $\nu<0$ 
and (b) for $\nu>0$.
\label{Fig:fsigma}}
\end{figure}
%%%%%%%%%%%%%%%%%%%%%%%%%%%%%%%%%%%%%%%%%%%%%
Since the function $f_\sigma$ is independent of $\delta$, the numbers and positions
of solutions to (\ref{Eq:fd}) can be easily read from the graph. Moreover, 
as the Hessian of $h$ at a critical point is diagonal, the type of
the critical point can be read from the slope of $f_\sigma$. 
A straightforward analysis 
shows that $f_\sigma$ with $\nu<0$ has a single non-degenerate
maximum in the neighbourhood of the origin, and it is monotone for $\nu>0$.
So the equation (\ref{Eq:ficpoint}) has two solutions for $\nu>0$, and from none to
 four solutions for $\nu<0$ (depending on the value of $\delta$).

The results of the above analysis are summarised in  
the bifurcation diagram for critical points of $h$  shown on
Figure~\ref{Fig:diagram}.  We note that qualitatively the bifurcation diagram
is the same for all $n\ge7$. In a neighbourhood of the origin on the $(\delta,\nu)$-plane
there are 4 domains which correspond to different numbers of saddle critical points of $h$.
\begin{figure}
\begin{center}
\includegraphics[width=5cm]{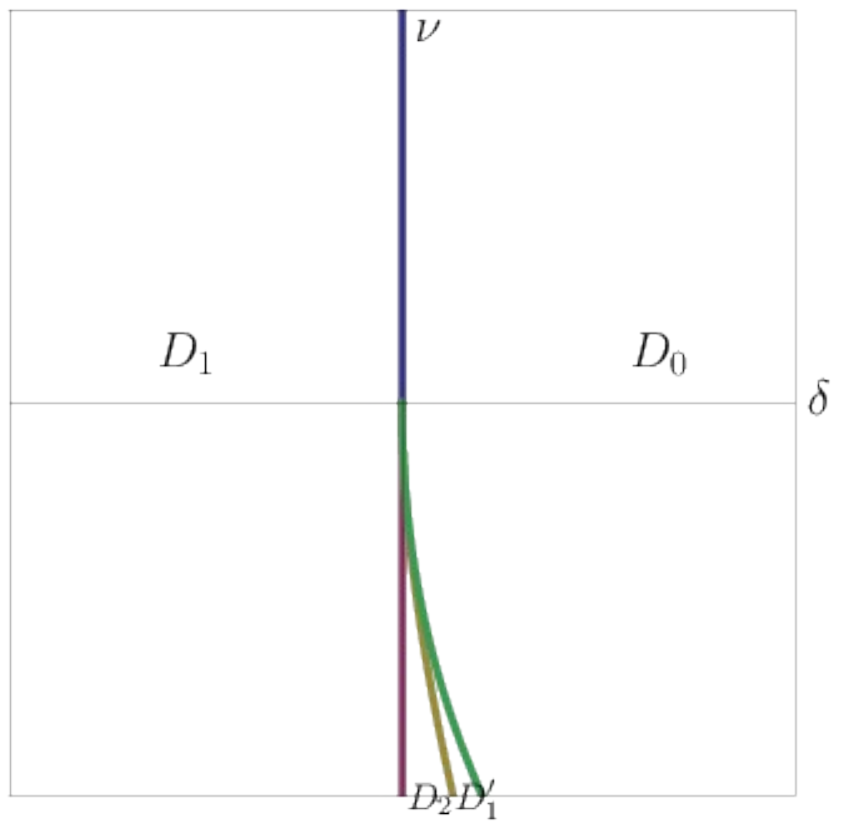}
\end{center}
\caption{Example of a bifurcation diagram on the plane $(\delta,\nu)$ ($n\ge7$).
\label{Fig:diagram}}
\end{figure}
If $(\delta,\nu)\in D_k$, $k=0,1,2$, there are exactly $kn$ saddle critical points of $h$.
The fourth domain is a narrow sector, which separates $D_0$ and $D_2$
and which we denote by $D_1'$.
In this sector the Hamiltonian has $n$ critical saddle points. 

Let us describe the level sets of $h$.
 When the parameters are in $D_0$
all level sets of $h$  are closed curves which look like the invariant curves
shown on  Figure~\ref{Fig:islands7}(a).  When the parameters are in $D_1$,
the critical level set of $h$ form a chain of $n$ islands similar to one shown on
Figure~\ref{Fig:islands7}(b). When the parameters cross the positive 
$\nu$-semiaxes moving from $D_0$ to $D_1$, a chain of islands bifurcates 
from the origin.
When the parameters cross the negative  $\nu$-semiaxes moving from $D_1$ to $D_2$
a second chain of $n$-islands bifurcates from the origin (a typical picture is
shown on Figure~\ref{Fig:islands7e2}(a)). 
In $D_2$ there is a line on which both chains of islands belong to a single level set of $h$
(see Figure~\ref{Fig:islands7e2}(c)). Near this line the level sets of $h$ change their topology
without any change in the number of critical points (see Figure~\ref{Fig:islands7e2}(d)).
We note that in this region the Hamiltonian $h$ has meandering invariant curves similar to ones
studied in \cite{Simo1998,DL2000}. Finally, the chains of islands disappear through
Hamiltonian saddle-node bifurcations when the 
parameters cross the boundary of $D_1'$: first the outer one and then the other one
(see Figure~\ref{Fig:islands7e2}(e) and \ref{Fig:islands7e2}(f)).

\begin{figure}[t]
\begin{center}
\includegraphics[width=4cm]{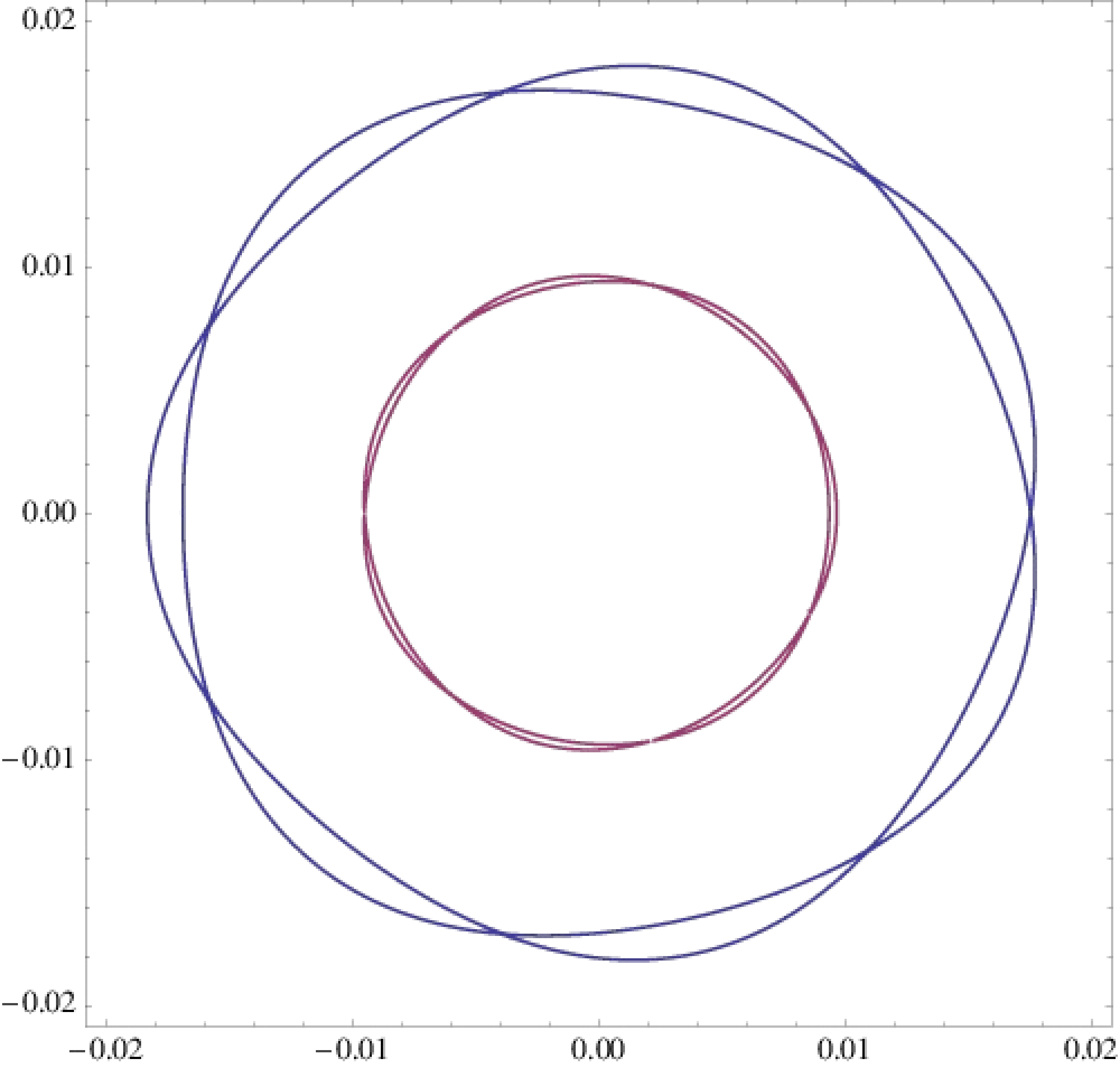}%
\includegraphics[width=4cm]{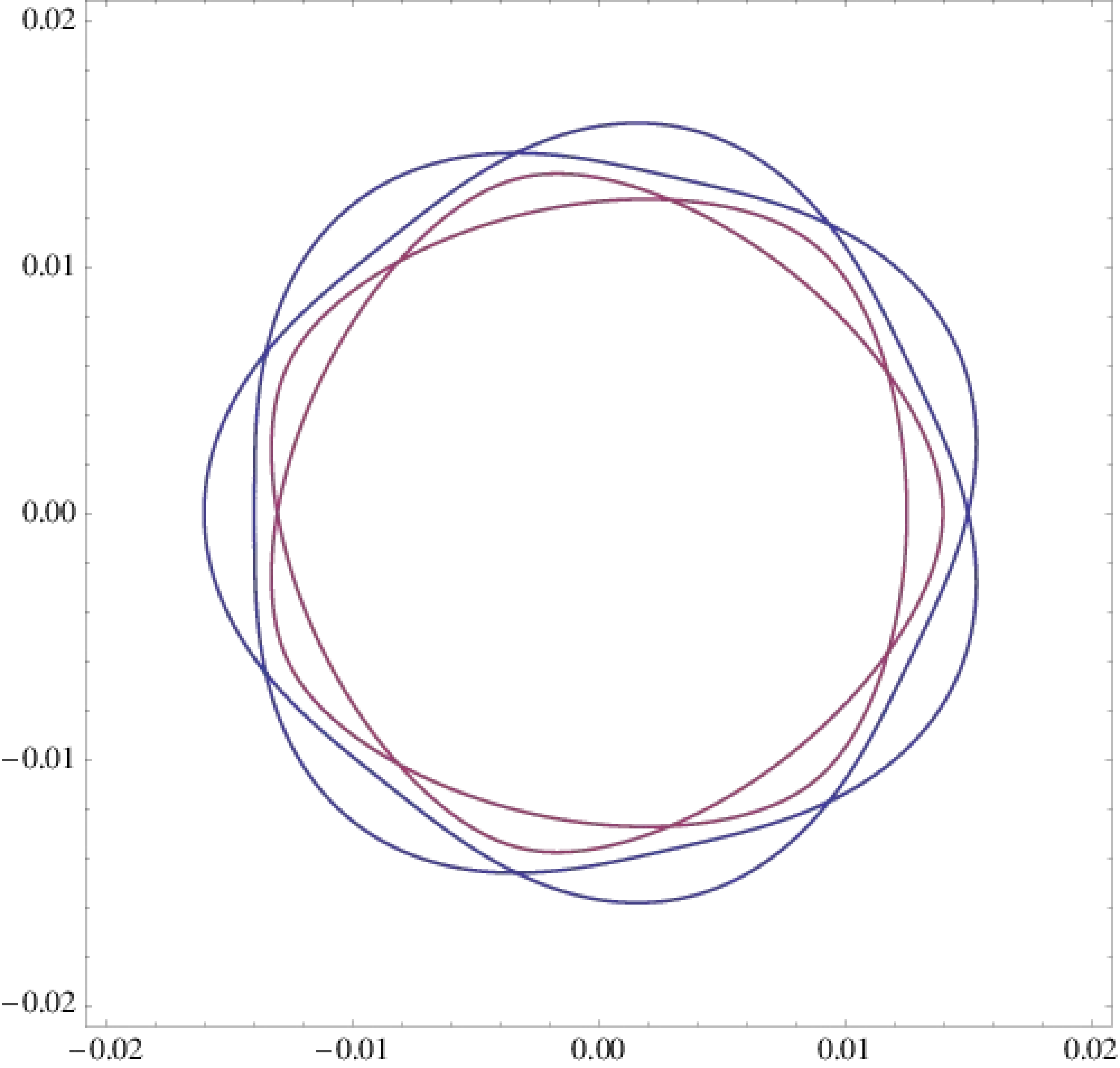}%
\includegraphics[width=4cm]{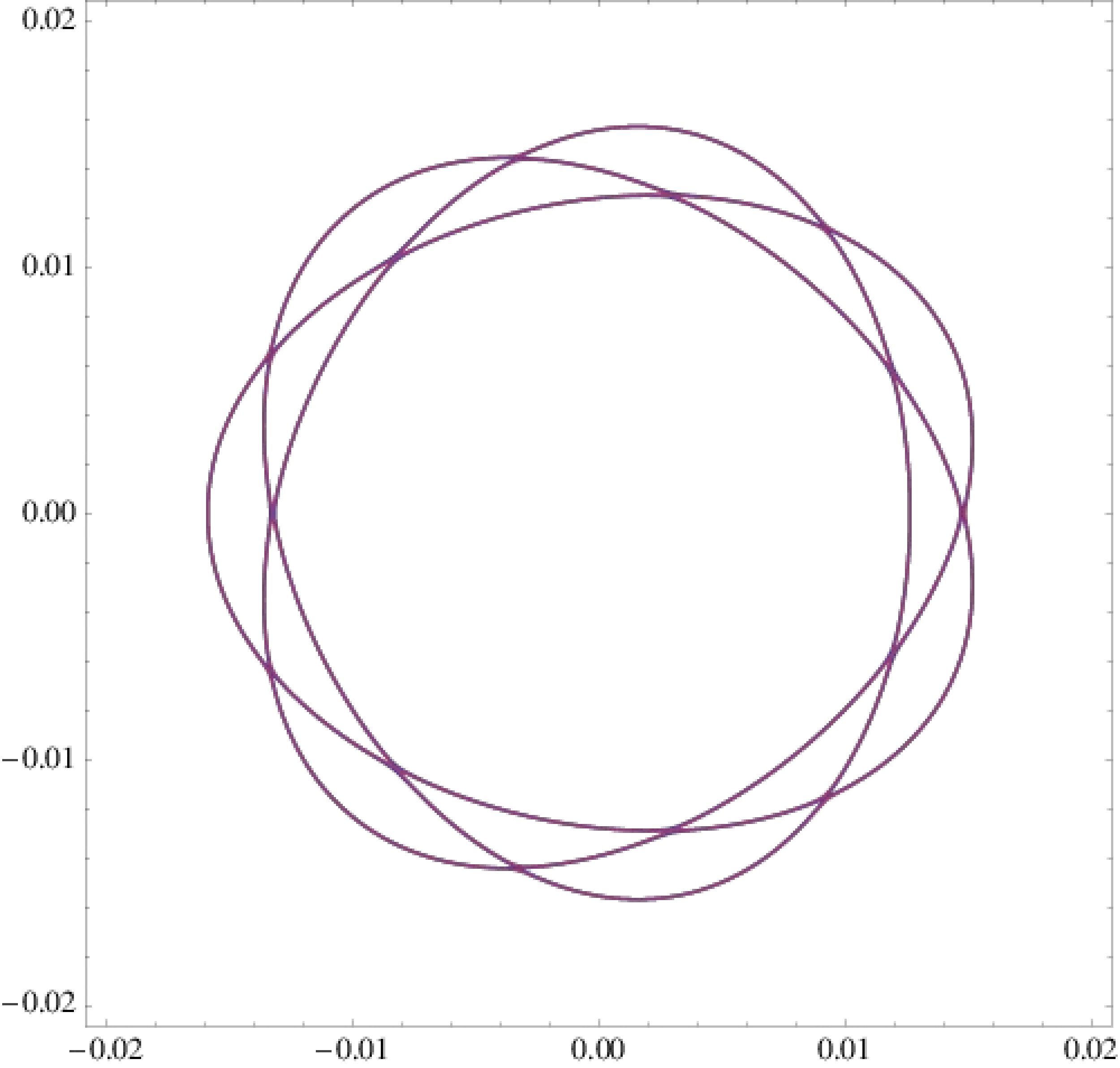}\\
(a)\hskip3.8cm(b)\hskip3.8cm(c)\\[6pt]
\includegraphics[width=4cm]{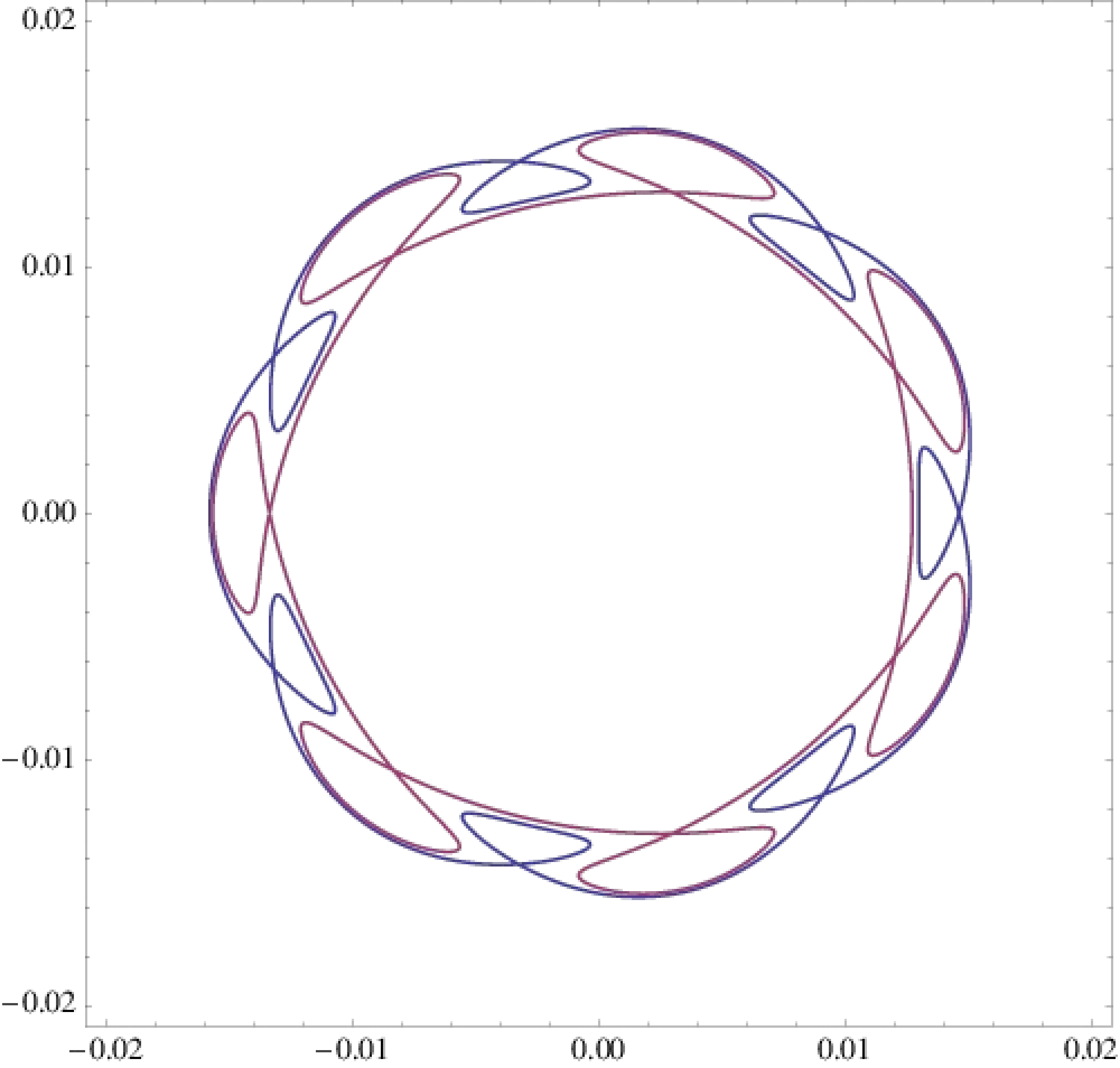}%
\includegraphics[width=4cm]{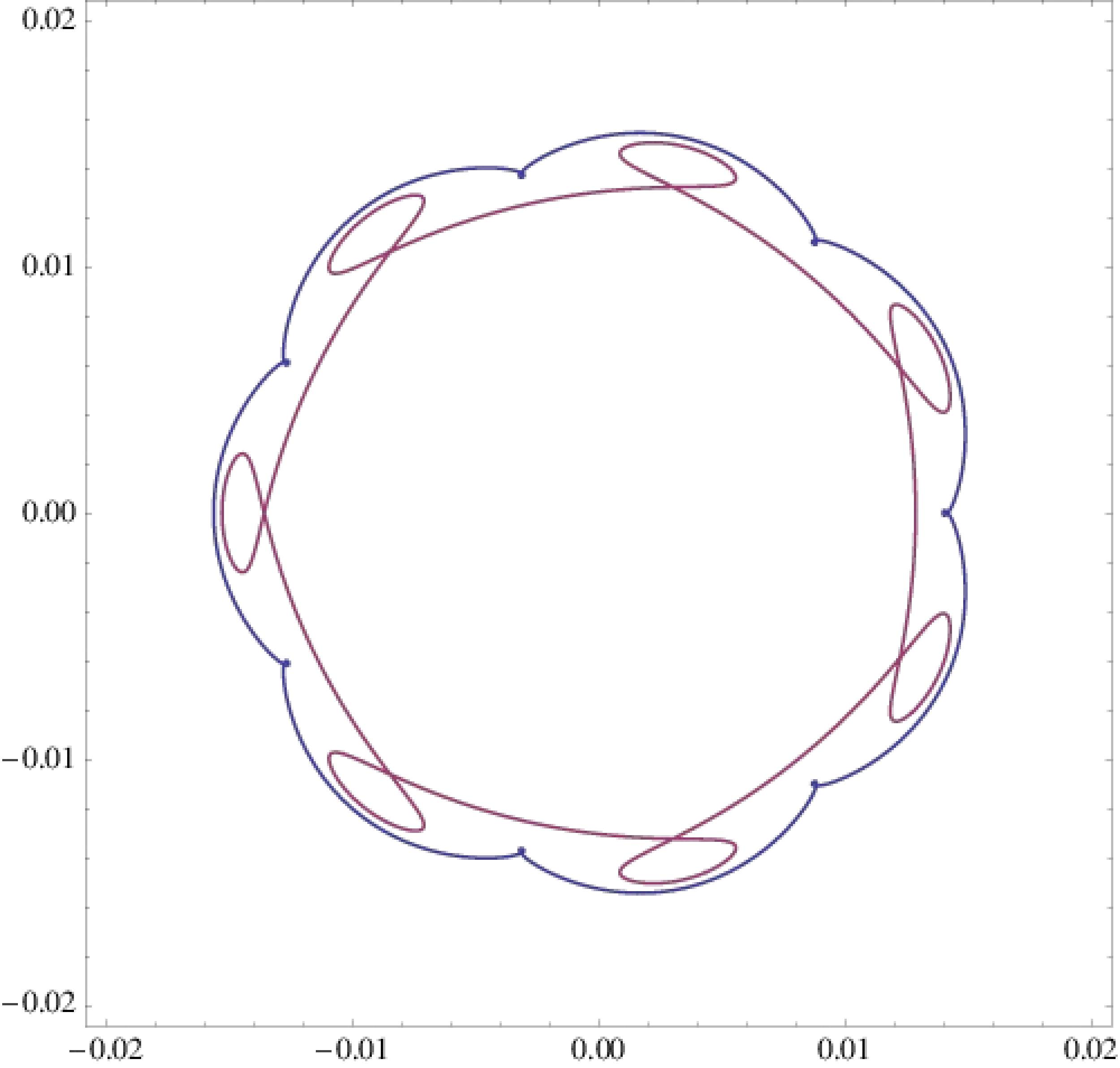}%
\includegraphics[width=4cm]{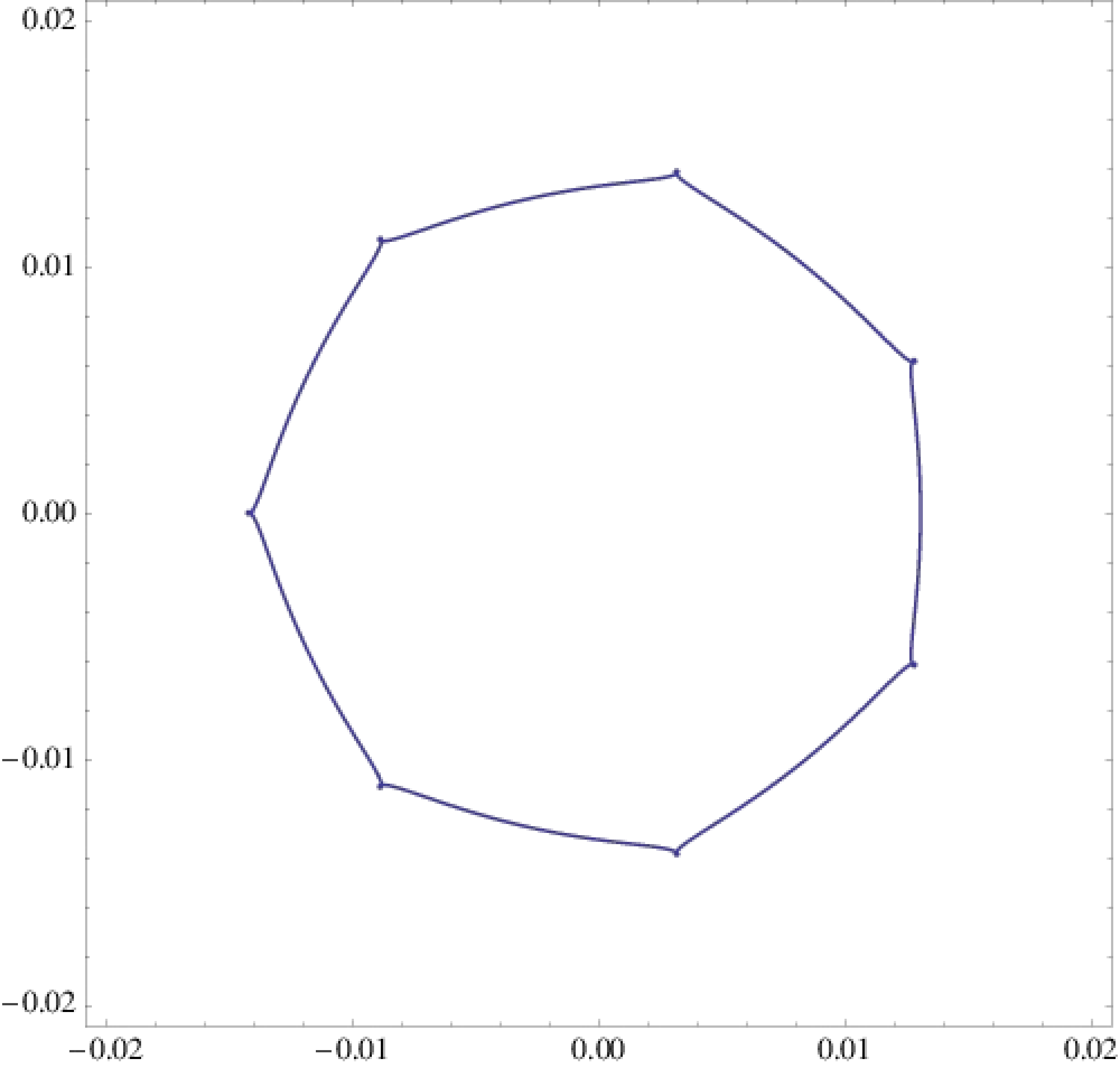}\\
(d)\hskip3.8cm(e)\hskip3.8cm(f)
\end{center}
\caption{Level lines of the normal form Hamiltonian which correspond to
the energy levels of unstable equilibria ($n=7$). 
The pictures correspond to the parameters $(\delta,\nu)$ moving along
a circle around the origin in the anti-clockwise direction and crossing the domains $D_2$
and $D_1'$. 
\label{Fig:islands7e2}}
\end{figure}

In order to derive an approximate expression for the curves bounding the
domain $D_1'$ on the bifurcation diagram we note that the corresponding values of $\delta$ coincide 
with the maximum value of the functions $f_\sigma$ (compare with equation~(\ref{Eq:fd})).
The equation $\frac{\partial f_\sigma}{\partial I}=0$
 takes the form
\begin{equation}\label{Eq:ficpoint1}
2\nu +6I+\sigma  \frac{n}{2}\left(\frac{n}{2}-1\right)I^{ n/2-2}=0\,.
\end{equation}
The equations (\ref{Eq:ficpoint}) and (\ref{Eq:ficpoint1}) 
together define a line on the plane of $(\delta,\nu)$ along which $h$ has a doubled critical point.

In order to solve the equation  (\ref{Eq:ficpoint1}), we rewrite it in the form
$$
I=-\frac{\nu}{3} -\sigma  \frac{n}{12}\left(\frac{n}{2}-1\right)I^{ n/2-2}
$$
and  apply the method of consecutive approximations starting with the initial approximation
$I_0=-\frac{\nu}{3}$.
A standard estimate from the contraction mapping theorem
implies
$$
I=-\frac{\nu}{3}-\sigma  \frac{n}{12}\left(\frac{n}{2}-1\right)\left(-\frac{\nu}{3}\right)^{ n/2-2}+O(\nu^{ n-5}).
$$
Substituting this approximation into (\ref{Eq:fd}) we get
$$
\delta=
\frac{\nu^2}{3}-\sigma\frac{n}2\left(-\frac\nu3\right)^{n/2-1}+O(\nu^{n-4})\,.
$$
This equation defines two lines on the plane $(\delta,\nu)$ (one for each value of $\sigma$).
Both lines enter the zero quadratically and differ by  $O(\nu^{n/2-1})$.
In this way we have derived an approximate expressions for the lines
which bound the domain $D_1'$ on the
bifurcation diagram of  Fig.~\ref{Fig:diagram}.
We note that $\sigma=1$ leads to a smaller $\delta$ for the same $\nu$
compared to $\sigma=-1$. Consequently, $\sigma=1$ corresponds to the boundary
between $D_2$ and $D_1'$, and $\sigma=-1$ corresponds to the
boundary between $D_1'$ and $D_0$.

%%%%%%%%%%%%%%%%%%%%%%%%%%%%%%%%%%%%%%%%%
\subsubsection*{Shape of the islands}
Let 
\begin{equation}\label{Eq:scaling7_0}
\nu=\mu-3\epsilon,\qquad
\delta=3\epsilon^2-2\mu\epsilon,
\qquad
I=\epsilon+ \epsilon^{n/4-1/2}J,\qquad
h=\epsilon^{-n/2}\bar h\,.
\end{equation}
In the new variables the Hamiltonian  (\ref{Eq;hn})  takes the form (after skipping a constant term)
$$
\bar h=J^2+\epsilon^{n/4}\mu^{-3/2}J^3+
(1+\epsilon^{n/4-1}\mu^{-1/2}J)^{n/2}\cos n\varphi
.
$$
Assuming $0<\epsilon< \mu$ we get
$$
\bar h=J^2+\cos n\varphi+O(\epsilon^{(n-6)/4})
.
$$
Solving equations (\ref{Eq:scaling7_0}) with respect to $\epsilon$ we
find that    $\epsilon=\frac{1}{3} \left(\sqrt{\nu ^2-3 \delta
   }-\nu \right)$. This expression gives a small value  when $\delta$ and $\nu$ are small.
 In terms of the original parameters the region $0<\epsilon< \mu$
has the form $\delta<\frac{\nu^2}4$ for $\nu\le0$ and $\delta<0$ for $\nu>0$.
In this region  the chain of islands is approximated by
a chain of the pendulum's separatrices. Reversing the scaling (\ref{Eq:scaling7_0})
we  obtain the radius and the width for the chain of islands.

  Note that equation (\ref{Eq:scaling7_0}) gives a second expression
  for $\epsilon$, namely
   $\epsilon=\frac{1}{3} \left(-\sqrt{\nu ^2-3 \delta
   }-\nu \right)$, which  corresponds to the second chain of islands
    for $\nu<0$ and $0<\delta<\frac{\nu^2}3$.

\subsubsection*{A model for the changes in critical level sets of $h$ in $D_2$.}
In order to study the changes in the level sets of $h$ near the boundary between
$D_2$ and $D_1'$  
we rescale the variables:
\begin{equation}\label{Eq:scaling7}
\nu=-3\epsilon,\quad
\delta=3\epsilon^2+\epsilon^{n/3}a,
\quad
I=\epsilon+\epsilon^{n/6} J,
\quad
h=\epsilon^{n/2}\bar h,
\quad
\psi=n\varphi\,.
\end{equation}
In the new variables the Hamiltonian takes the form (after skipping a constant term)
\begin{equation}
\bar h= a J+J^3+(1+\epsilon^{n/6-1} J)^{n/2}\cos \psi\,.
\end{equation}
Introducing $\epsilon_1=\frac n2\epsilon^{n/6-1}$ we can rewrite the
Hamiltonian in the form
$$
\bar h=a J+J^3+\cos\psi+\epsilon_1 J\cos\psi+O(\epsilon_1^2)\,.
$$
Naturally, we compare the level sets of this Hamiltonian with
\begin{equation}\label{Eq:h61}
\bar h_0=a J+J^3+\cos\psi+\epsilon_1 J\cos\psi
\end{equation}
which does not have any explicit dependence on $n$.
In this equation $\epsilon_1$ is a small parameter while the parameter $a$
is not necessarily small. Critical level sets of $\bar h_0$
for various values of parameter $a$ are shown on Figure~\ref{Fig:barhn}.
For aesthetic reasons, the diagrams shows two complete periods
of the Hamiltonian $\bar h_0$. The original Hamiltonian (\ref{Eq;hn})
posses a chain of $n$ islands.

\begin{figure}
\begin{center}
\includegraphics[width=3.3cm]{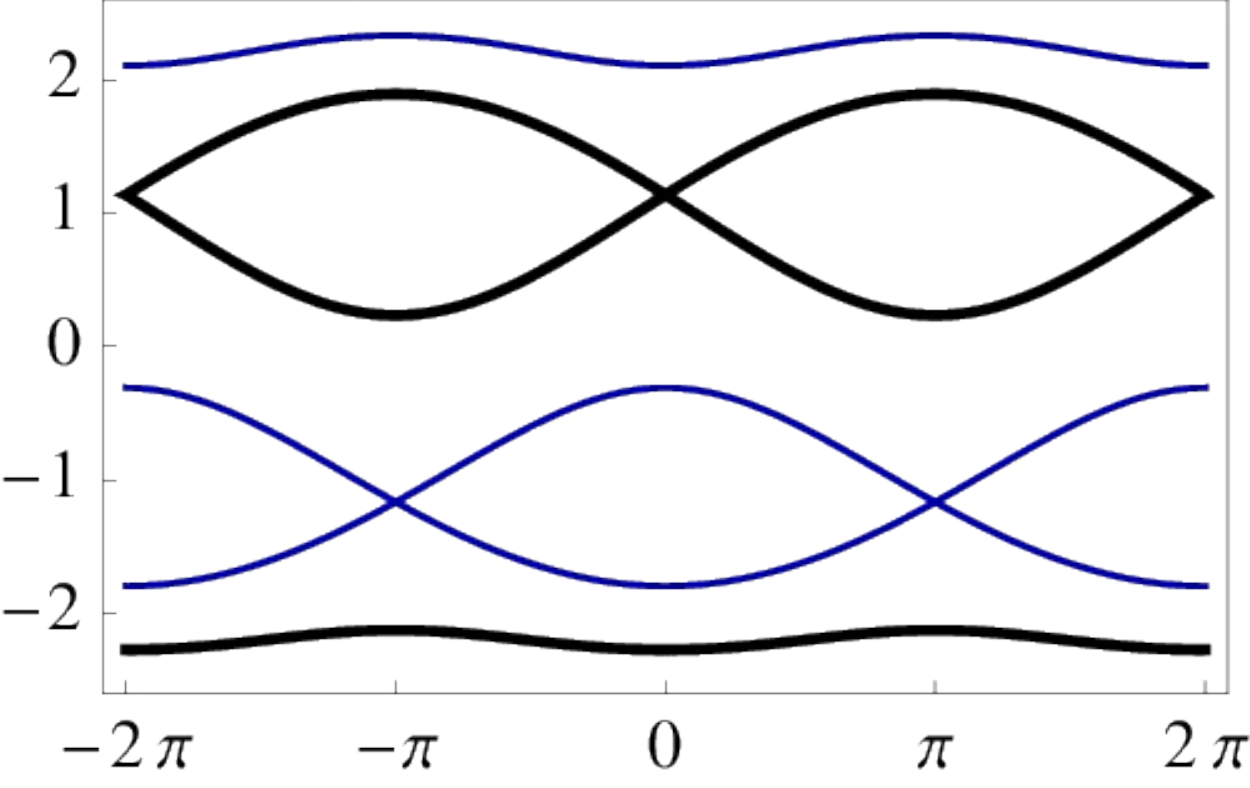}
\hskip0.1cm
\includegraphics[width=3.3cm]{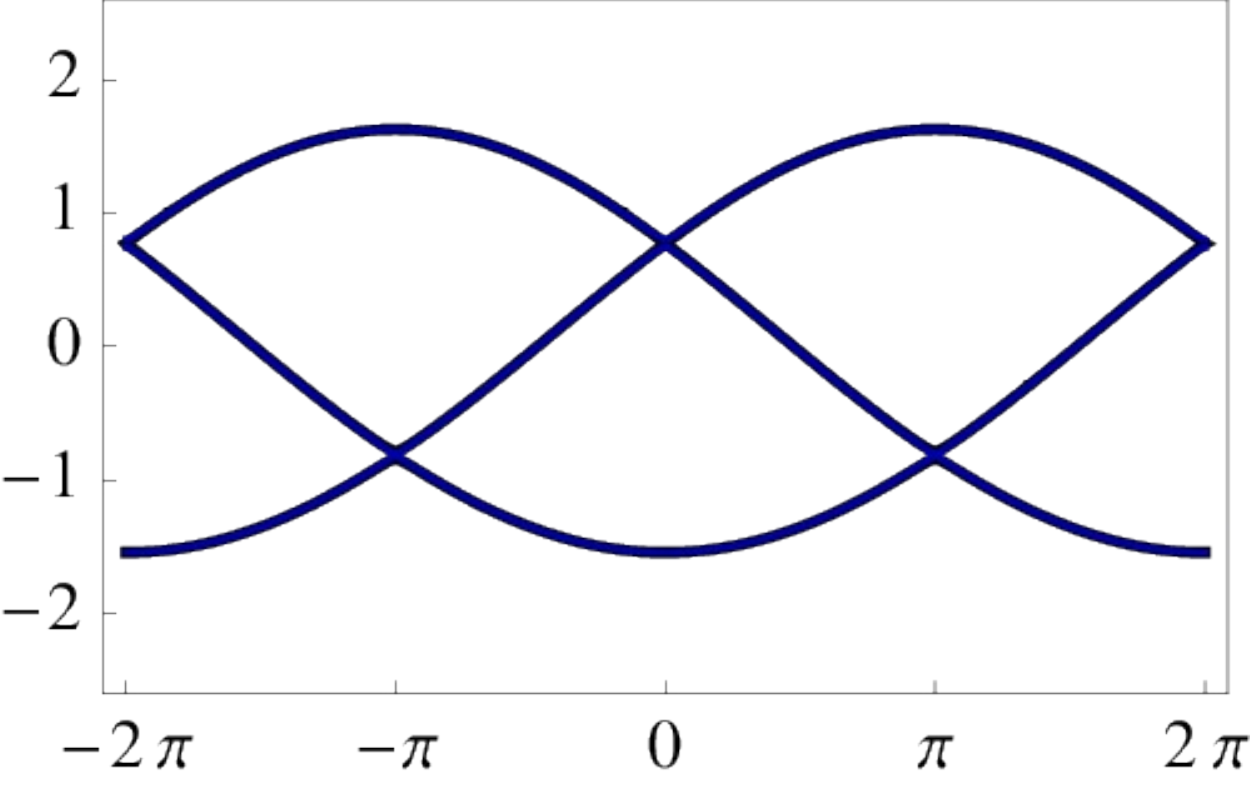}
\hskip0.1cm
\includegraphics[width=3.3cm]{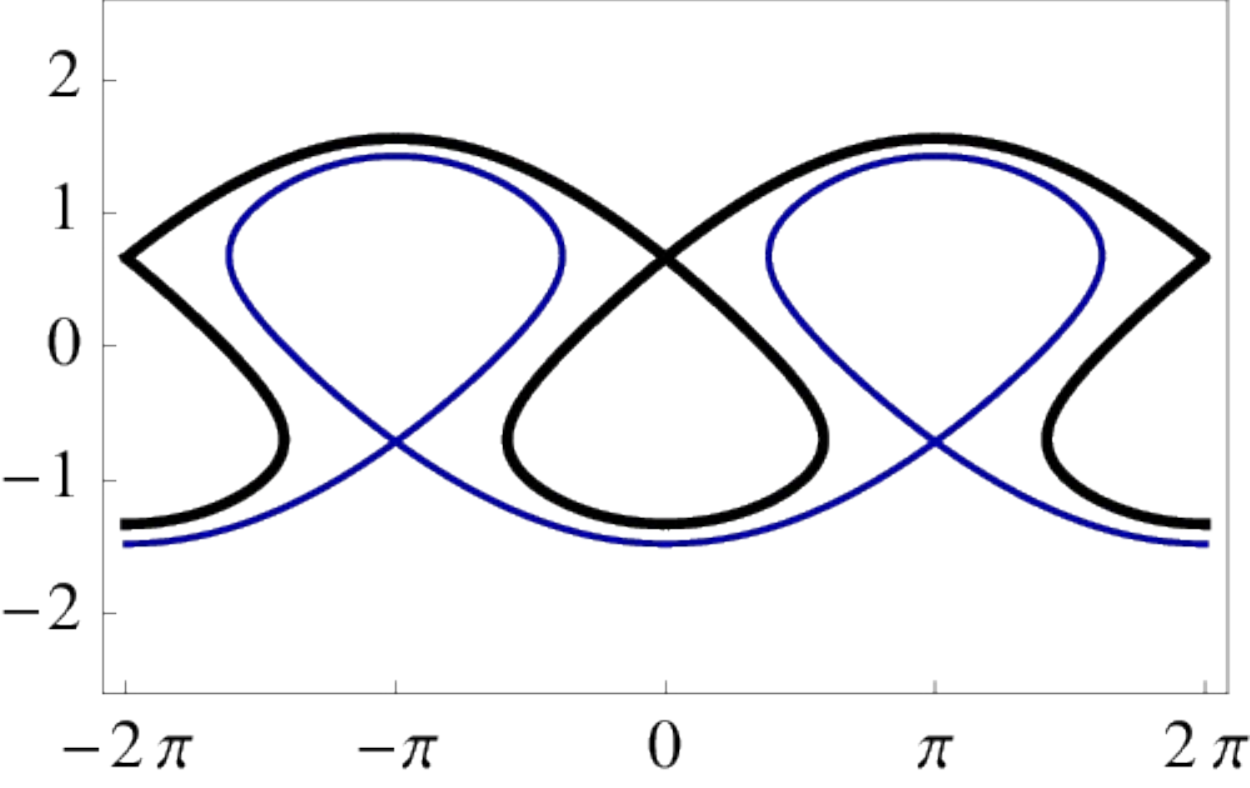}
\\
(a)\hskip3cm(b)\hskip3cm(c)
\\[12pt]
\includegraphics[width=3.3cm]{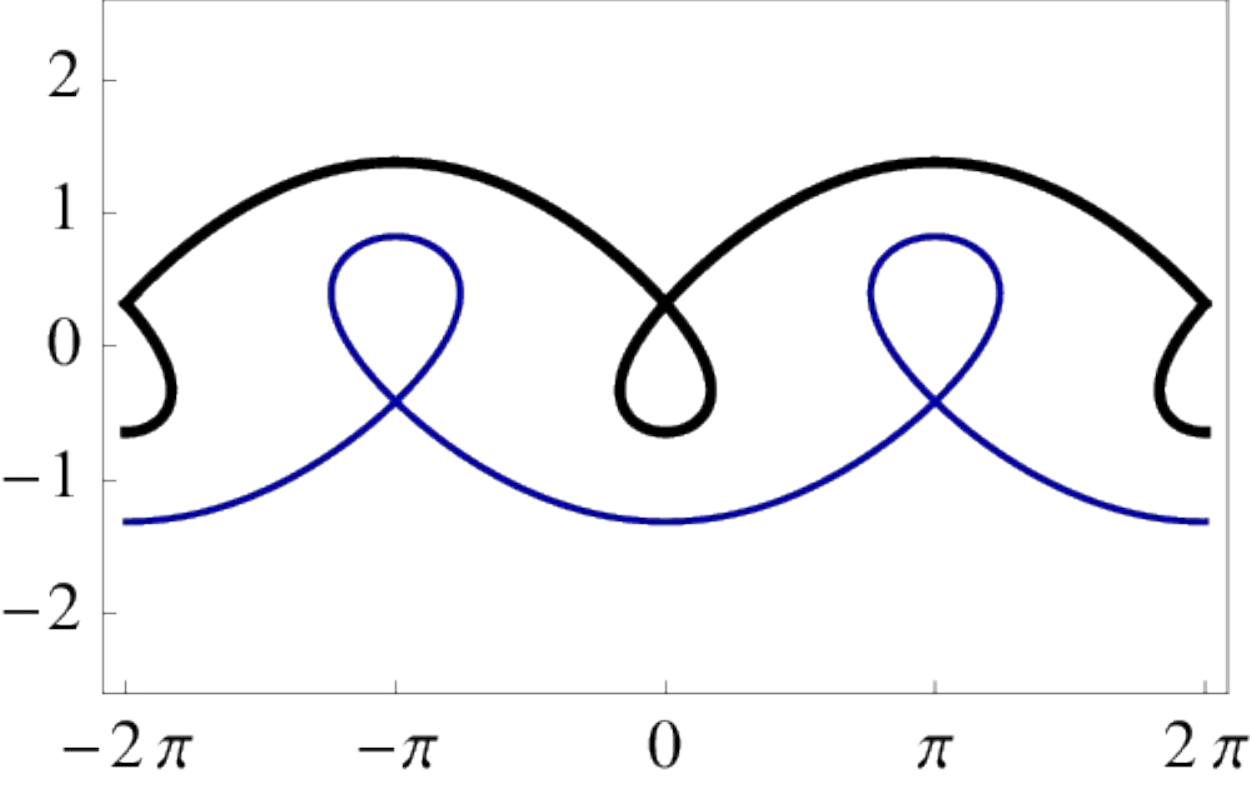}
\hskip0.1cm
\includegraphics[width=3.3cm]{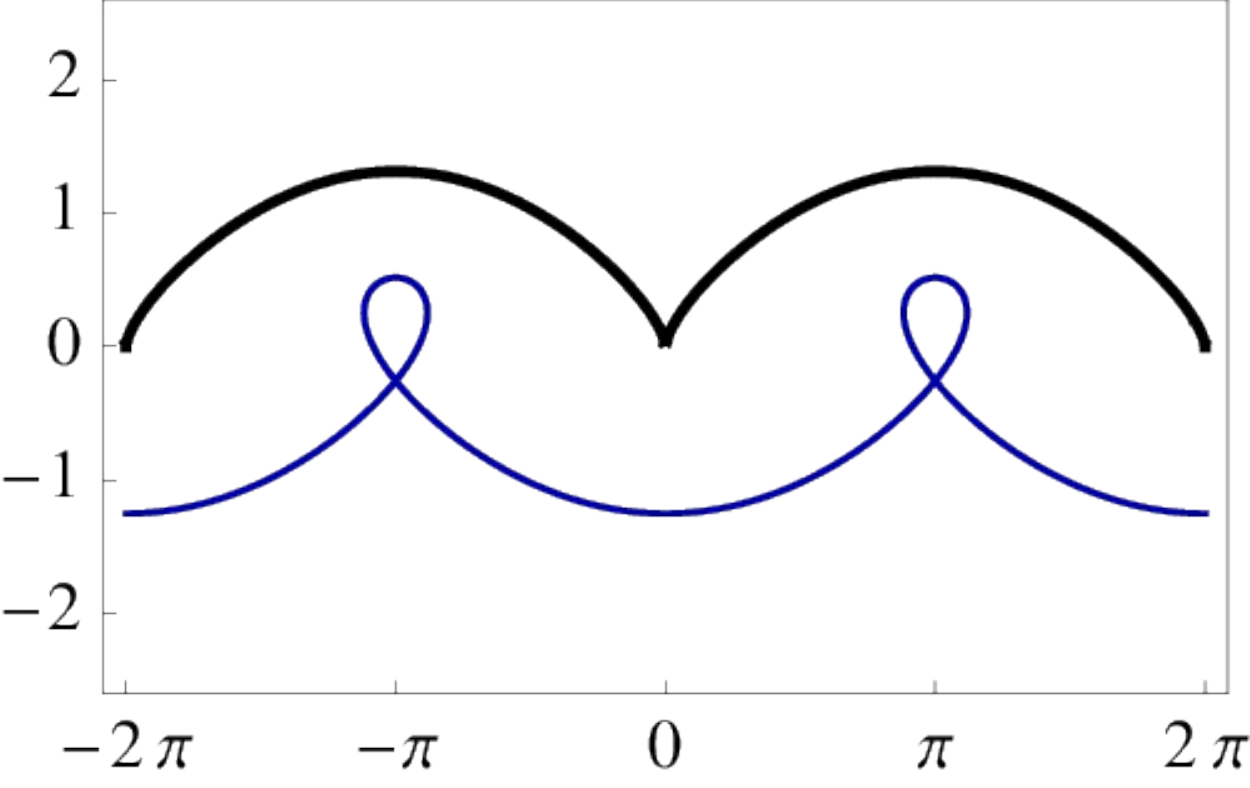}
\hskip0.1cm
\includegraphics[width=3.3cm]{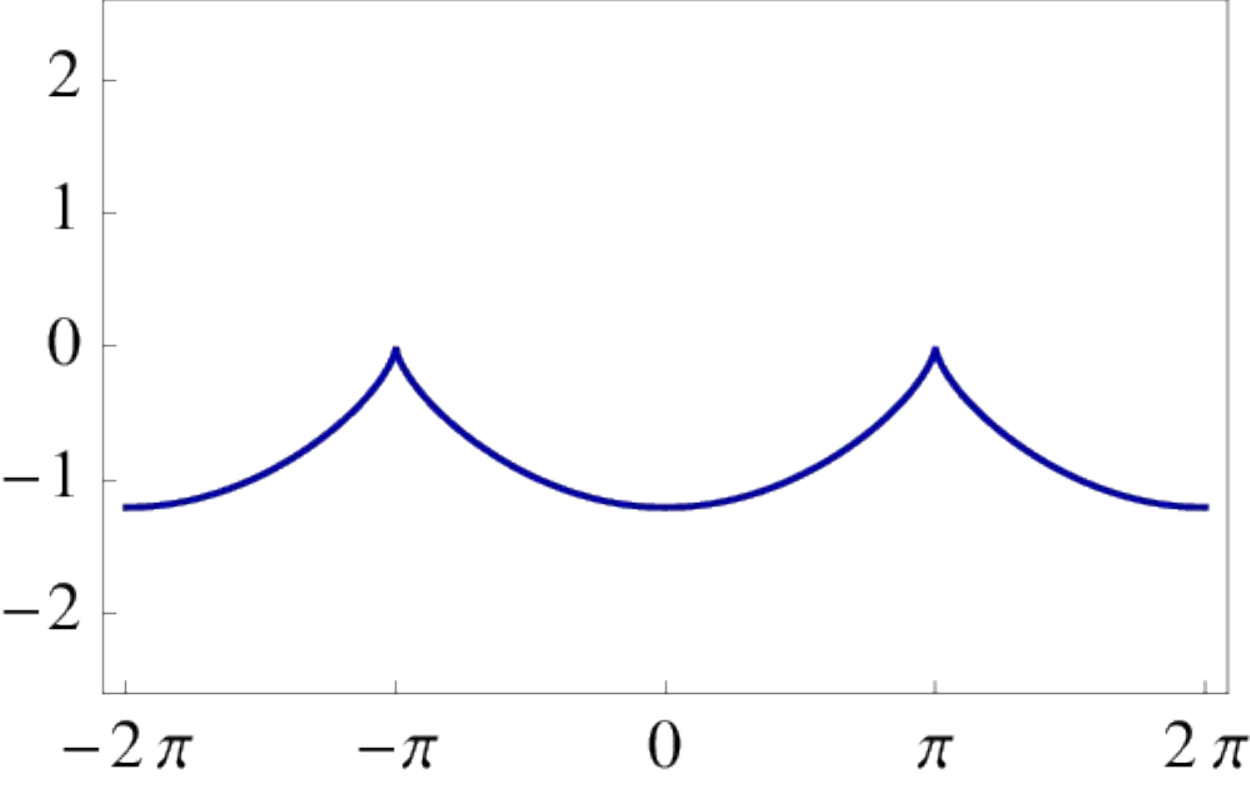}
\\
(d)\hskip3cm(e)\hskip3cm(f)
\end{center}
\caption{Critical level sets of the Hamiltonian (\ref{Eq:h61}). The parameter $a$ increases from left to right
and from top to bottom:
(a) $a=-5$; (b) $a=a_c$;
(e) $a=-\epsilon_1$; (f) $a=\epsilon_1$. For all pictures $\epsilon_1=\frac1{10}$.
\label{Fig:barhn}}
\end{figure}

Obviously, the critical points of $\bar h_0$ 
are located at  $\cos\psi=\sigma_\psi=\pm1$ and
$$
J=\pm\sqrt{-(a+\epsilon_1\sigma_\psi)/3}\,.
$$
This explicit expression for $J$ shows that
$\bar h_0$ has exactly two saddle critical points per period
for $a<-\epsilon_1$, one critical point for $-\epsilon_1<a<\epsilon_1$ and
no critical points for $a>\epsilon_1$. The critical level sets for $a=\pm\epsilon_1$
are shown on Figures~\ref{Fig:barhn}(e) and~(f) respectively,
which correspond to a Hamiltonian saddle-node bifurcation.

We note that at $a=-a_c$, $a_c=3\cdot 2^{-2/3}+O(\epsilon_1^2)$, the two critical saddle points belong
to a single critical level set as shown on Figure~\ref{Fig:barhn}(b).

\subsection{$n=6$}

In the case $n=6$ the model Hamiltonian takes the form
\begin{equation}\label{Eq:model6}
h=\delta I+\nu I^2+I^3+b_0I^{3}\cos n\varphi\,.
\end{equation}
For $\delta=\nu=0$, the origin is stable if $|b_0|<1$ and unstable if $|b_0|>1$.
We assume that $b_0\not\in\{\,{-}1,0,1\,\}$ as these three cases correspond to
a degeneracy of a higher co-dimension.

The critical points of $h$ have either $\varphi=0$ or $\frac{\pi}n\pmod{\frac{2\pi}{n}}$,
and
\begin{equation}\label{Eq:ficpoint6}
\delta +2\nu I+3I^2(1+\sigma_\varphi b_0)=0
\end{equation}
where $\sigma_\varphi=\cos\varphi=\pm1$.
The equation (\ref{Eq:ficpoint6}) can be easily solved explicitly but 
it is more convenient to write it in the form 
\begin{equation}\label{Eq:fd6}
\delta=f_\sigma(I,\nu)
\end{equation}
where
$$
f_\sigma=-2\nu I-3(1+\sigma b_0)I^2\,.
$$
Since the function $f_\sigma$ is independent of $\delta$, the number and positions
of solutions to the equation (\ref{Eq:fd6}) for  given $\delta$ and $\nu$
can be easily read from the graphs of the functions $f_+$ and~$f_-$. 
The analysis leads to different conclusions for the stable and unstable cases.
Moreover, the qualitative behaviour of the graphs depend on the sign of  $\nu$.
In the rest, this analysis is completely straightforward  and its
results are summarised on the  bifurcation diagrams
of Figure~\ref{Fig:n6diagram}(a) and~(b) for the stable and unstable case
respectively.

In the stable case the bifurcation diagram is similar to the case of a  resonance of order
$n\ge7$. Although later in this section we will observe some quantitative differences.

In the unstable case the bifurcation diagram consists of four domains:
$D_1$, $D_1'$, $D_2$ and $D_2'$, which correspond to 
the existence of $n$ and $2n$ saddle critical points of $h$ as indicated by
the subscript in the domain's name.

\begin{figure}
\begin{center}
\includegraphics[width=4.5cm]{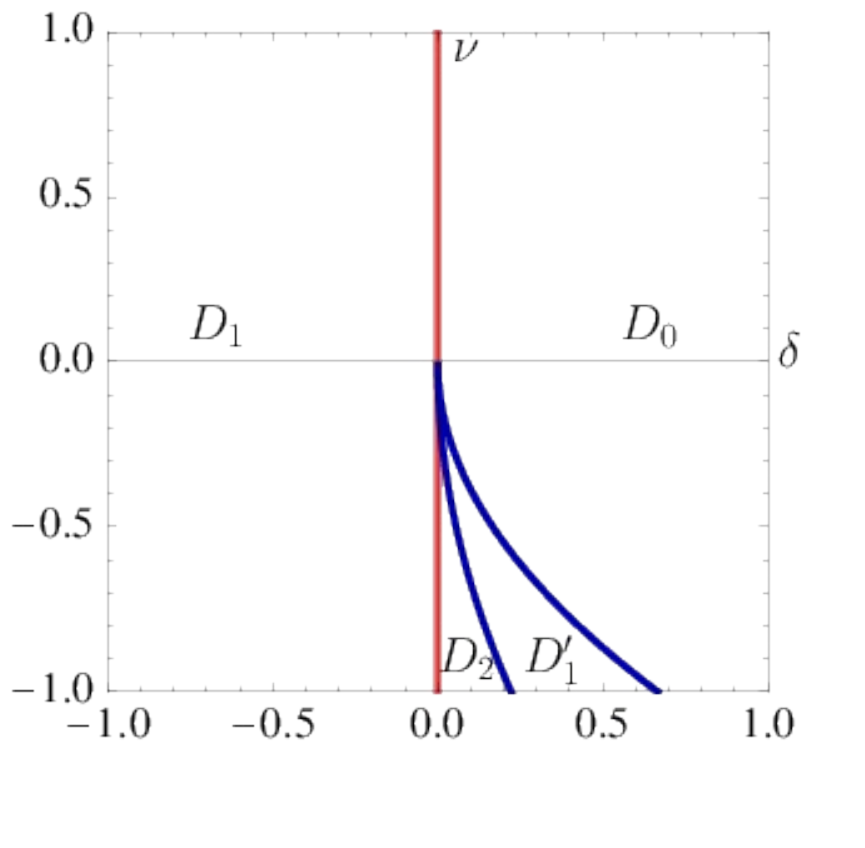}\hskip1cm
\includegraphics[width=4.5cm]{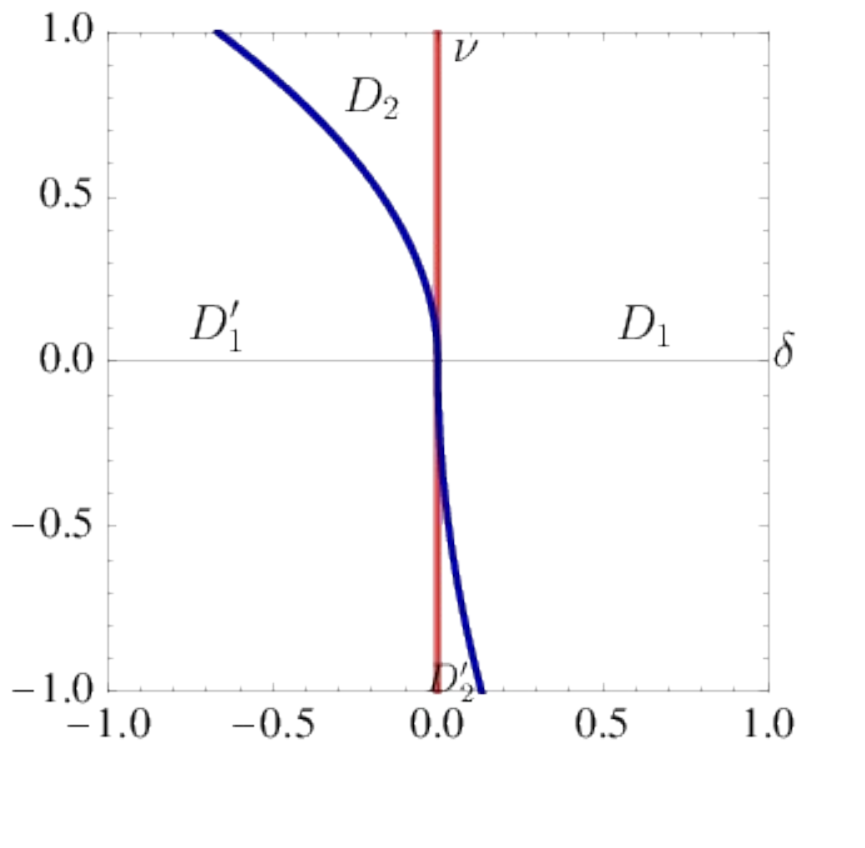}\\
(a)\hskip4.5cm(b)
\end{center}
\caption{Bifurcation diagrams for $n=6$: the stable case (a) and the unstable one (b)}
\label{Fig:n6diagram}
\end{figure}

In order to derive an equation for the boundaries  between the domains
on the bifurcation diagram, we note that the corresponding value of $\delta$ coincide
with a critical value of  $f_\sigma$. The equation $\frac{\partial f_\sigma}{\partial I}=0$
has a unique solution
$$
I=-\frac{\nu}{3(1+\sigma  b_0)}\,.
$$
The corresponding critical value is
$$
\delta=\frac{\nu^2}{3(1+\sigma b_0)}\,.
$$
This equation defines two lines  corresponding to double critical points of~$h$
(for $\sigma=\pm1$ respectively). Note that 
the bifurcation diagrams of Figure~\ref{Fig:n6diagram} include only the parts of the parabolas
which correspond to positive critical points.

\begin{figure}
\begin{center}
\includegraphics[width=3.3cm]{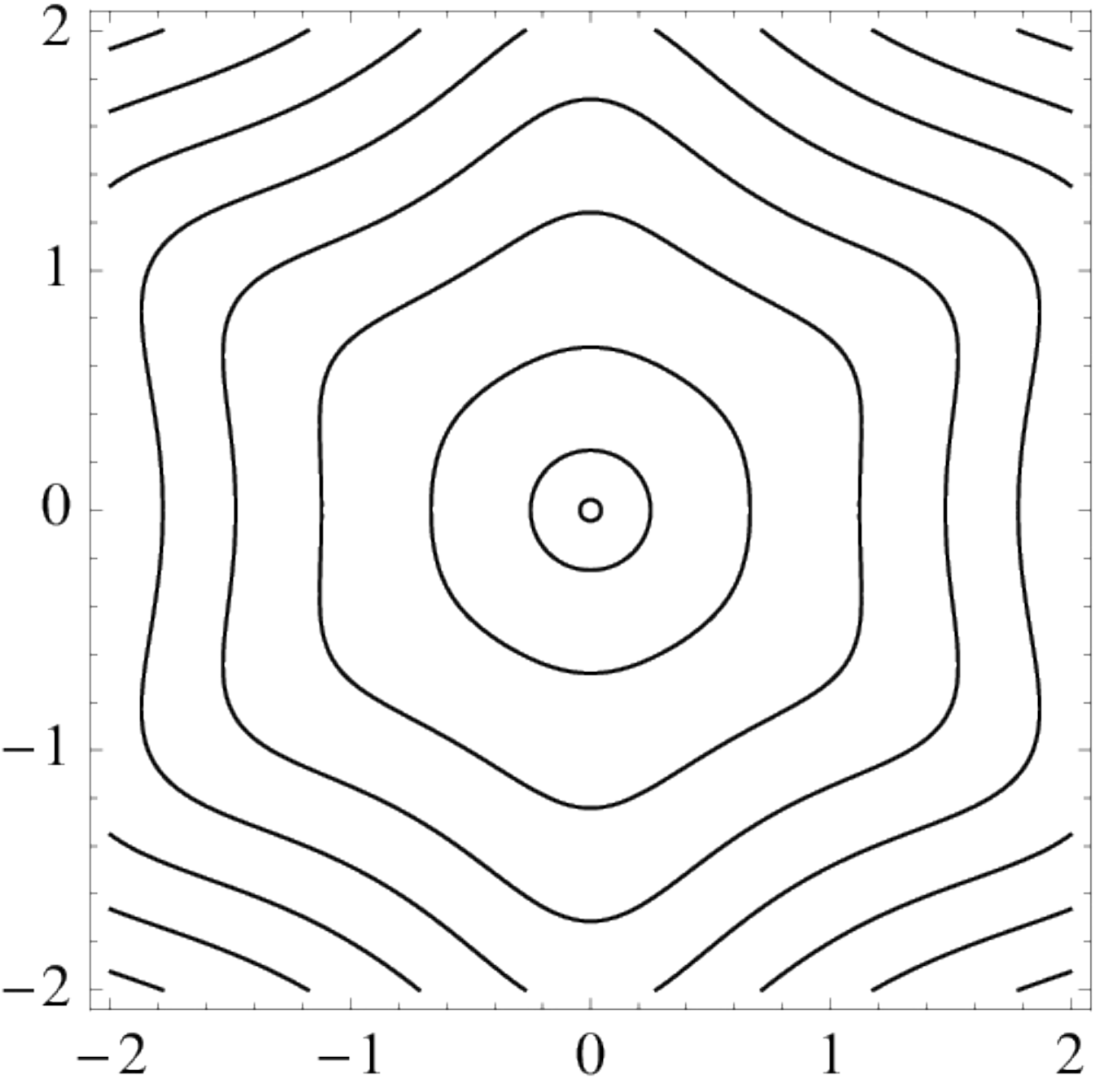}
\hskip1cm
\includegraphics[width=3.3cm]{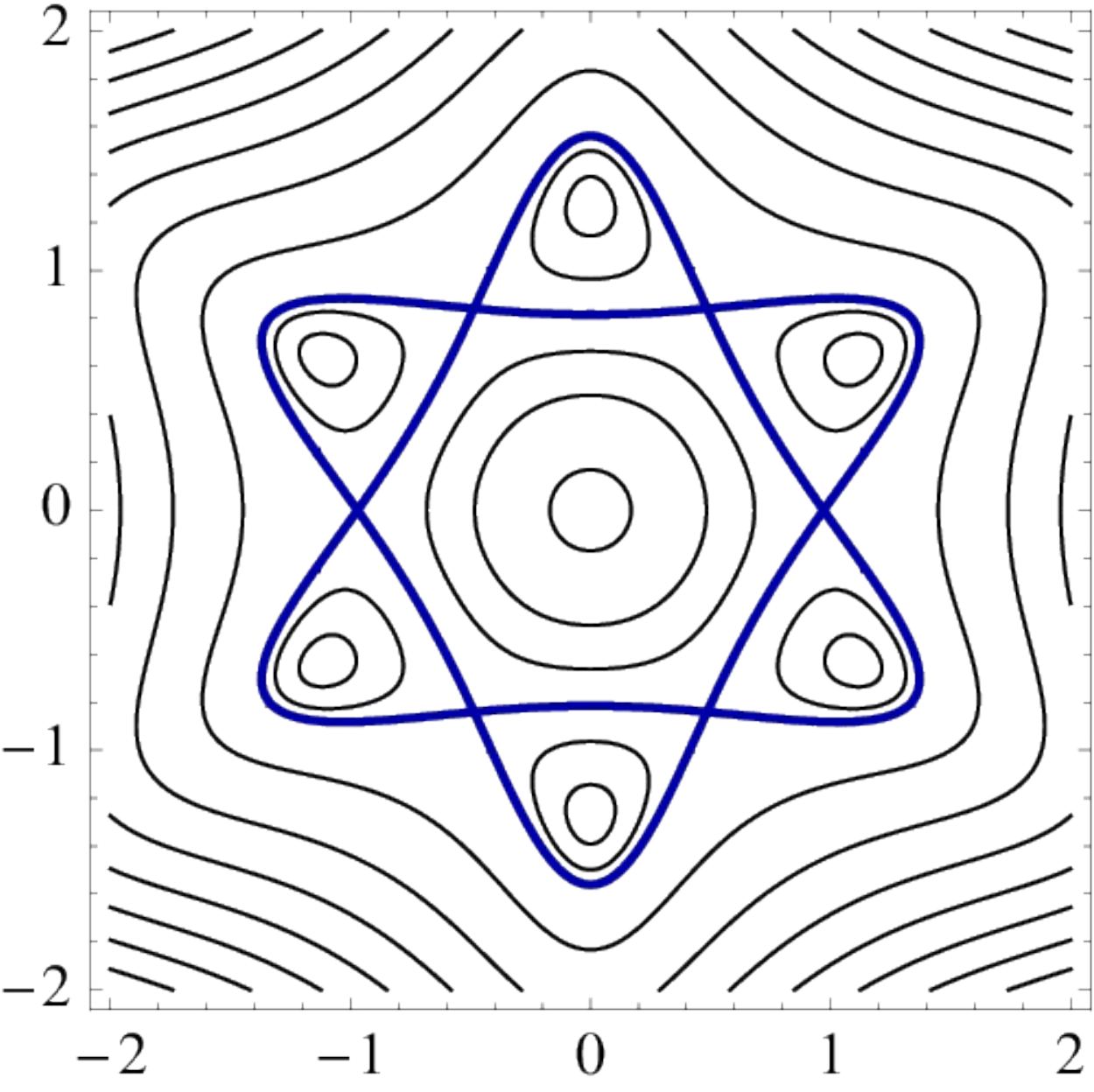}
\hskip1cm
\includegraphics[width=3.3cm]{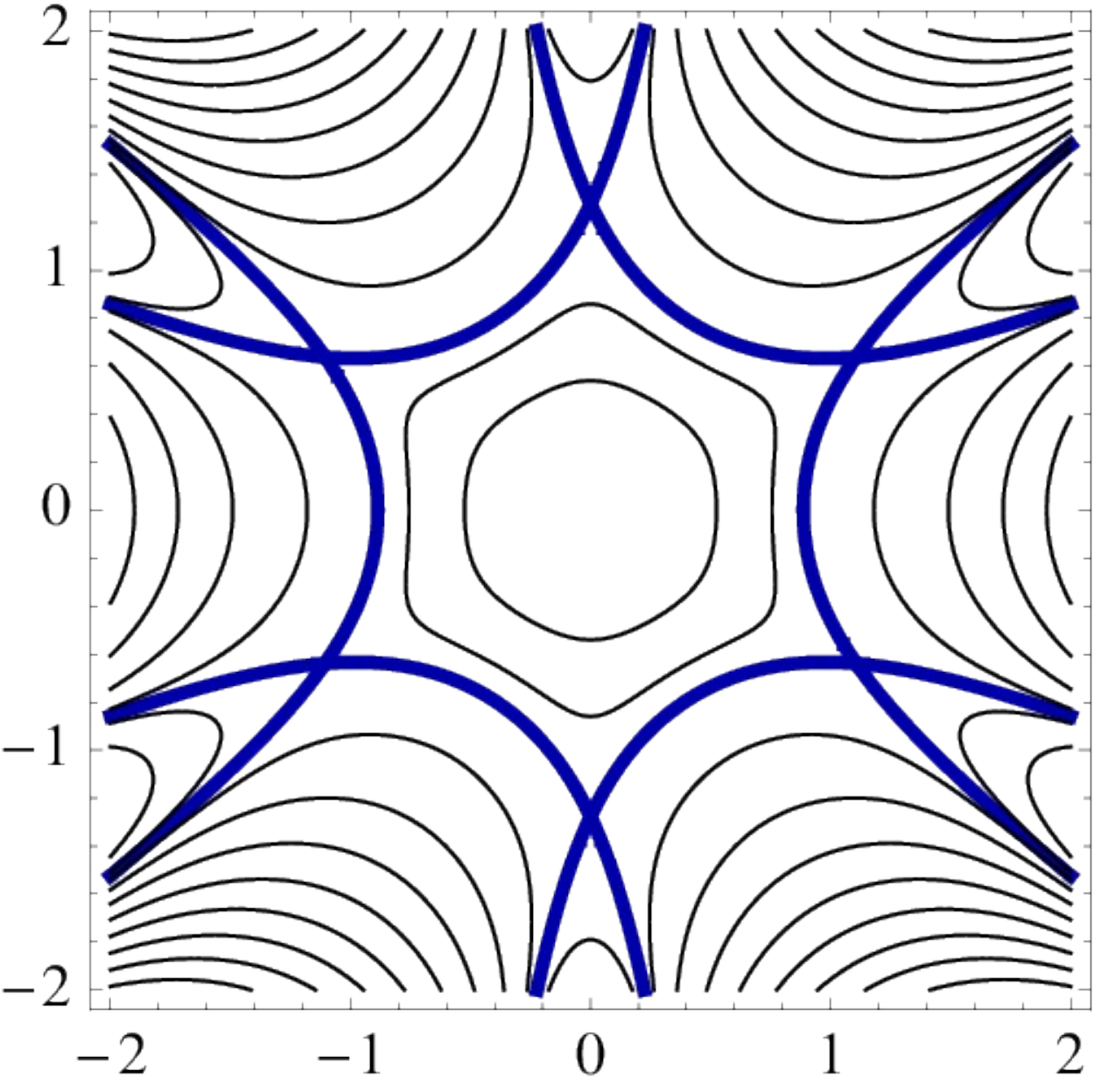}
\\
(a)\hskip4cm(b)\hskip4cm(c)
\end{center}
\caption{Level sets for the limit Hamiltonian (\ref{Eq:barh60}):
(a) $\mathrm{sign}(\delta)=1$ and
$b_0=\frac12$,
(b)
$\mathrm{sign}(\delta)=-1$ and  
$b_0=\frac12$,
(c) 
$\mathrm{sign}(\delta)=1$
and $b_0=\frac32$.
\label{Fig:n6gen} }
\end{figure}

\begin{figure}[t]
\begin{center}
\includegraphics[width=3.2cm]{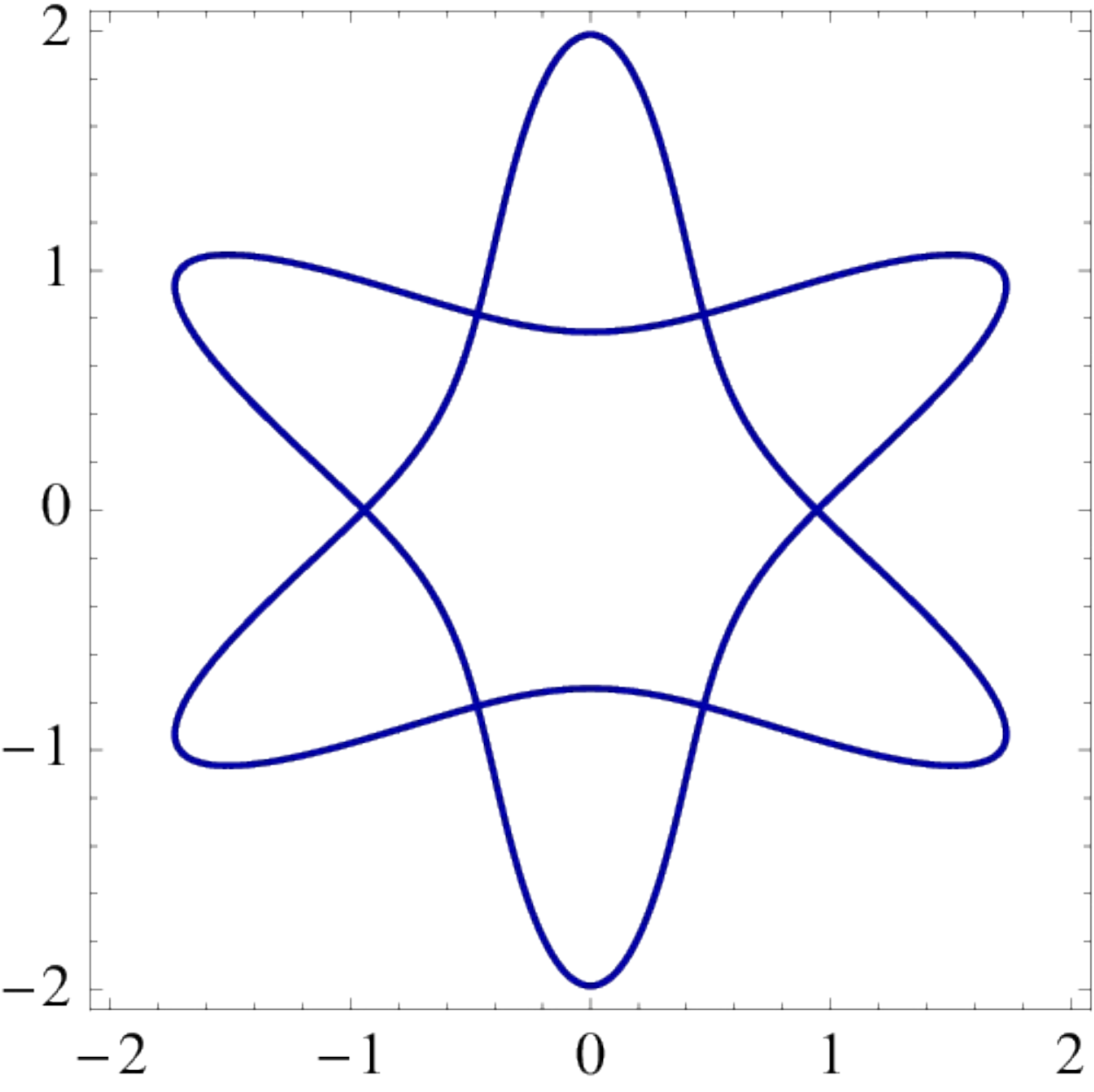}
\hskip0.1cm
\includegraphics[width=3.2cm]{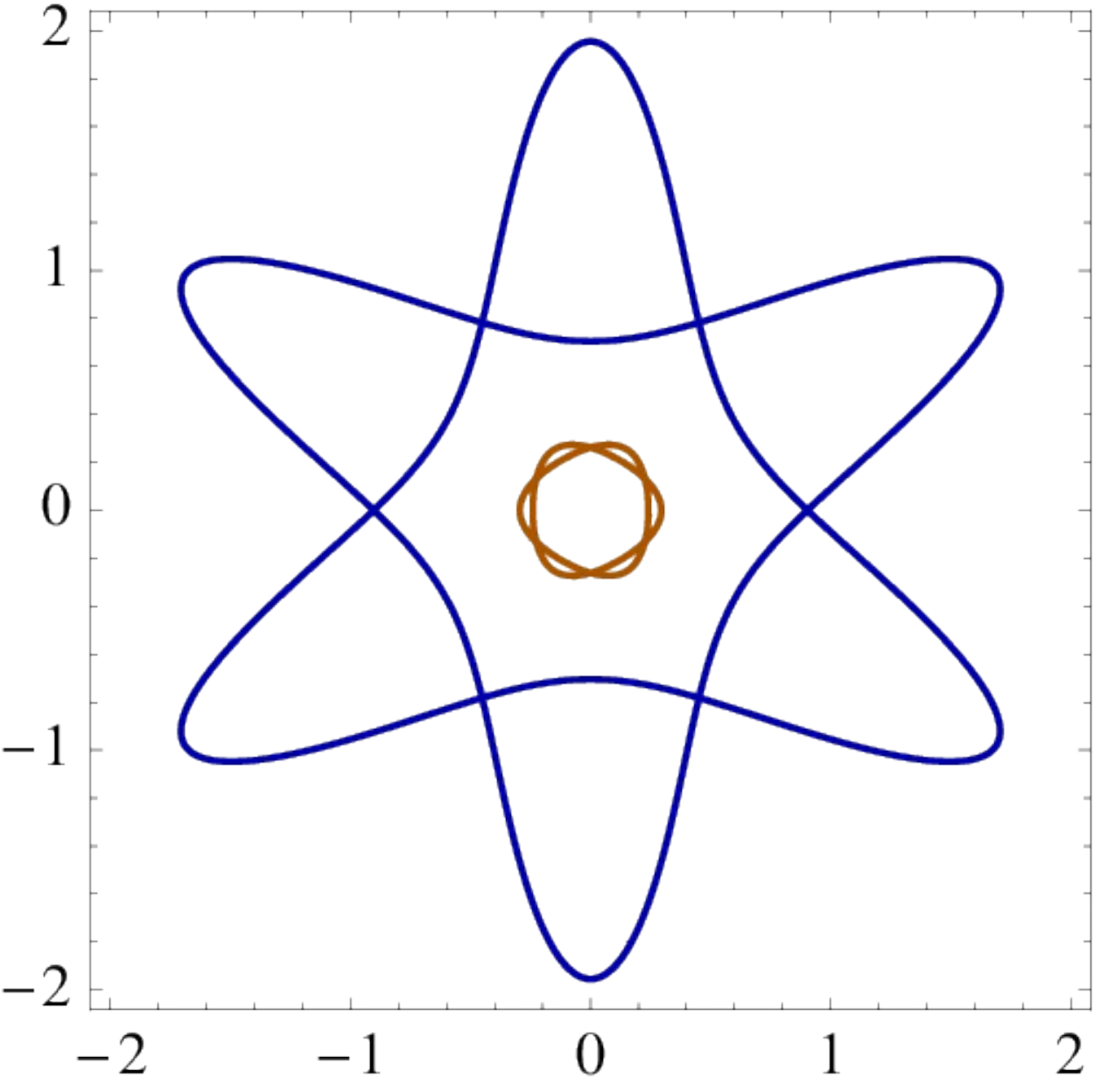}
\hskip0.1cm
\includegraphics[width=3.2cm]{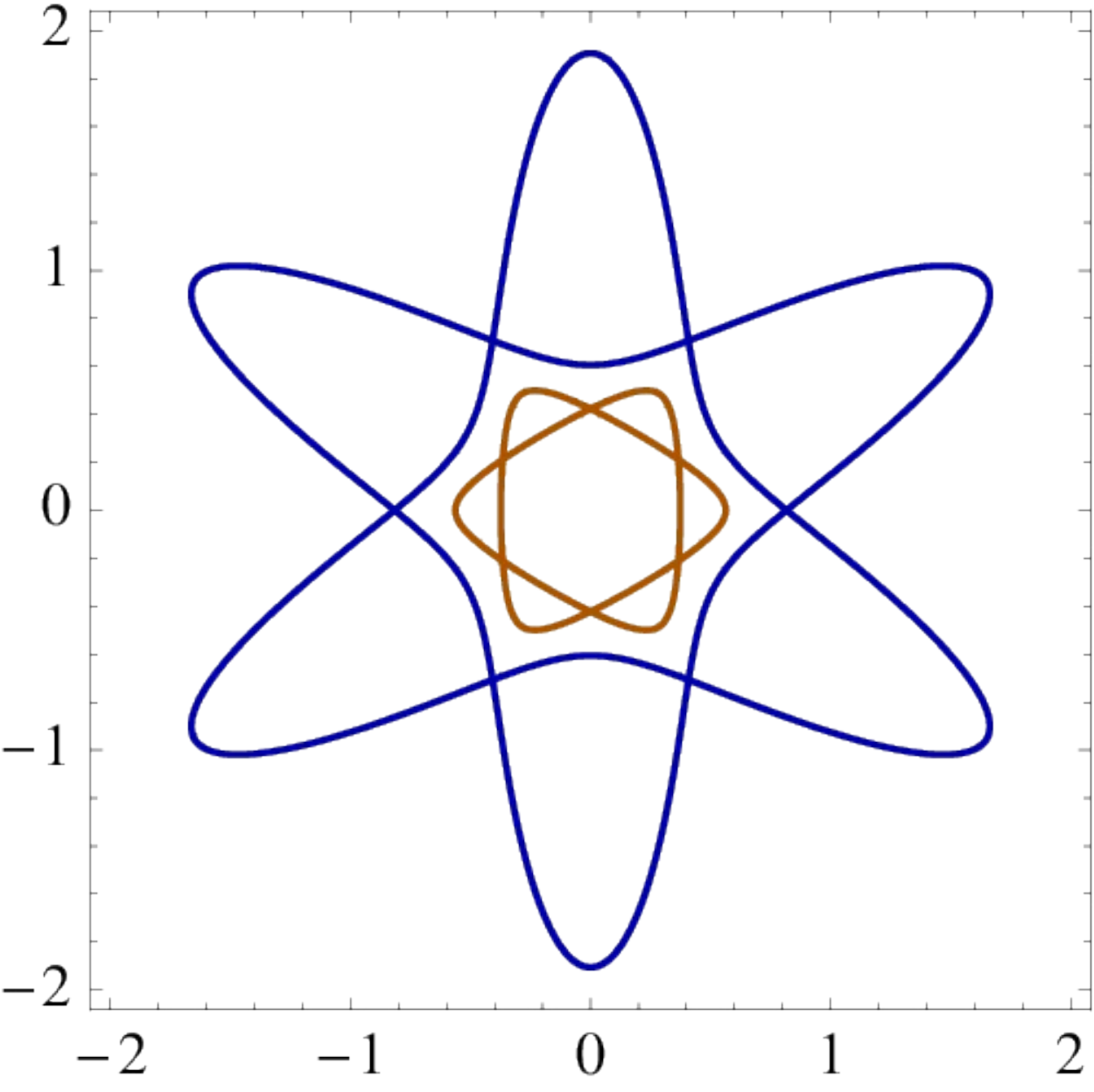}
\hskip0.1cm
\includegraphics[width=3.2cm]{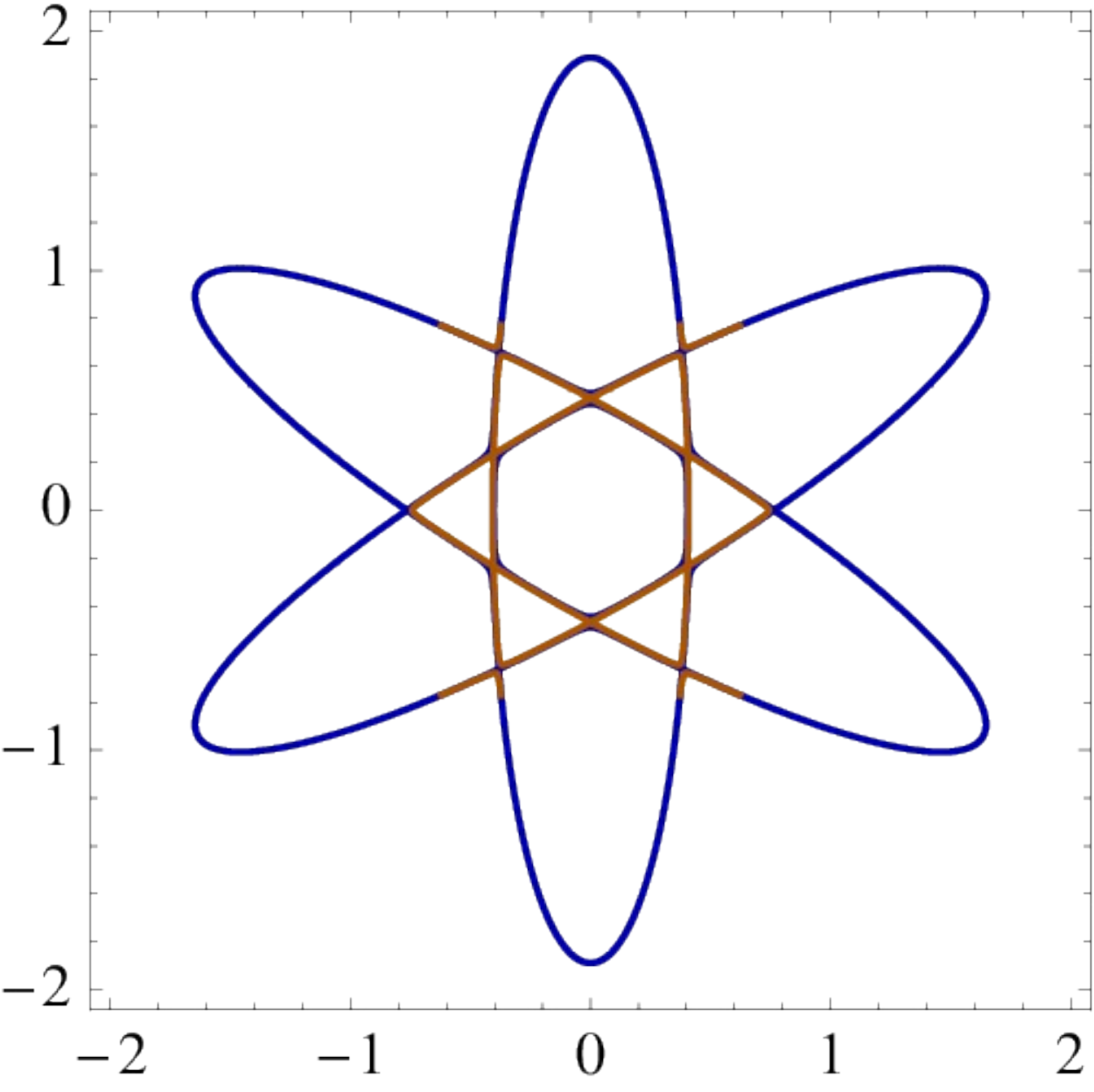}
\\
(a)\hskip3cm
(b)\hskip3cm
(c)\hskip3cm
(d)\hskip3cm
\\[12pt]
\includegraphics[width=3.2cm]{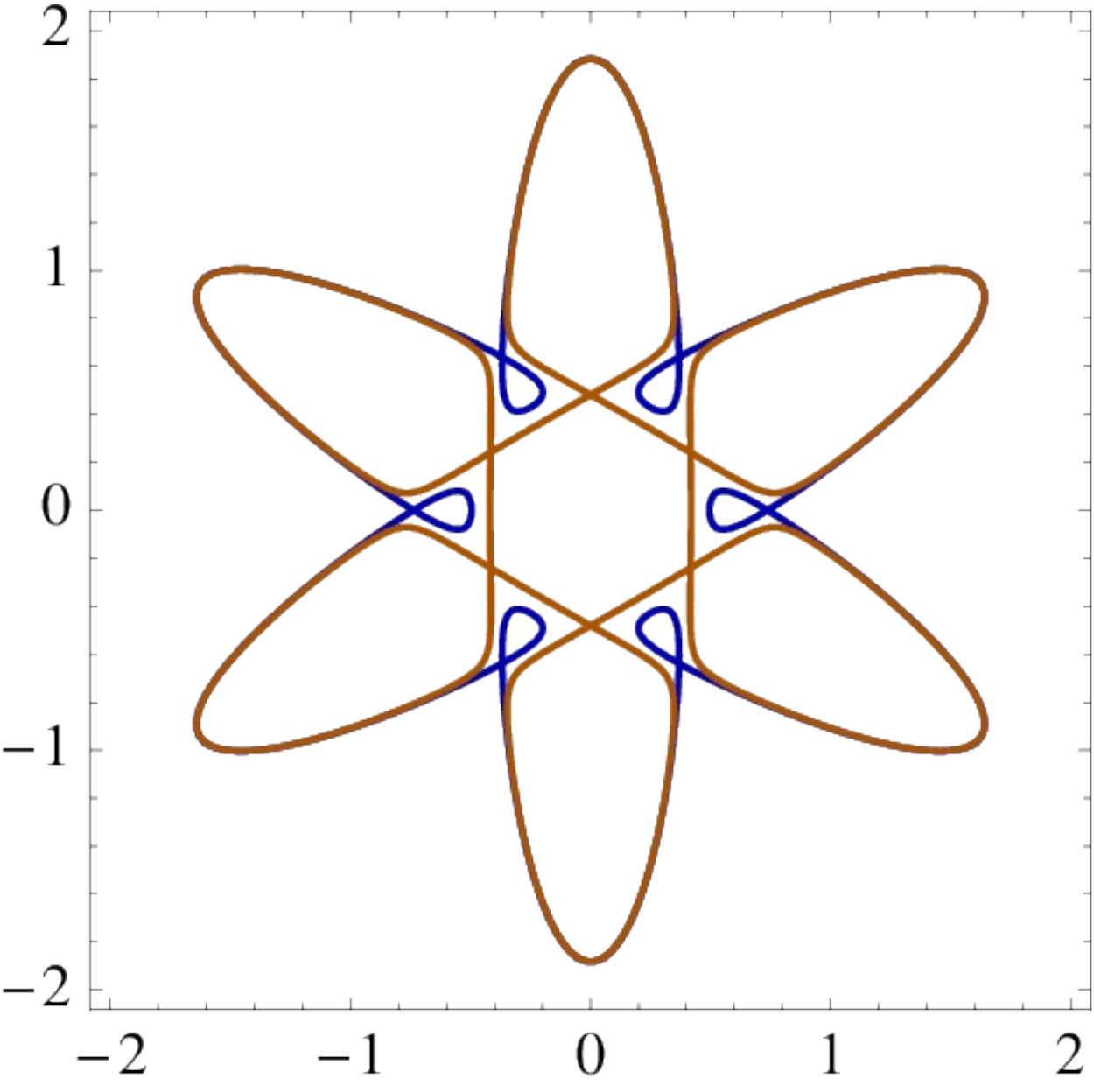}
\hskip0.1cm
\includegraphics[width=3.2cm]{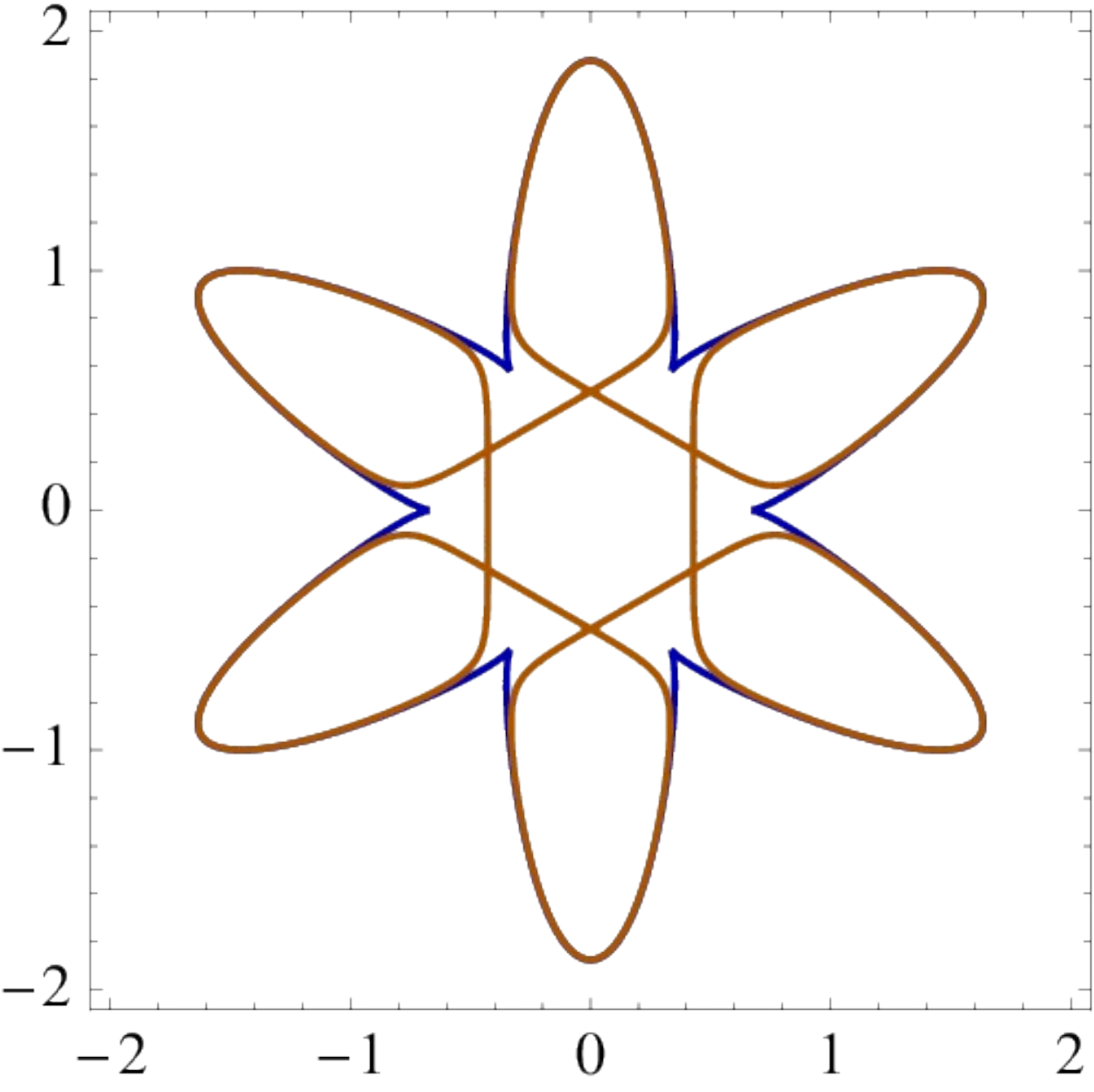}
\hskip0.1cm
\includegraphics[width=3.2cm]{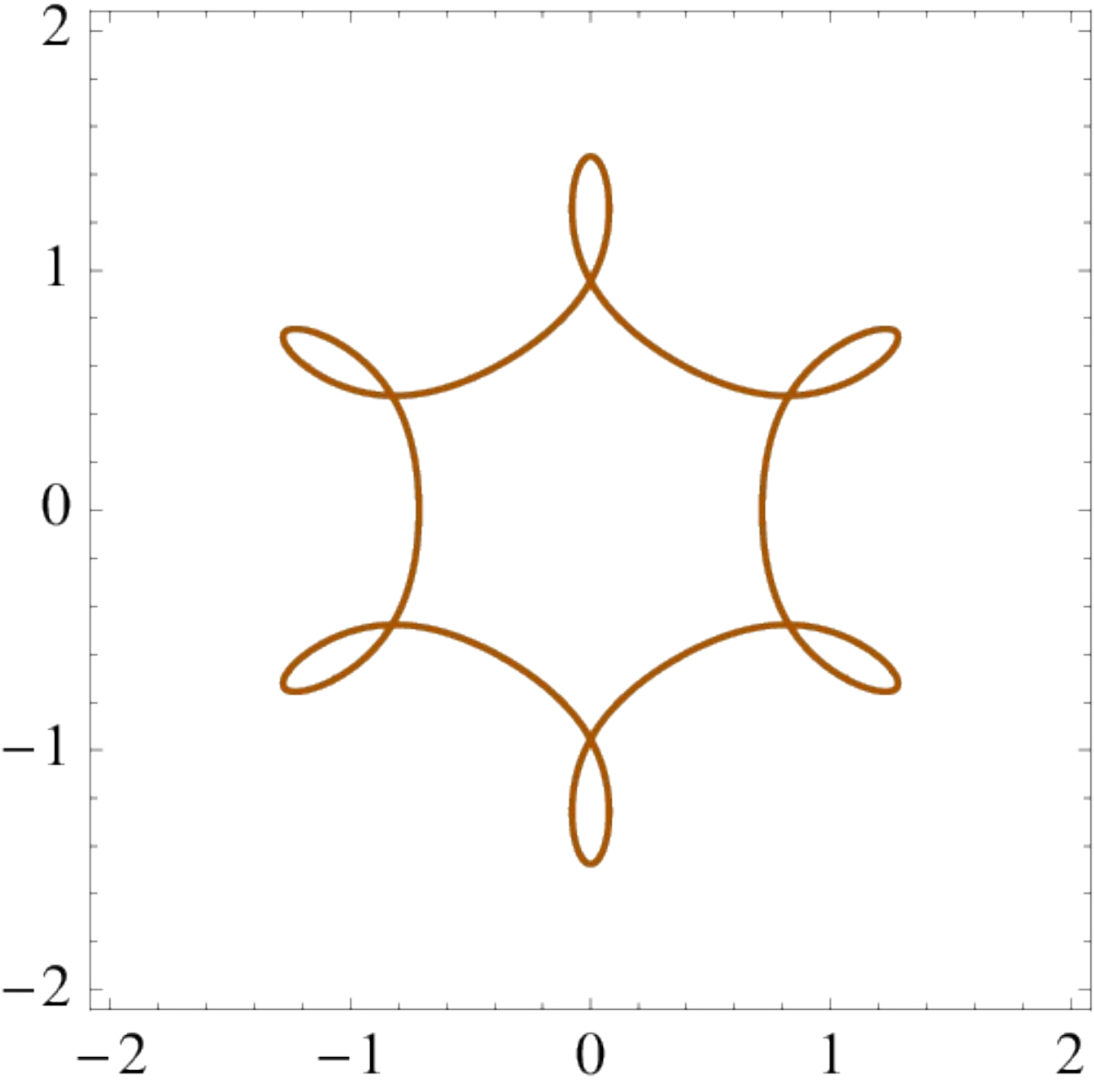}
\hskip0.1cm
\includegraphics[width=3.2cm]{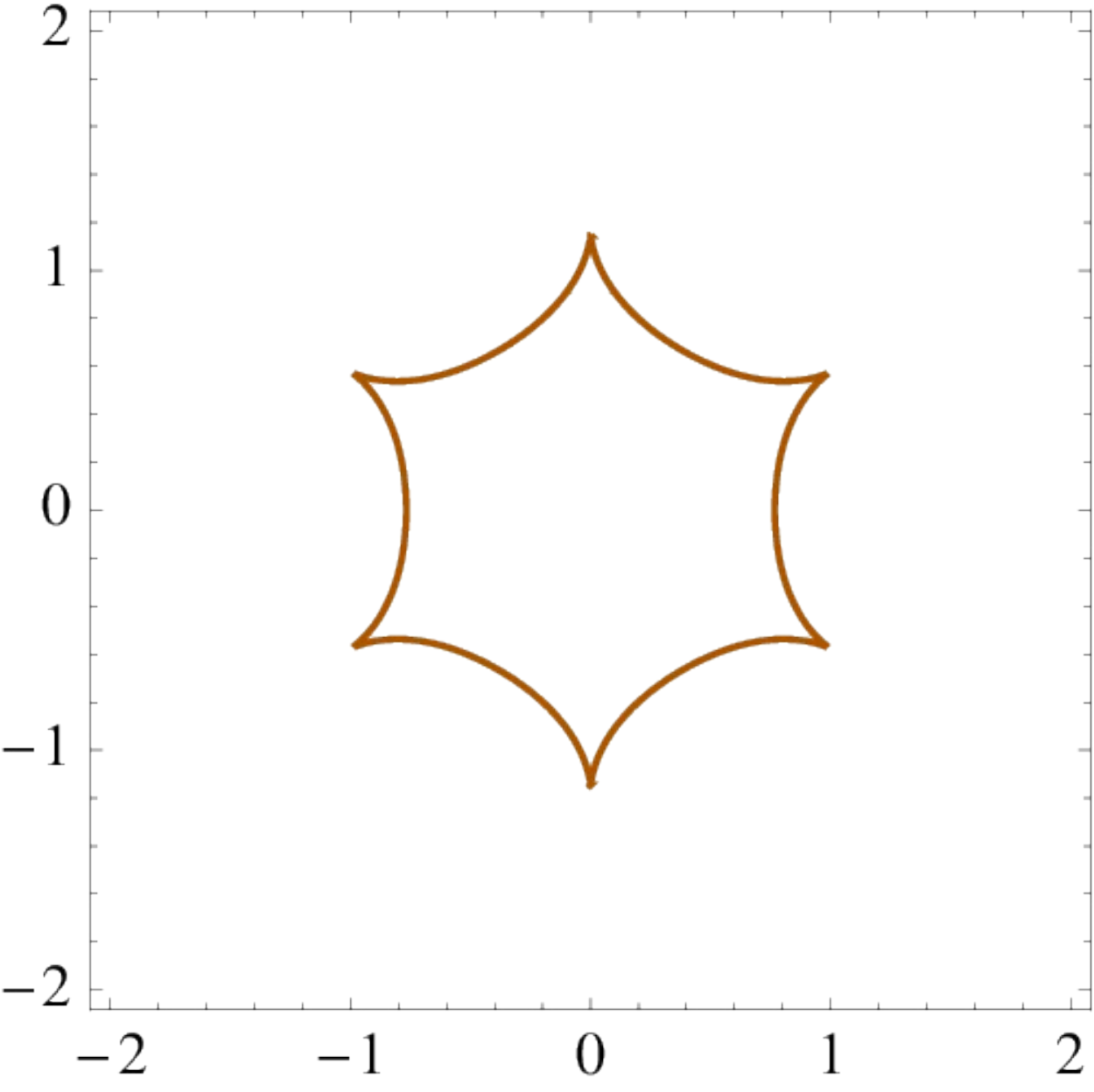}
\\
(e)\hskip3cm
(f)\hskip3cm
(g)\hskip3cm
(h)\hskip3cm

\end{center}
\caption{
Critical level sets of the Hamiltonian (\ref{Eq:barh6})
for $b_0=\frac12$
(the stable case).
\label{Fig:islands6s}}
\end{figure}

\begin{figure}[t]
\begin{center}
\includegraphics[width=3.2cm]{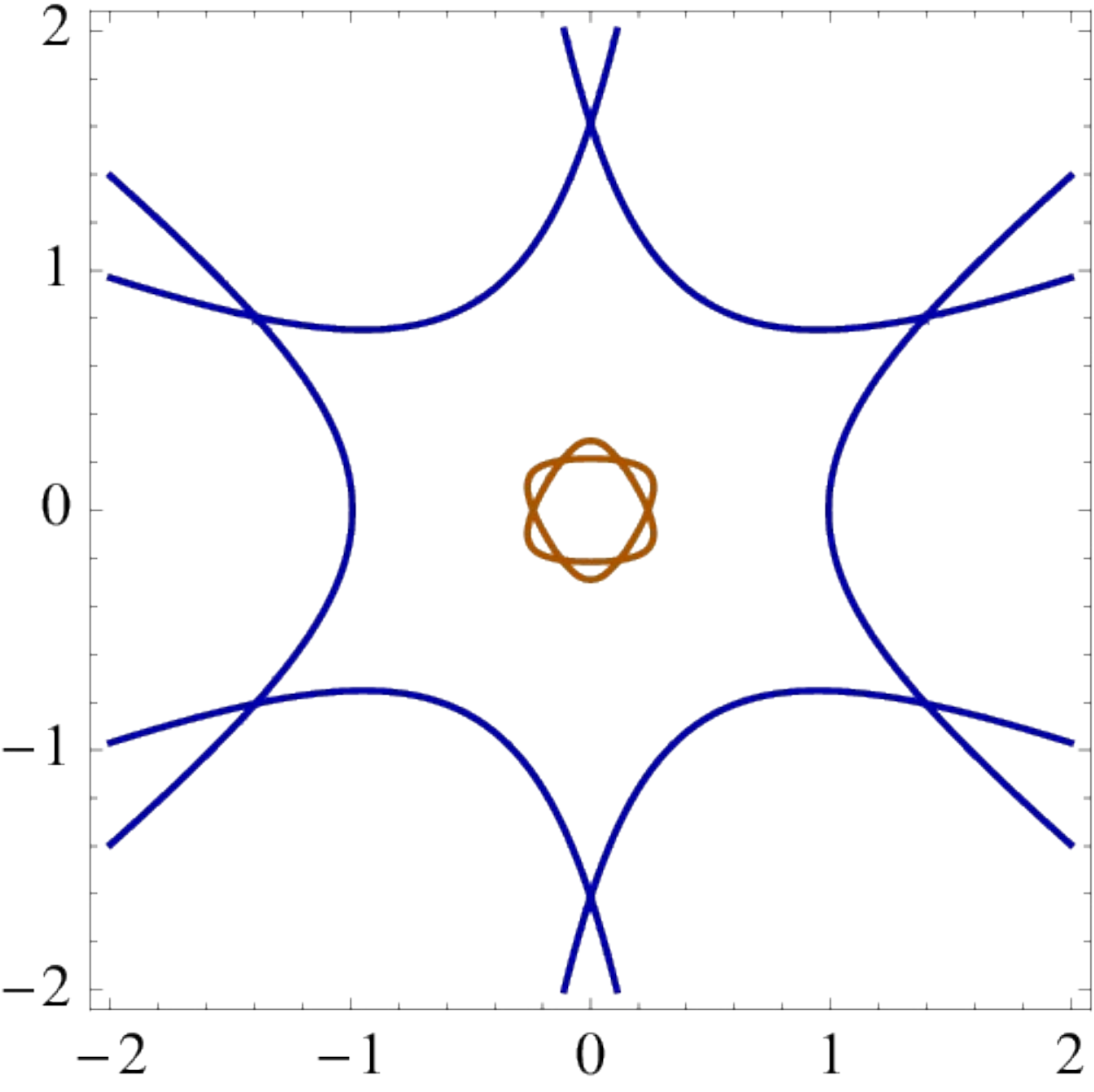}
\includegraphics[width=3.2cm]{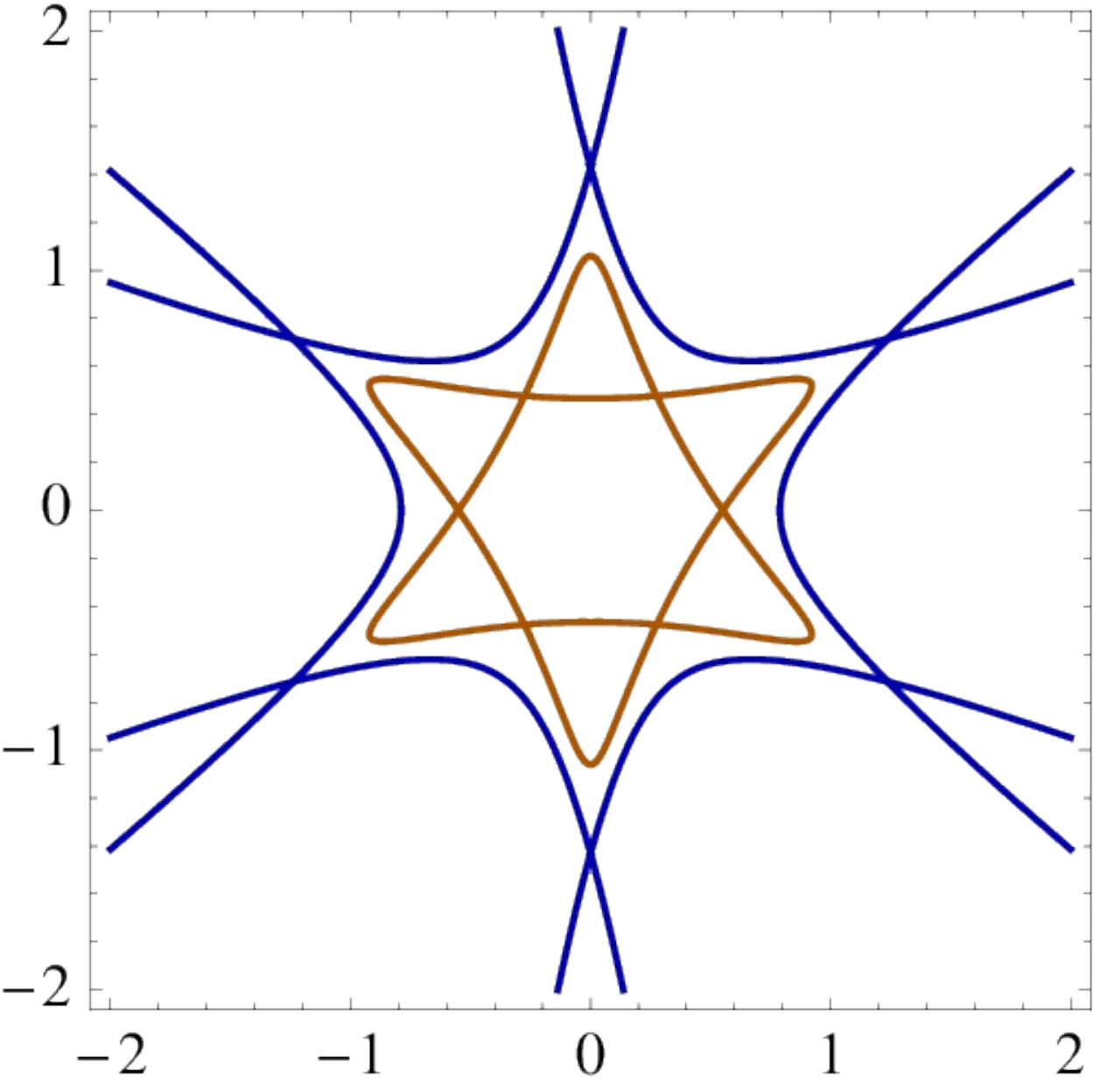}
\includegraphics[width=3.2cm]{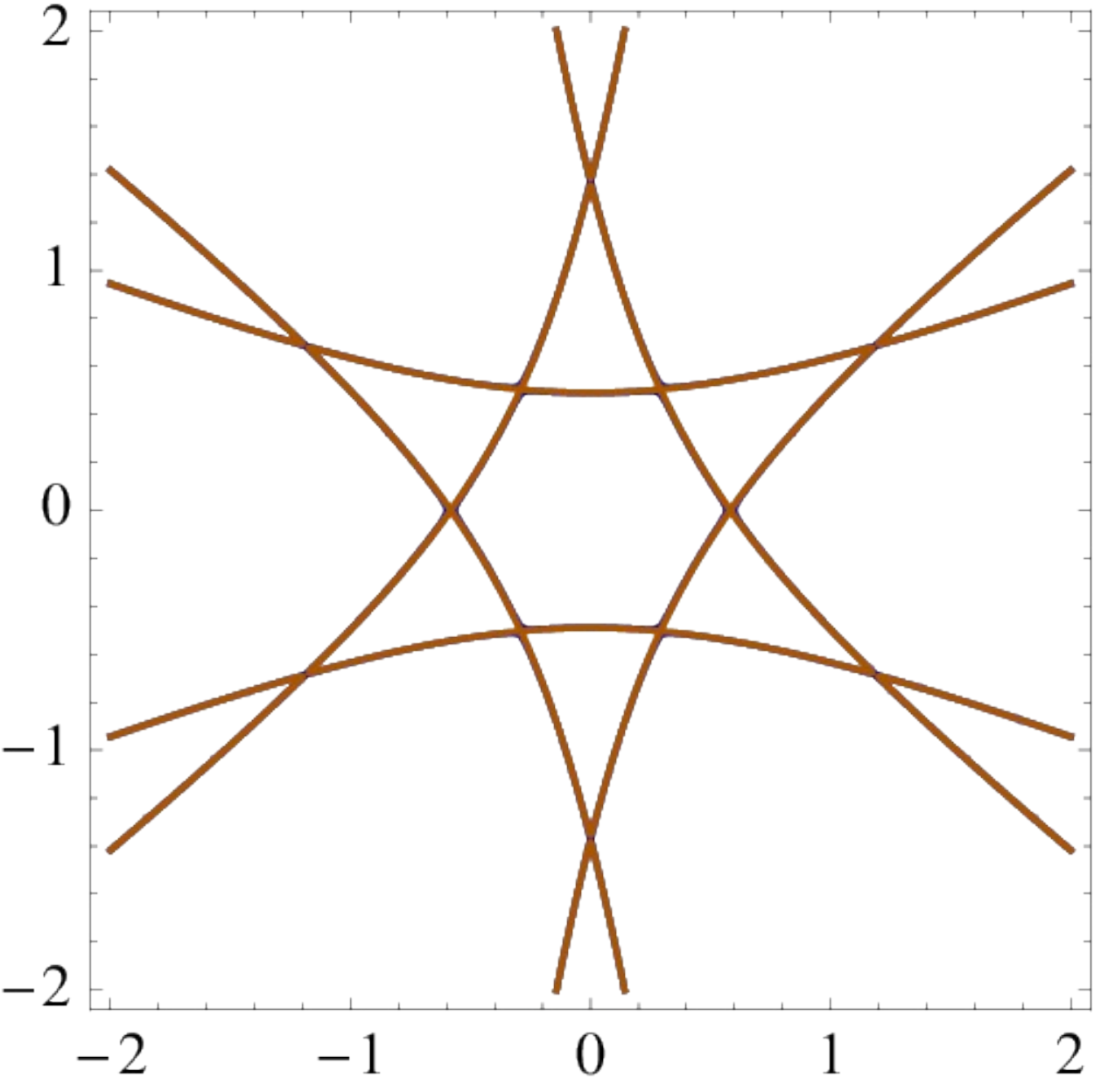}
\\
(a)\hskip3cm
(b)\hskip3cm
(c)\hskip3cm\\[12pt]
\includegraphics[width=3.2cm]{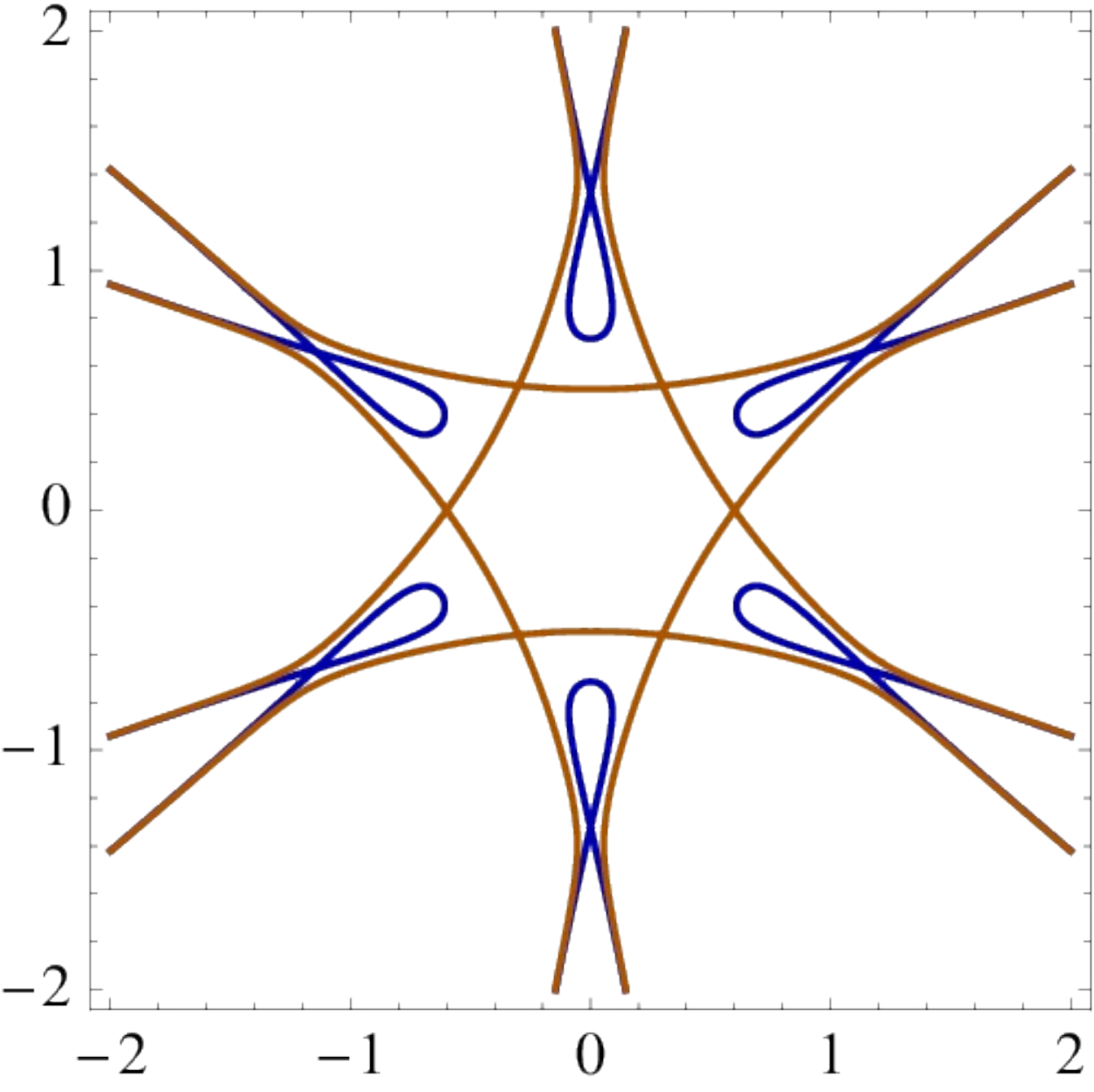}
\includegraphics[width=3.2cm]{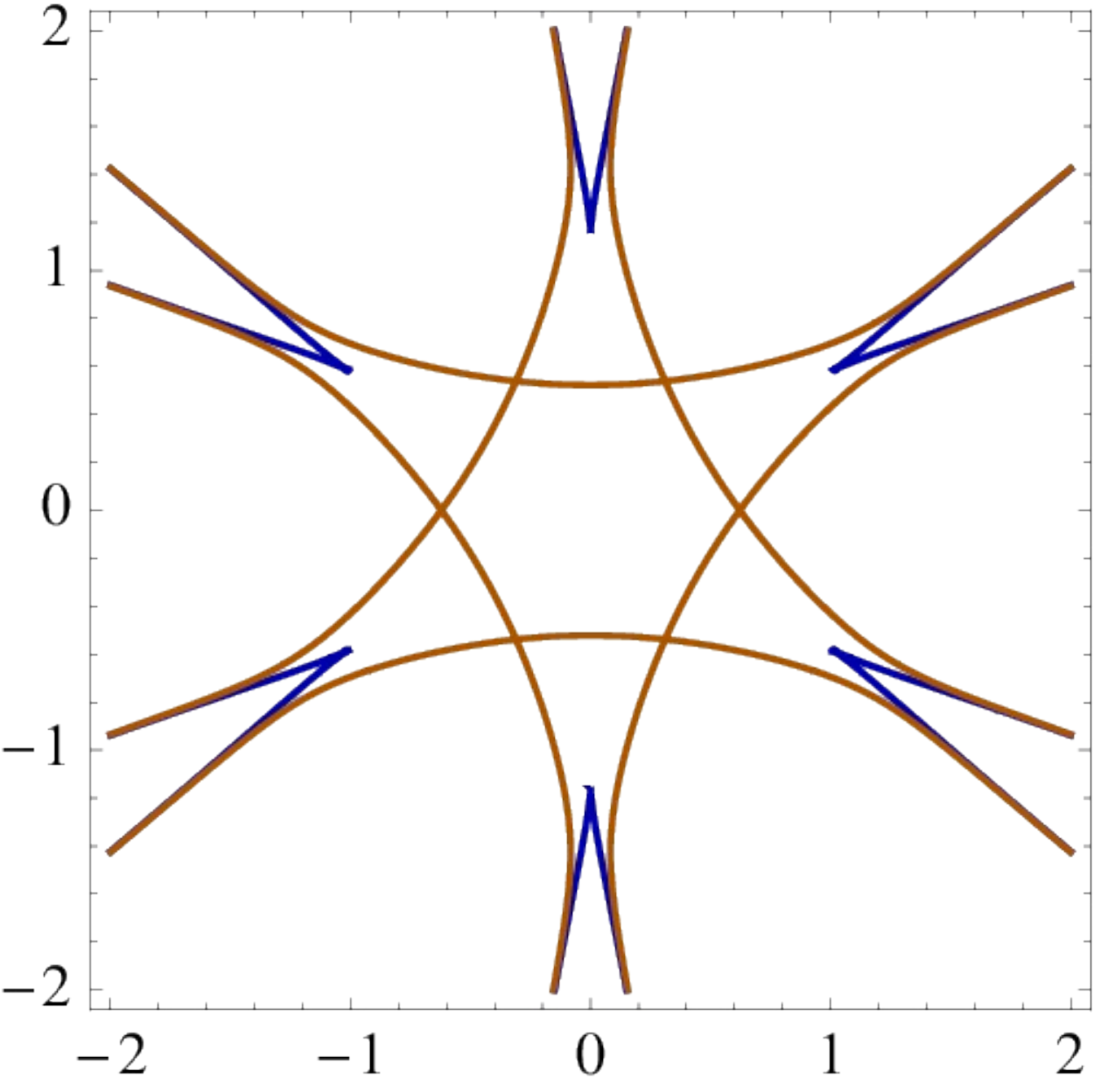}
\includegraphics[width=3.2cm]{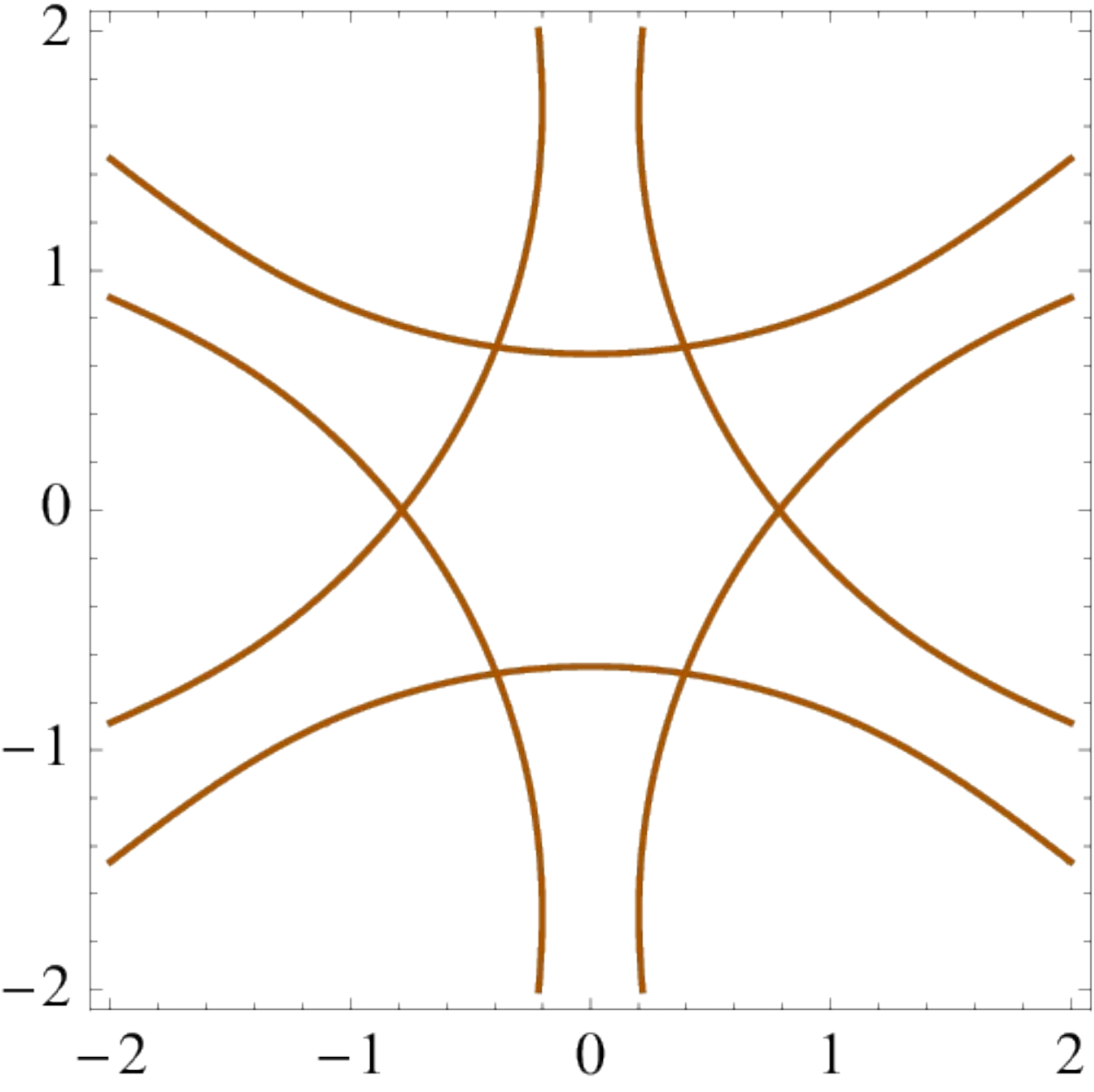}
\\
(d)\hskip3cm
(e)\hskip3cm
(f)\hskip3cm
\end{center}
\caption{Critical level sets of the Hamiltonian (\ref{Eq:barh6})
for $b_0=\frac32$
(the unstable case).
\label{Fig:islands6u}}
\end{figure}

Typical pictures of level sets are illustrated by Fig.~\ref{Fig:n6gen}. The bifurcations 
of islands are illustrated on Figs.~\ref{Fig:islands6s} 
and~\ref{Fig:islands6u} for the stable and unstable case respectively.

\medskip

We note that after the substitution $I=|\delta|^{1/2}J$, $h=|\delta|^{3/2}\bar h$
the Hamiltonian (\ref{Eq:model6}) takes the form
$$
\bar h=\mathrm{sign}(\delta)\,J+|\delta|^{-1/2}\nu J^2+J^3(1+b_0\cos6\varphi)\,.
$$
If $|\delta|\gg\nu^2$, this Hamiltonian is approximated by 
\begin{equation}\label{Eq:barh60}
\bar h_0=\pm J+J^3(1+b_0\cos6\varphi)\,.
\end{equation}
The Hamiltonian $\bar h_0$ depends on the coefficient $b_0$ but is independent of both
$\delta$ and $\nu$.
We conclude that outside a narrow sector near the $\delta$-axis, the critical
level sets of the model Hamiltonian (\ref{Eq:model6})  have the shape which is described
by the Hamiltonian $\bar h_0$.
The level sets for this Hamiltonian are shown on Figure~\ref{Fig:n6gen}.

In order to study bifurcations of the critical level sets of the Hamiltonian (\ref{Eq:model6}) we use the scaling
\begin{equation}\label{Eq:scaling6}
\nu=\epsilon,\quad
\delta=\epsilon^{2} a,
\quad
I=\epsilon J,
\quad
h=\epsilon^{3}\bar h.
\end{equation}
In the new variables the Hamiltonian
  takes the form
\begin{equation}\label{Eq:barh6}
\bar h= a J+J^2+J^3(1+b_0\cos n\varphi)\,.
\end{equation}
We conclude that for $\nu\ne0$ the two-parameter family (\ref{Eq:model6})
is reduce to the family (\ref{Eq:barh6}) which depends on one parameter only.
We note that the transition from (\ref{Eq:model6}) to (\ref{Eq:barh6})
is just a change of variables and parameters and does not involve any approximations.
The bifurcations of the critical level sets of $\bar h$ is illustrated on 
Figs.~\ref{Fig:islands6s}  and~\ref{Fig:islands6u} for the stable and unstable case respectively.
Note that the corresponding analysis takes into account that $\bar h$ 
is to be considered on the domain  $\nu J=\epsilon^2 I>0$ only.

\subsection{$n=5$}

For $n=5$ the normal form  is modelled by the Hamiltonian
\begin{equation}\label{Eq:model5}
h=\delta I+\nu I^2+I^{5/2}\cos 5\varphi\,.
\end{equation}
We note that, in contrast to the cases with $n\ge 6$, in the cases of $n\le 5$
the cubic term $I^3$ can be excluded from the model as it is much smaller
than the term proportional to $I^{n/2}$ and, consequently,
does not affect the qualitative properties of the bifurcation diagram.

The bifurcation diagram for $n=5$ looks similar to the unstable case of~$n=6$.
In particular the equilibrium at the origin is unstable when $\delta=\nu=0$.
The corresponding level set $\{h=0\}$ is not compact and consists of 5 straight lines
defined by the equation $\cos 5\varphi=0$.

Similar to the previous cases, the critical points of $h$ are determined from the equation
$\frac{\partial h}{\partial I}=0$ and $\varphi=0$ or $\varphi=\pi/5 \pmod{2\pi} $.
Therefore the problem is reduced to finding positive solutions for the equations:
\begin{equation}\label{Eq:n5crit}
\delta +2\nu I+\sigma \tfrac{5}{2} I^{3/2}=0 , \qquad \sigma =\pm 1.
\end{equation}
These equations can be rewritten in the form $\delta=f_\sigma(I,\nu)$.
For a fixed $\nu\ne 0$, one of the functions  $f_\sigma$ is monotone
($\sigma\nu>0$) and the other one ($\sigma\nu<0$)
has an extremum at
$$
I^{1/2}=-\tfrac{8}{15}\sigma\nu\,.
$$
Substituting this value into the equation (\ref{Eq:n5crit})
we obtain the relationship between the parameters $\delta$ and $\nu$
which corresponds to a doubled critical point of the Hamiltonian $h$:
$$
\delta= -\tfrac{128}{675}\nu^3.
$$
This equation defines the boundary between domains $D_1$ and $D_2$,
and between $D_1'$ and $D_2'$ on the bifurcation diagram shown on Figure~\ref{Fig:n5diagram}(a).
In $D_1$ and $D_2$ the Hamiltonian $h$ has 5 critical points (of saddle type) 
and in $D_2$ and $D_2'$ it has 15 critical points:
5  elliptic and 10 saddles.

When the parameters  $(\delta , \nu )$ are in $D_1$ or $D_1'$,
singular level sets of $h$ look similar to the one shown  on Figure~\ref{Fig:n5diagram}(b).

\begin{figure}
\begin{center}
\includegraphics[height=4.7cm]{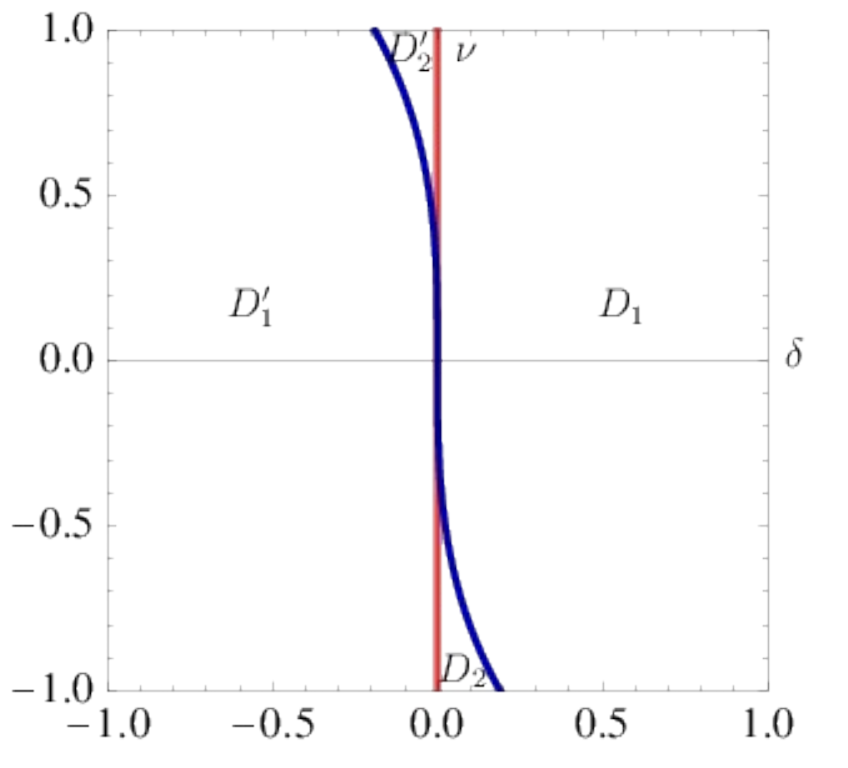}
\hskip1cm
{\includegraphics[height=4.7cm]{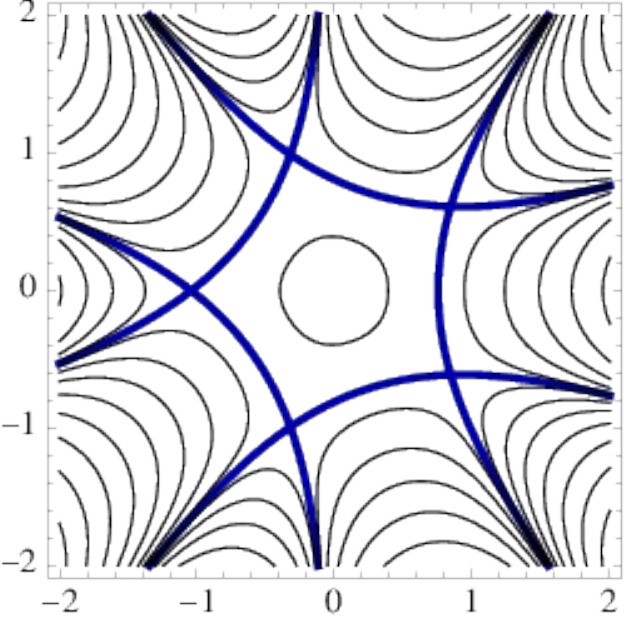}}\\
(a)\kern5.5cm(b)
\end{center}
\caption{(a) Bifurcation diagram on the plane $(\delta,\nu)$ for $n=5$
and (b)  level sets for the limit Hamiltonian (\ref{Eq:barh50}), the critical level set is 
shown using the bold lines.
\label{Fig:n5diagram}}
\end{figure}

\begin{figure}
\begin{center}
\includegraphics[width=3.3cm]{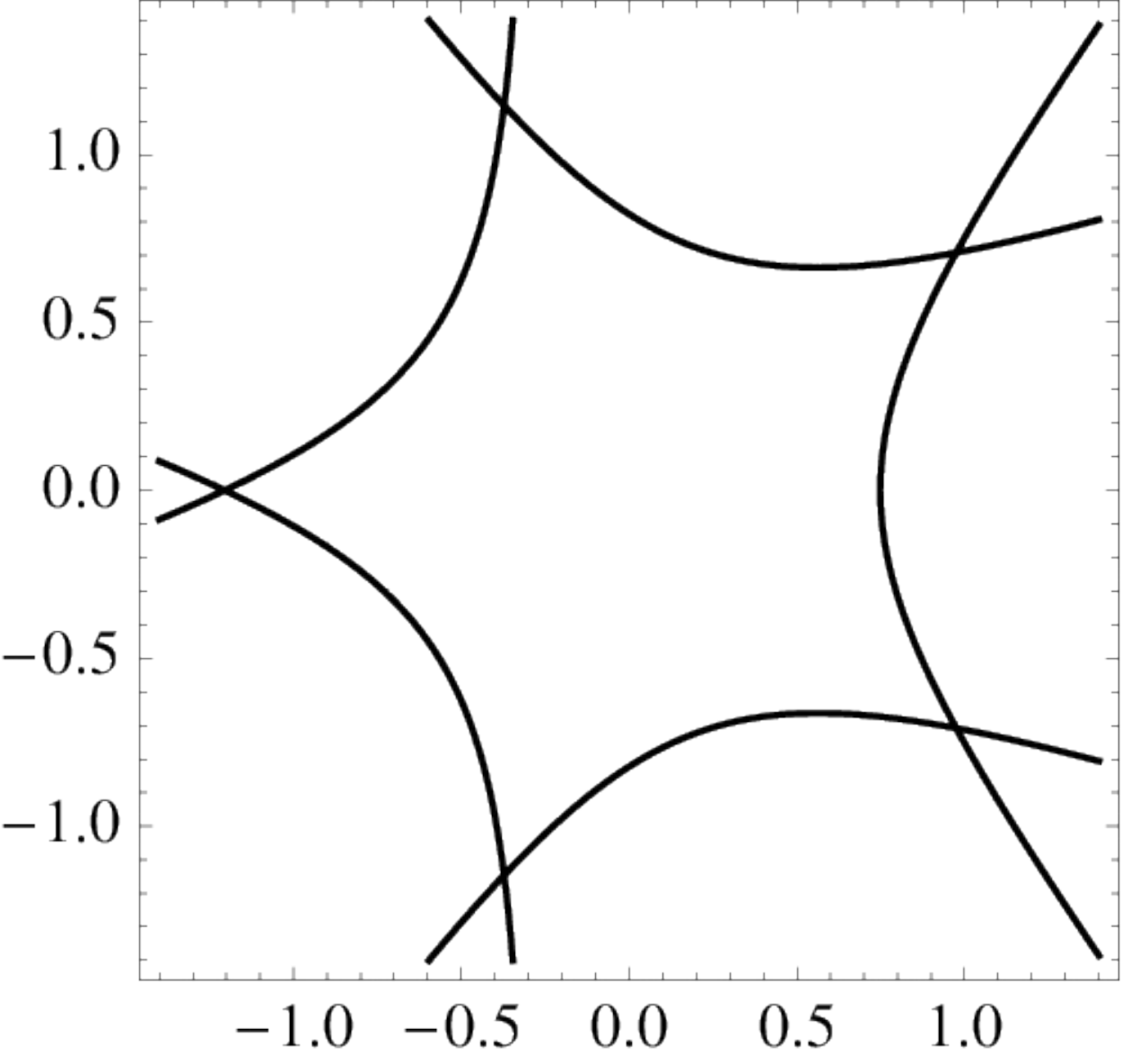}
\includegraphics[width=3.3cm]{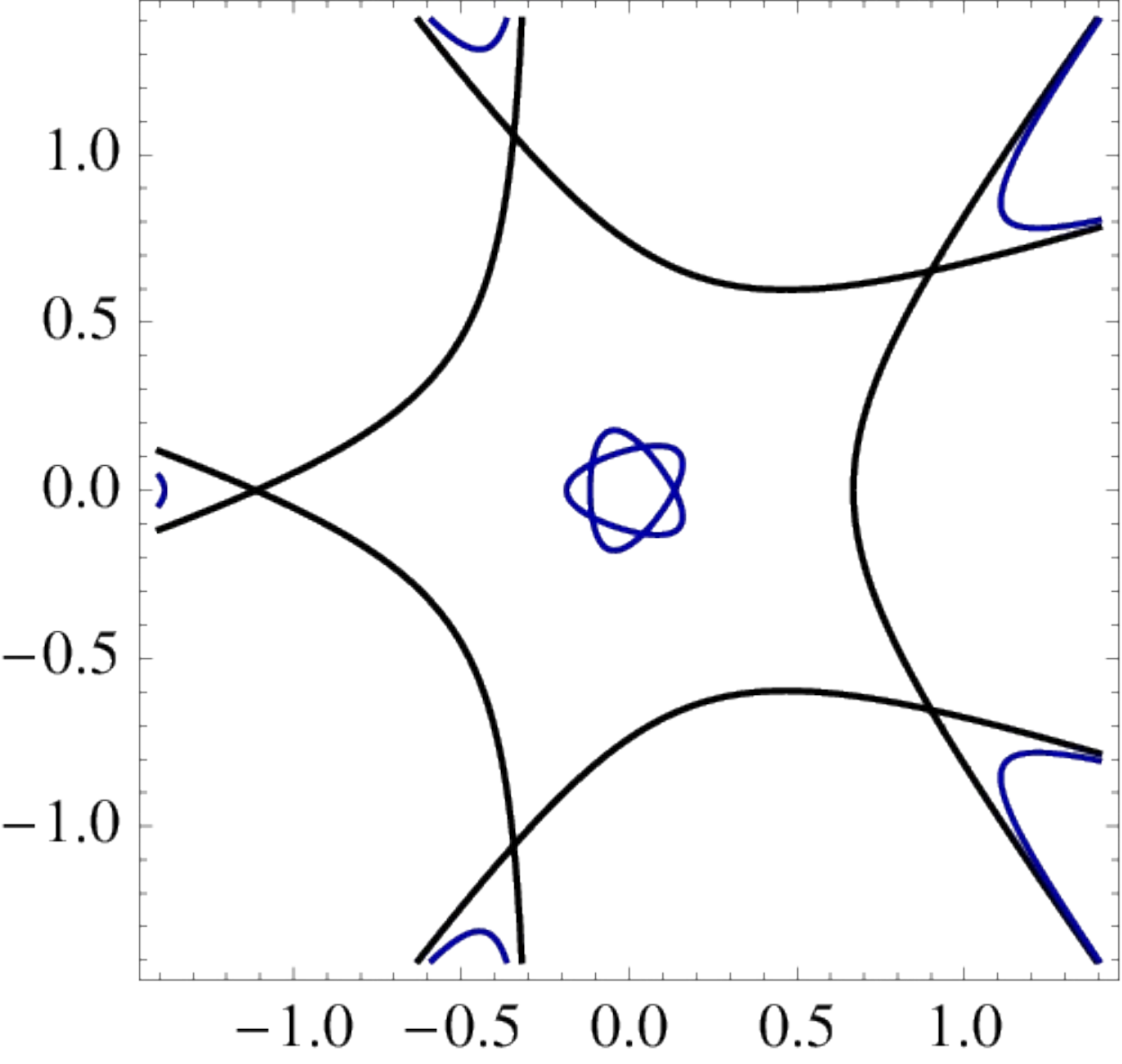}
\includegraphics[width=3.3cm]{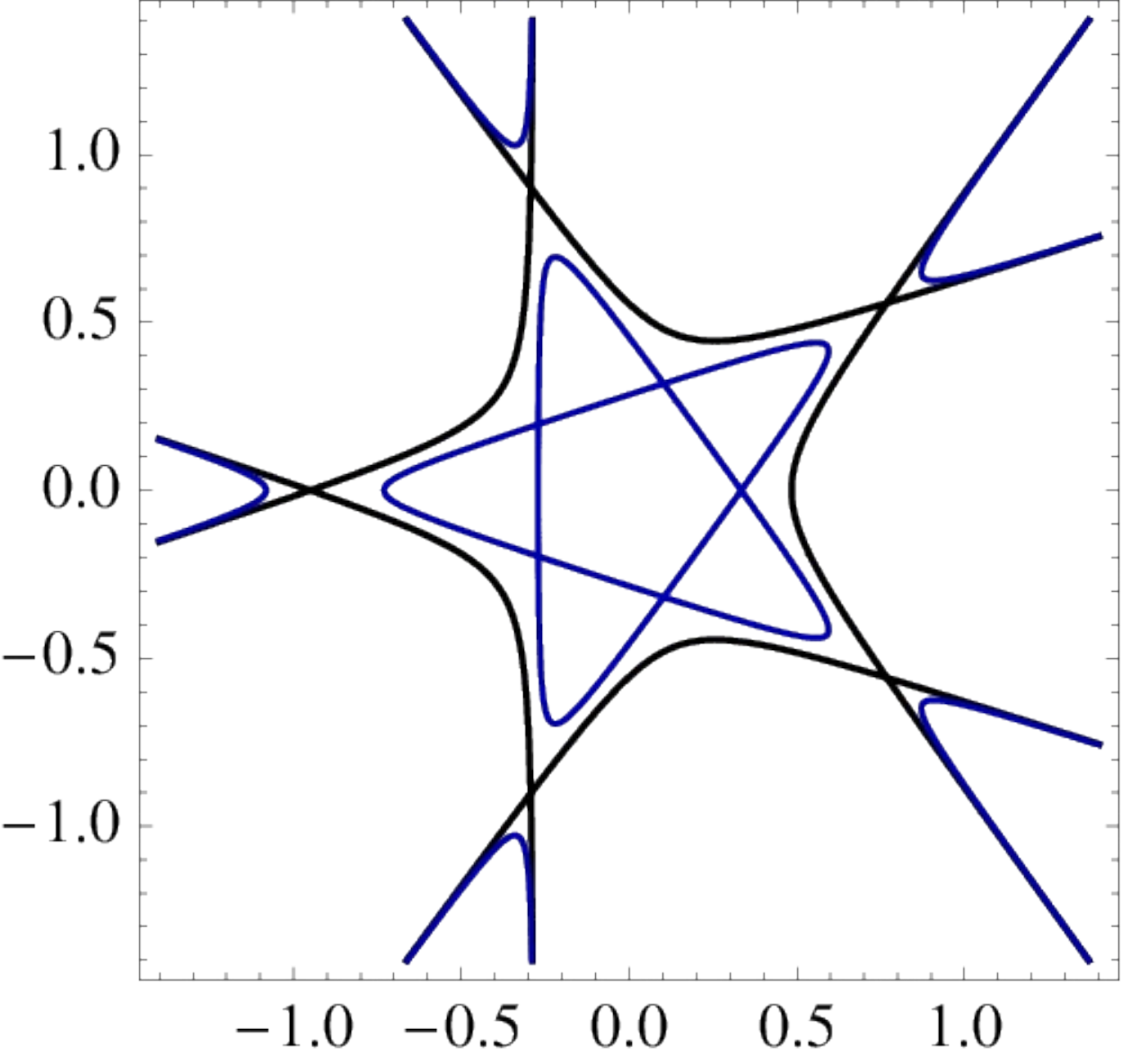}
\includegraphics[width=3.3cm]{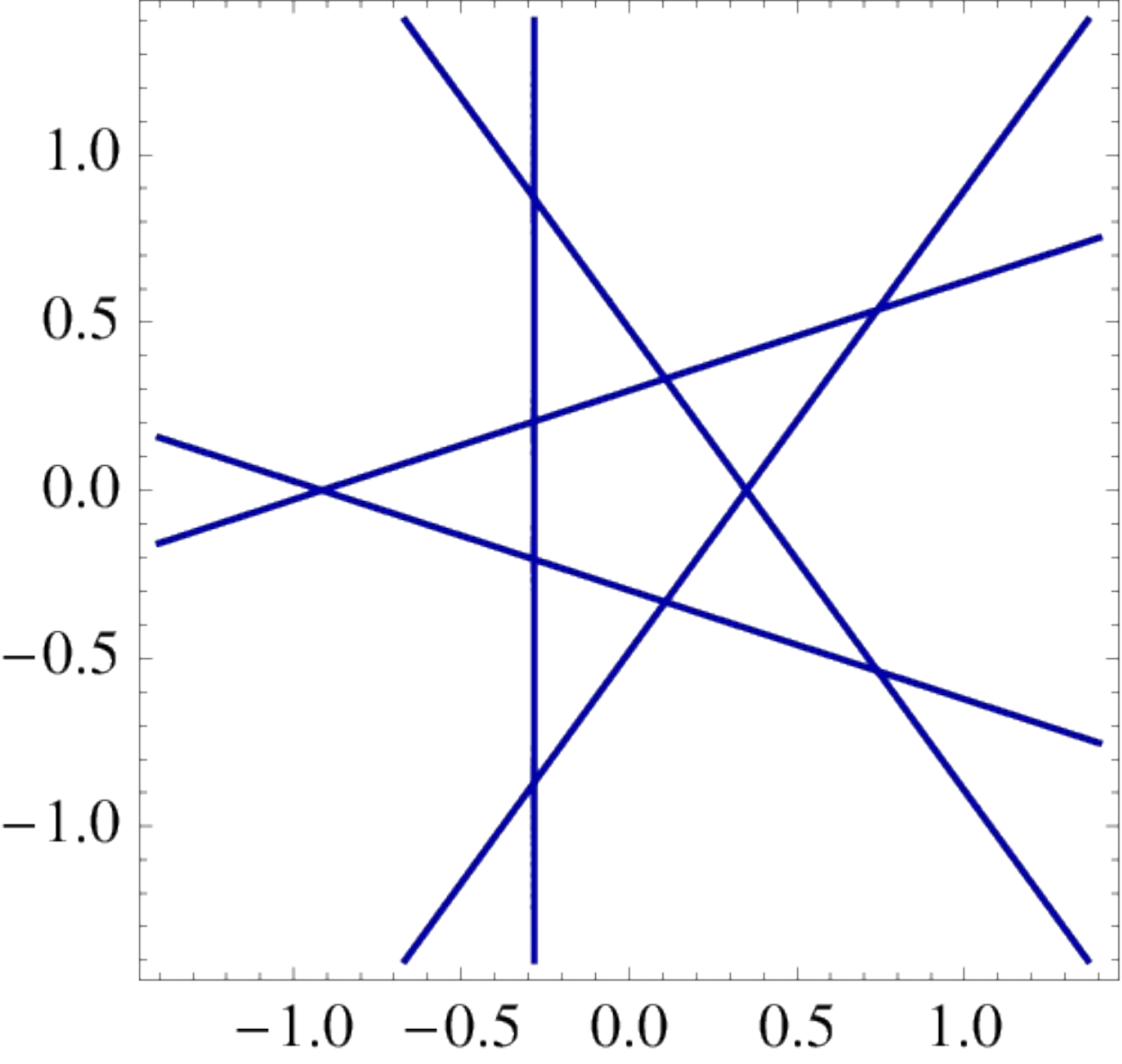}\\
(a)\hskip3.0cm
(b)\hskip3.0cm
(c)\hskip3.0cm
(d)\hskip3.0cm
\\[12pt]
\includegraphics[width=3.3cm]{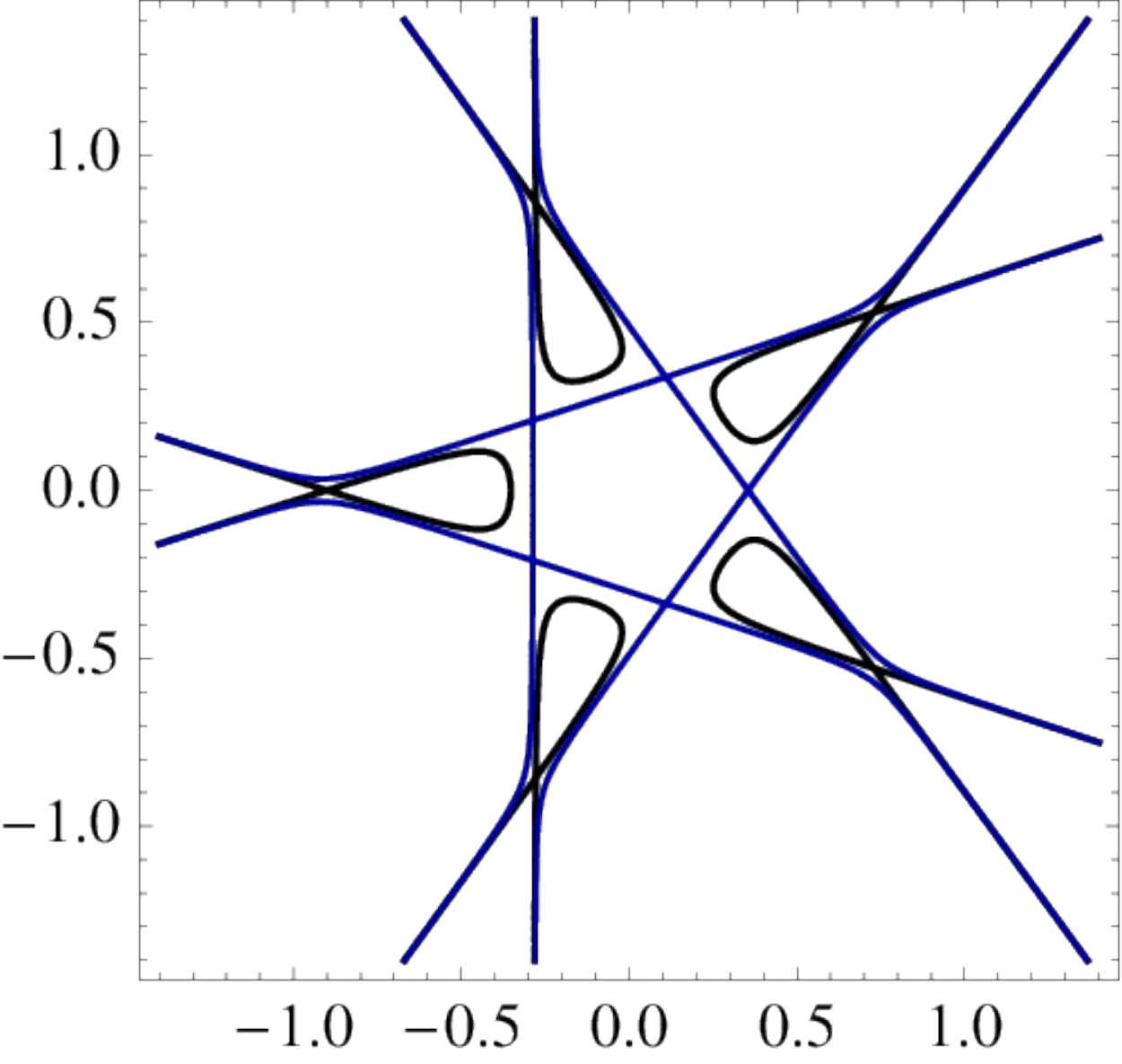}
\includegraphics[width=3.3cm]{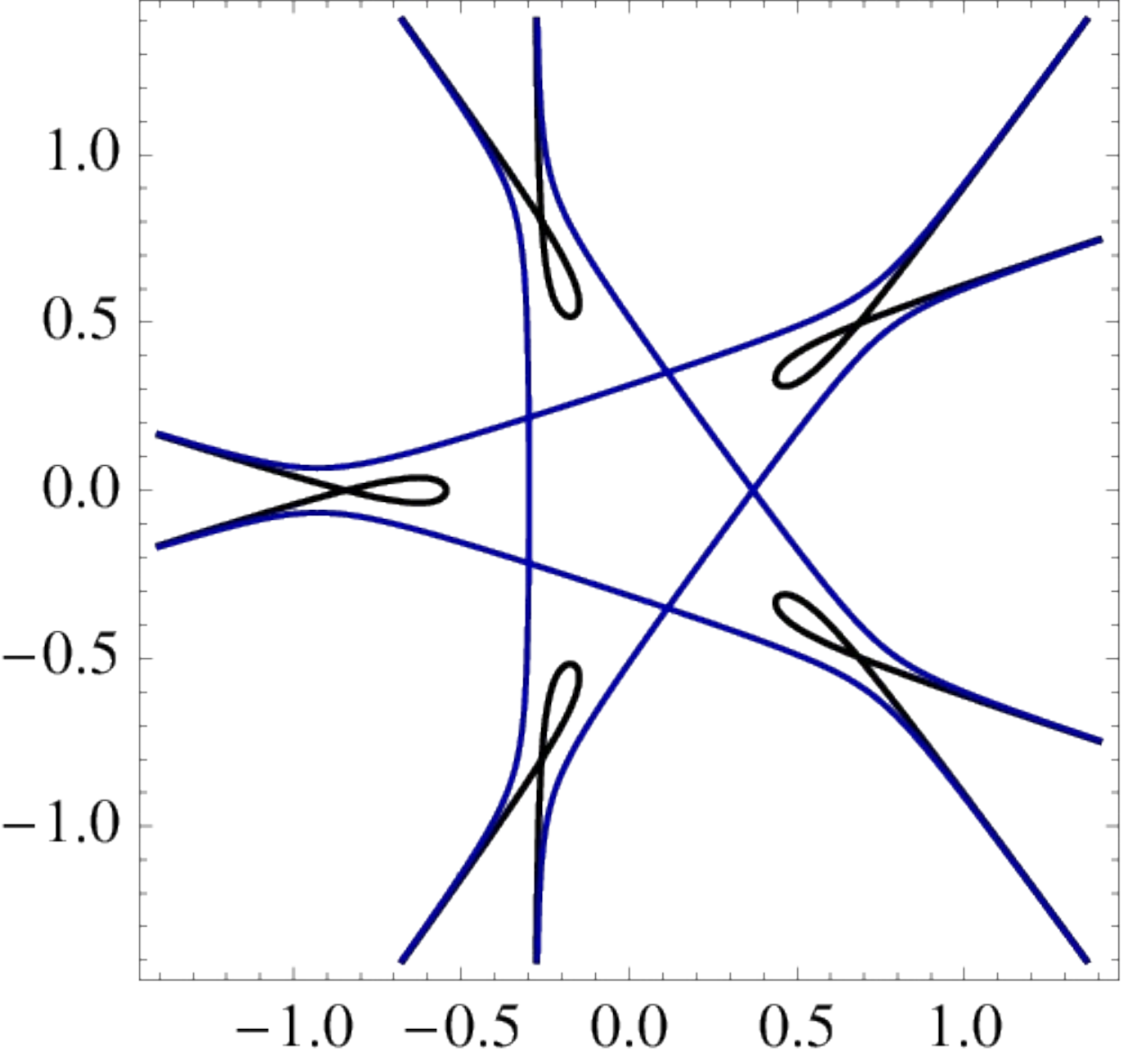}
\includegraphics[width=3.3cm]{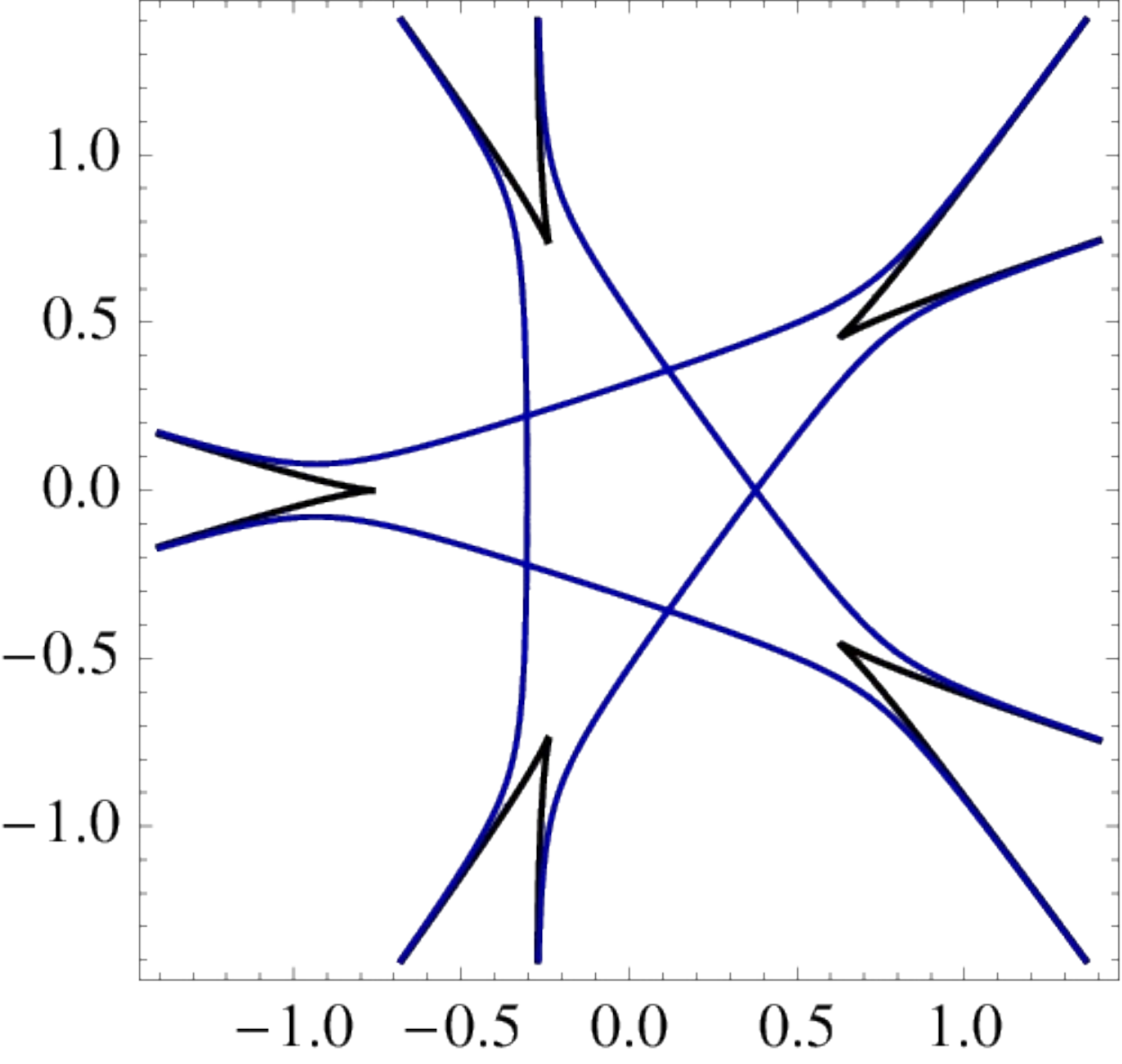}
\includegraphics[width=3.3cm]{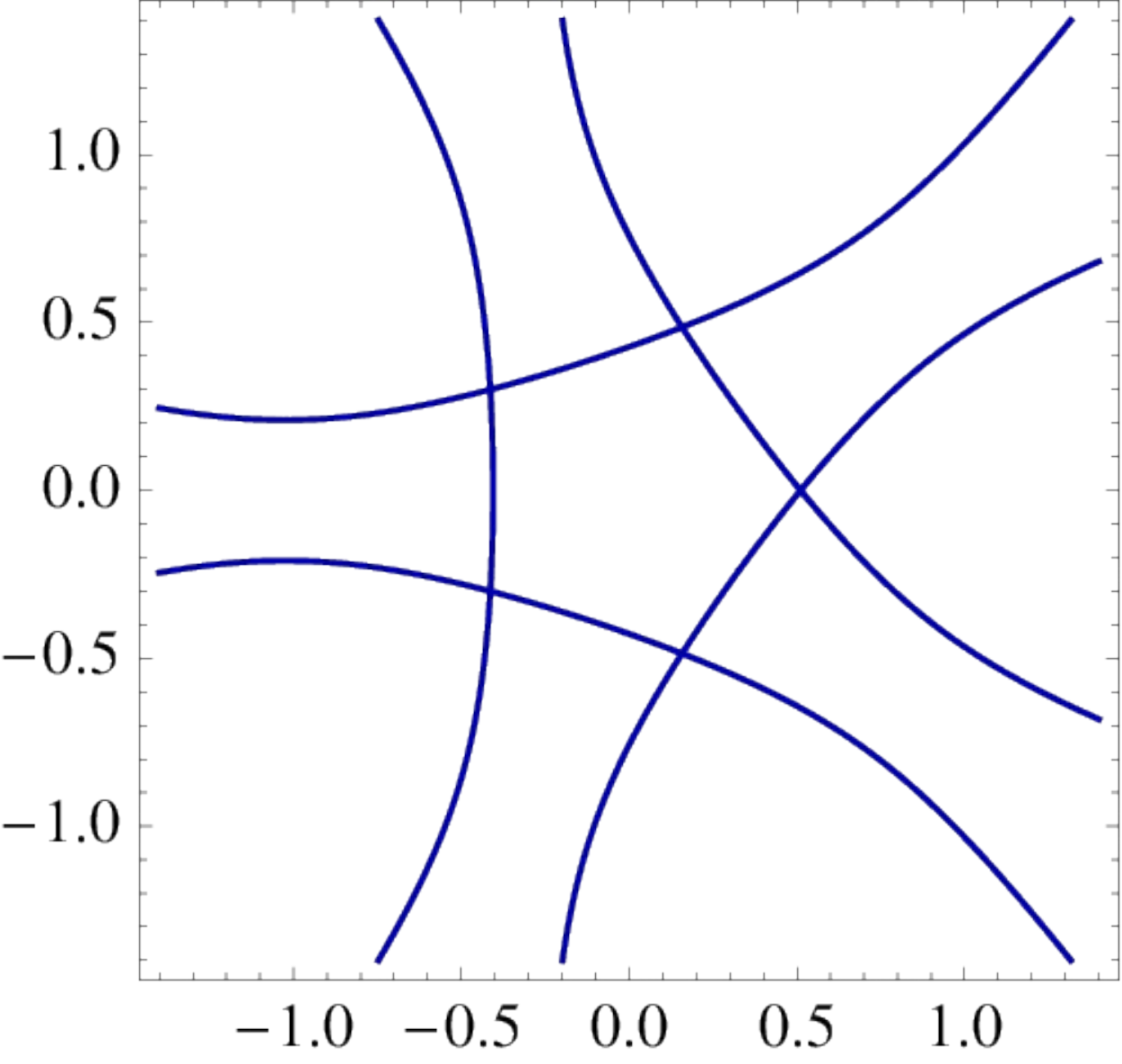}\\
(e)\hskip3cm
(f)\hskip3cm
(g)\hskip3cm
(h)\hskip3cm
\end{center}
\caption{Critical level sets lines for the Hamiltonian (\ref{Eq:modelbif5}), the parameter $a$ decreases from left to right
and from top to bottom. Special values:
(a) $a=0$, (d) $a=-\frac4{25}$, (g) $a=-\frac{128}{675}$.}
\label{Fig:n5bif}
\end{figure}

A chain of $n$ islands is born from the origin and go through a bifurcation shown in Figure~\ref{Fig:n5bif} as the parameters $(\delta , \nu )$
 move in anticlockwise direction through $D_2$ or $D_2'$ in the plane of Figure~\ref{Fig:n5diagram}(a).
We note that the symmetry of the Hamiltonian (\ref{Eq:model5})
allows us to consider the case $\delta\ge0$ only.

In order to study the shape of the islands we apply the scaling
\begin{equation}\label{Eq:scaling5a}
I=\delta^{2/3} J,
\quad
h=\delta^{5/3}\bar h
\,.
\end{equation}
In the new variables the Hamiltonian
function (\ref{Eq:model5})  takes the form
\begin{equation}
\bar h=   J+ \nu \delta^{-1/3}J^2+J^{5/2}\cos 5\varphi\,.
\end{equation}
If $\nu \delta^{-1/3}$ is small, the quadratic term can be ignored and we arrive to the  Hamiltonian
\begin{equation}\label{Eq:barh50}
\bar h_0 =J+J^{5/2} \cos 5 \varphi \, . 
\end{equation}
Its critical  points are located at the point $J_{cr}=(\tfrac{2}{5})^{2/3}$, $\varphi_{cr}=\pi /5 \pmod {2\pi} $ and the corresponding
 critical value is given by $ \bar h_0 (J_{cr},\varphi_{cr} ) =\tfrac{3}{5}(\tfrac{2}{5})^{2/3}$.
 So the critical set is defined by the equation
\begin{equation}
\label{Eq:ham5D1} 
J+J^{5/2} \cos 5 \varphi =\tfrac{3}{5}(\tfrac{2}{5})^{2/3}  .
\end{equation}
This critical set and some non-critical level lines of $\bar h_0$ are shown on  Figure~\ref{Fig:n5diagram}(b).

If $|\delta|\gg|\nu|^3$, the critical points of $\bar h$ are close to the critical points of $\bar h_0$
and the shape of the corresponding critical set is 
approximately defined by the equation (\ref{Eq:ham5D1}). 
Taking into account the scaling, we conclude that the critical counts of the
original model (\ref{Eq:model5}) are located on the circle $I=J_{\mathrm{cr}}\delta^{2/3}(1+O(\nu\delta^{-1/3}))$.

Note that on a small disk $\delta^2+\nu^2\le \epsilon_0^2\ll 1$, 
 the assumption $|\delta| \gg |\nu^3 |  $  is violated 
 only in a narrow zone near the $\delta$-axis.
 The relative area of this zone converges to zero when $\epsilon_0$ goes to zero. This zone includes 
the domains  $D_2$ and $D_2'$.

In order to study the changes in the critical level sets of $h$ in the region where the above
assumption is possibly violated,
we introduce another scaling
\begin{equation}\label{Eq:scaling5b}
\delta=a\nu^3 ,
\quad
I=\nu^2 J,
\quad
h=\nu^5\bar h\,.
\end{equation}
In the new variables the Hamiltonian
function (\ref{Eq:model5})  takes the form
\begin{equation}\label{Eq:modelbif5}
\bar h= a   J+J^2+J^{5/2}\cos 5\varphi\,.
\end{equation}
We note that the Hamiltonian $\bar h$ depends on one parameter only.
Moreover, the change is defined provided $\nu\ne0$.
On every cubic curve $\delta = a \nu^3$ in the plane $(\delta, \nu)$, 
 the Hamiltonian (\ref{Eq:model5}) is equivalent to the same Hamiltonian
 (\ref{Eq:modelbif5}) up to scaling of the space, energy and time variables.
 
 The Hamiltonian (\ref{Eq:modelbif5}) has $n$ saddle critical points if $a>0$ and $a<-\frac{128}{625}$.
 The number of saddle critical points is $2n$ when $0>a>-\frac{128}{625}$. We note that
 for $a=-\frac4{25}$ both families of critical points belong to a single critical level set. The 
  changes in the critical level sets are illustrated on Figure~\ref{Fig:n5bif}.

%%%%%%%%%%%%%%%%%%%%%%%%%%%%%%%%%%%%%%%%%%%%%%%

\section{Simplification of formal Hamiltonians}
\subsection{Simplification of a degenerate Hamiltonian \label{Se:h22is0}}

It is well known that the complex variables defined by
$$
z=x+iy \qquad\mbox{and}\qquad\bar z=x-i y
$$
facilitate manipulation with real formal series in variables $x,y$.
It is important to note that $\bar z$ is not necessarily the complex conjugate of $z$, the latter
is denoted by $z^*$ in this paper. Of course, $\bar z=z^*$ if both $x$ and $y$ are real.

In this variables, a rotation $R_{\alpha_0}$ 
takes the form of the multiplication: $z\mapsto \lambda_0 z$, $\bar z\mapsto \lambda_0^*\bar z$,
where $\lambda_0=e^{i\alpha_0}$. Let $\lambda_0$ be resonant of order $n$, i.e., $\lambda_0^n=1$. 
Any formal
power series $h$ invariant with respect to the rotation $R_{\alpha_0}$ has the form
\begin{equation}\label{Eq:hgeneral}
h(z,\bar z)=\sum_{\substack{ k,l\ge0\cr k=l\pmod n}} h_{kl}z^k \bar z^l \,.
\end{equation}
We say that the series $h$ consists of resonant terms as each term of the series is invariant under the rotation by $\alpha_0$.
It is easy to come back from the variables $(z,\bar z)$
to the symplectic polar coordinates by substituting $z=\sqrt{2I} e^{i\varphi}$ and $\bar z=\sqrt{2I}e^{-i\varphi}$.

We note that the formal interpolating Hamiltonian of the resonant normal form theory
(described in the introduction) has the form (\ref{Eq:hgeneral}) and satisfy two additional properties.
First, the series is real-valued, i.e., it has real coefficients when written in terms of $x,y$ variables,
which is equivalent to the condition $h_{kl}=h_{lk}^*$ for all $k,l\ge0$.
Second, the series do not contains terms of order lower than two.

We use tangent-to-identity formal canonical transformations to simplify the series $h$
by eliminating as many resonant terms as possible without braking the
symmetry of the series. The  simplification procedure depends on
the order of the resonance and on some lowest order terms of the series. 
The following proposition describes one of the possible simplifications of the series
which still involve infinitely many coefficients.
We note that no further simplification is  possible as the coefficients of the
simplified series are defined uniquely and consequently 
can be considered as moduli of the formal canonical classification.

\begin{proposition}\label{Pro:h22is0}
Let $h$ be a formal power series of the form
\begin{equation}\label{Eq:hres}
h(z,\bar z)=\sum_{\substack{k+l\ge 3, \;k,l\ge0\cr k=l\pmod n}} h_{kl}z^k \bar z^l \,.
\end{equation}
If  $h_{22}=0$, $ h_{n0}>0$ and $h_{kl}=h_{lk}^*$ for all $k,l$, 
(and additionally $h_{33}\ne0$ for $n\ge6$ only)
then there exists a formal tangent-to-identity
canonical change of variables which 
transforms the Hamiltonian $h$ into $\tilde h$, where $\tilde h$ has the
following form.
\begin{itemize}
\item
If $n=3$,
\begin{equation}\label{Eq:h22is0nis3}
\tilde h(z,\bar z)=(z\bar{z})^{3}\sum_{\substack{k\ge0\cr k\ne 2\pmod 3}} a_k {(z\bar{z})}^k
+(z^{3}+\bar z^3) \sum_{\substack{k\ge0\cr k\ne 2\pmod 3}} b_k {(z\bar{z})}^k ,
\end{equation}%
\item
If $n=4$,
\begin{equation}\label{Eq:h22is0nis4}
\tilde h(z,\bar z)=(z\bar{z})^{4}\sum_{\substack{k\ge0\cr k\ne 3\pmod 4}} a_k {(z\bar{z})}^k
+(z^{4}+\bar z^4) \sum_{\substack{k\ge0\cr k\ne 3\pmod 4}} b_k {(z\bar{z})}^k ,
\end{equation}
\item
If $n=5$,
\begin{equation}\label{Eq:h22is0nis5}
\tilde h(z,\bar z)=(z\bar{z})^{3}\sum_{\substack{k\ge0\cr k\ne 1\pmod 5}} a_k {(z\bar{z})}^k
+(z^{5}+\bar z^5) \sum_{\substack{k\ge0\cr k\ne 4\pmod 5}} b_k {(z\bar{z})}^k ,
\end{equation}
\item
If $n\ge 6$,
\begin{equation}\label{Eq:h22is0nge6}
\tilde h(z,\bar z)=(z\bar{z})^{3}\sum_{k\ge0} a_k {(z\bar{z})}^k
+(z^{n}+\bar z^n) \sum_{\substack{k\ge0\cr k\ne 2\pmod 3}} b_k {(z\bar{z})}^k .
\end{equation}
In all cases the coefficients $a_k$ and $b_k$ are real. They
 are  defined uniquely provided the leading coefficient is normalised by the condition $b_0>0$.
\end{itemize}
\end{proposition}

\begin{proof}
The case $n=3$ is covered by a theorem of \cite{GG2009}.

%%%%%%%%%%%%%%%%%%% New version
\medskip\noindent{\bf Case of $n=4$.}
The proposition is proved by induction. Following the classical strategy, 
we perform a sequence of canonical coordinate changes normalising one order of the formal
Hamiltonian at a time.

In the case of $n=4$, any resonant monomial has an even order
because $k=l\pmod 4$ implies that $k+l$ is even.
So any resonant monomial of order $2m$ has the form
\begin{equation}\label{Eq:Pmj}
Q_{m,j}=z^{m+2j}\bar z^{m-2j}
\end{equation}
where $ |j| \le \left \lfloor{\frac m2}\right\rfloor $. Then  any real homogeneous resonant  polynomial
of order $2m$ has the form
 $\sum_{j=-\left\lfloor{\frac m2}\right\rfloor}^{\left\lfloor{\frac m2}\right\rfloor}c_{j}Q_{mj}$
where   $c_{-j}=c_{j}^*$ due to the realness.
Let $\mathcal H^4_m$ denote the set of all such polynomials.
We conclude that the set $\mathcal H^4_m$ is a real vector space and 
 $\dim\mathcal H^4_m=1+2\left\lfloor{\frac m2}\right\rfloor$.

\medskip

The formal Hamiltonian  can be written in the form $h=\sum_{k\ge 2} h_k$
 where $h_k \in \mathcal H^4_k$. In particular, $h_2(z, \bar z)=h_{4,0}z^4+h_{0,4}\bar z^4$
 where $h_{4,0}=h_{0,4}$ as $h_{4,0}$ is assumed to be real.
We see that $h_2$ already has the desired form.
Setting $b_0=h_{40}$ we write
\[ 
h_2=a_0z^2\bar z^2+b_0(z^4+\bar z^4).
\]
All other orders are transformed to the simplified normal form inductively using the Lie series method
(see e.g. \cite{Deprit1969}).

Let $p\ge2$ and take a polynomial $\chi_p\in\mathcal H^4_p$. After the  substitution
$(z, \bar z ) \to \Phi^1_{\chi_p}(z, \bar z )$ the Hamiltonian takes the form
\[
\tilde h=h+L_{\chi_p}h+\sum_{k\ge2}\frac1{k!}L^k_{\chi_p}h\,,
\]
where
\[
L_{\chi_p} h=
-2i\{ h, \chi_p \} =
-2i \left( \frac{\partial h}{\partial z}\frac{\partial \chi_p}{\partial \bar z}
-\frac{\partial h}{\partial \bar z} \frac{\partial \chi_p}{\partial z}  \right) .
\]
We note that the Lie series are not necessarily convergent 
but the formal sum still has a precise meaning as
$L_{\chi_p}$ is a linear operator which increases the order of a monomial and, consequently,
each term in the formal sum is represented by a finite sum. In particular,
$\tilde h_m=h_m$  for $ 2\le m\le p$
and
\begin{equation}\label{Eq:homol}
\tilde h_{p+1}=h_{p+1}+L_{\chi_p}h_2\,.
\end{equation}
Let  $\Lambda :\mathcal H^4_p\to \mathcal H^4_{p+1}$ be the linear operator defined by
\[
\Lambda \chi=L_{\chi}h_2=
 -4ia_0 \left(z\bar z^2 \frac{\partial \chi }{\partial \bar{z}} -z^2 \bar z \frac{\partial \chi }{\partial {z}}
\right)
-8ib_0\left(
z^3 \frac{\partial \chi }{\partial \bar{z}} - \bar z^3 \frac{\partial \chi }{\partial {z}}
\right).
\]
It is sometimes called the {\em homological operator}. 
We show that $\tilde h_{p+1}$ takes the form stated in the proposition 
after an appropriate  choice of  $\chi_p$.

Using the definition of $\Lambda$ we easily check that 
\begin{eqnarray*}
\Lambda  Q_{p,j}&=&
16jia_0 
Q_{p+1,j}
-8ib_0 \left(  (p-2j)Q_{p+1,j+1}-  (p+2j)Q_{p+1,j-1}\right),
\qquad -\frac p2\le j\le\frac p2,
\end{eqnarray*}
where  $Q_{p,j}$ is defined by (\ref{Eq:Pmj}).
Since  $\chi\in\mathcal H^4_p$, we can write it in the form
\(
\chi=\sum_{j=-{\left\lfloor{\frac p2}\right\rfloor}}^{\left\lfloor{\frac p2}\right\rfloor}
c_{j}Q_{p,j}
\), where
\(
 c_{-j}=c_j^*
\). 
So $c_0$ is real and $c_j$ with $j\ge 1$ may be complex. 
%%Then
%%\[
%%\Lambda\chi=16jia_0Q_{p+1,j}
%%-8ib_0\sum_{j=-{\left\lfloor{\frac {p+1}{2}}\right\rfloor}}^{\left\lfloor{\frac {p+1}{2}}\right\rfloor} \left[ c_{j-1}(p-2j+2)-c_{j+1}(p+2j+2)   \right] Q_{p+1,j}
%%\]
Let
\[
h_{p+1}=\sum_{j=-{\left\lfloor{\frac {p+1}{2}}\right\rfloor}}^{\left\lfloor{\frac {p+1}{2}}\right\rfloor}
d_{j}Q_{p+1,j}\,,\qquad \tilde h_{p+1}=\sum_{j=-{\left\lfloor{\frac {p+1}{2}}\right\rfloor}}^{\left\lfloor{\frac {p+1}{2}}\right\rfloor}
\tilde d_{j}Q_{p+1,j}.
\]
Then equation (\ref{Eq:homol}) is equivalent to the system
\begin{equation}\label{Eq:djcj}
 \tilde d_{j}=d_j
 +16jia_0 c_j
 -8ib_0\left[ (p-2j+2)c_{j-1}-(p+2j+2)c_{j+1}   \right]\,.
\end{equation}
were we follow the convention that $c_k=0$ for $k>k_0=\left\lfloor{\frac p2}\right\rfloor$.
Because of the real symmetry it is sufficient to consider 
$0 \le j \le j_0=\left \lfloor{\frac {p+1}{2}}\right\rfloor$.
We note that $k_0=j_0$ if $p$ is even, and
$k_0=j_0-1$ if $p$ is odd.
Then equation (\ref{Eq:djcj}) with $j=j_0$ has the form
\begin{equation}
\begin{array}{ll}
 \tilde d_{j_0}=d_{j_0}
 +16j_0ia_0 c_{j_0}
 -8ib_0 2 c_{j_0-1},
 \qquad&\mbox{$p$ even}
\\
 \tilde d_{j_0}=d_{j_0}
 -8ib_0 c_{j_0-1}.
 \qquad&\mbox{$p$ odd}
 \end{array}
\end{equation}
We also note that the coefficients in the system (\ref{Eq:djcj}) are purely imaginary,
so the equations for real and imaginary
parts of $c_j$ are disjoint. Let $c_j=c_j'+i c_j''$, $d_j=d_j'+i d_j''$, $\tilde d_j=\tilde d_j'+i\tilde  d_j''$. 
First we analyse the imaginary part of the equations:
\begin{eqnarray*}
 \tilde d_{j}''&=&d_j''
 +16ja_0 c_j'
 -8b_0\left[ (p-2j+2)c_{j-1}'-(p+2j+2)c_{j+1}'   \right], \quad c_k'=0 \mbox{ for } k>k_0
 .\end{eqnarray*}
The real symmetry implies that  
the equation with $j=0$ is an identity $\tilde d_0''=d_0''=0$ as $c_{-1}'=c_1'$. 
So we are left with $j_0$ linear relations which involve $k_0+1$  coefficients $c_k'$
with $0\le k \le k_0$.
We immediately note that  since $b_0\ne0$ we can  set $\tilde d_j''=0$ for $1\le j\le j_0$
by choosing $c_{j-1}'$ recursively starting from $j=j_0$. 
 Thus the imaginary part of $d_j$ is eliminated
completely.

If $p$ is even, $k_0=j_0$ and the coefficient $c_{k_0}'$ remains free.
This freedom is related to the obvious fact: the identity $\Lambda h_2^{p/2}=-2i\{h_2,h_2^{p/2}\}=0$
provides a solution for the homogeneous homological equation when $o$ is even.

Now we analyse the real part of the system (\ref{Eq:djcj})  which takes the form
\begin{eqnarray}\label{Eq:dj}
   \tilde d_{j}'&=&d_j'
 -16ja_0 c_j''
 +8b_0\left[ (p-2j+2)c_{j-1}''-(p+2j+2)c_{j+1}''  \right], \quad c_k''=0 \mbox{ for } k>k_0
\end{eqnarray}
The real symmetry implies $c_0''=0$ and $c_{-1}''=-c_1''$, so the first two equations 
take the form
\begin{eqnarray}
 \tilde d_{0}'&=&d_0'
-16 b_0(p+2)c_{1}'',  
\label{Eq:d0}
\\
   \tilde d_{1}'&=&d_1'
 -16a_0 c_1''
 -8b_0(p+4)c_{2}''.
 \label{Eq:d1}
\end{eqnarray}
In total we obtain $j_0+1$ linear relations which involve $k_0$ coefficients $c_k''$
with $1\le k \le k_0$. Then we need to consider the following cases separately.

If $p=2k_0+1$, then $j_0=k_0+1$. The system (\ref{Eq:dj}) with $2\le j\le j_0$
has a non-degenerate matrix
 and consequently there are unique $c_k''$, $1\le k\le k_0$,
such that $\tilde d_j'=0$ for $2\le j\le j_0$. Then monomials of the form
$$
\tilde d_0 Q_{p+1,0}+\tilde  d_1( Q_{p+1,1}+Q_{p+1,-1}),
\qquad
{\tilde d_0,\tilde d_1\in\mathbb R}.
$$
form  a subspace complementary to the image of the homological operator.

If $p=2k_0$, then $j_0=k_0$ and we have to consider two sub-cases.

If $k_0$ is even (or equivalently $p=0\pmod 4$), then we take
equations (\ref{Eq:dj}) with  $1\le j\le j_0$.
This system has a non-degenerate matrix%
\footnote{In these two cases we obtain a
 system of equations with a tridiagonal matrix. The non-degeneracy is proved using the following rule
for the determinant.
Let $\mathrm A=(a_{ij})$ be a square $k_0\times k_0$ matrix with non-zero elements  on the three
main diagonals only: $a_{jj}=\alpha_j$, $a_{j,j+1}=\beta_j$, $a_{j+1,j}=-\gamma_j$.
Then $\det\mathrm A=K_{k_0}$ where the continuants $K_j$ are defined recursively:
$$
K_0=1,\qquad K_1=\alpha_1,\qquad K_j=\alpha_j K_{j-1}+\beta_{j-1}\gamma_{j-1}K_{j-2}
\quad\mbox{for $j\ge2$}.
$$
The conclusion $K_{k_0}\ne0$ is checked with the help of  a straightforward induction.
}
 and consequently there are unique $c_k$, $1\le k\le k_0$,
such that $\tilde d_j=0$ for $1\le j\le j_0$. Then monomials of the form
$$
\tilde d_0 Q_{p+1,0}
\qquad\mbox{with }
{\tilde d_0\in\mathbb R}
$$
form  a subspace complementary to the image of the homological operator.

If $k_0$ is odd  (or equivalently $p=2k_0=2\pmod 4$), then we take
equations (\ref{Eq:dj}) with  $0\le j\le j_0$, $j\ne1$.
This system has a non-degenerate matrix and consequently there are unique $c_k$, $1\le k\le k_0$,
such that $\tilde d_j=0$ for $0\le j\le j_0$, $j\ne1$. Then monomials of the form
$$
\tilde  d_1( Q_{p+1,1}+Q_{p+1,-1})
\qquad\mbox{with }
{\tilde d_1\in\mathbb R}
$$
form  a subspace complementary to the image of the homological operator.

%%%%%%%%%%%%%%%%%%%%%%%%%%%%%%%%%%%%%%%%%%%%%%%%
Repeating the coordinate changes inductively we conclude that the
Hamiltonian can be transformed to the form
\begin{eqnarray*}
h&=&b_0(z^4+\bar z^4)+\sum_{\substack{p\ne 3\pmod 4\\
p\ge 3}}c_p Q_{p,0}+\sum_{\substack{p\ne1\pmod4\\ p\ge3}}d_p (Q_{p,1}+ Q_{p,-1})
%%\\
%%&=&
%%\sum_{\substack{p\ne 3\pmod 4\\
%%p\ge3}}c_p z^p\bar z^p+\sum_{\substack{p\ne1\pmod4\\ p\ge2}}d_p (z^{p+2}\bar z^{p-2}+z^{p-2}\bar z^{p+2})
\end{eqnarray*}
where $c_p$ and $d_p$ are real. The last series has the desired form (\ref{Eq:h22is0nis4}).

In order to complete the proof we need to establish uniqueness of
the series (\ref{Eq:h22is0nis4}). We note that the transformation
constructed in the first part of the proof is not unique
because the kernel of $\Lambda $ is not always empty.
Nevertheless the normalised Hamiltonian is unique. Indeed, suppose that
$h$ can be transformed to two different simplified normal forms
$\tilde h$ and $\tilde h'$ due to non-uniqueness of
transformations to the normal form. Then there is
a canonical transformation $\phi$ such that
\[
\tilde h'=\tilde h\circ\phi\,.
\]
Since the transformation $\phi$ is tangent to identity
there is a formal real-valued Hamiltonian $\chi$ such that
\[
\phi=\Phi^1_{\chi}\,.
\]
Suppose that $2p$ is the lowest order in the formal series $\chi$.
Then $\tilde h$ and $\tilde h'$ coincide
up to the order $2p$, i.e. $h_k=h_k'$ for $2\le k\le p$, and
\[
\tilde h_{p+1}=\tilde h'_{p+1}+\Lambda (\chi_{p})\,.
\]
Since both $\tilde h_{p+1}$ and $\tilde h'_{p+1}$
are in the  subspace complementary to the image of $\Lambda $,
we conclude that $\Lambda (\chi_{p})=0$ and
$\tilde h_{p+1}=\tilde h'_{p+1}$. Then $\chi_{p}$
is in the kernel of $\Lambda$, thus either $\chi_{p}=0$ if $p$ is odd, or $\chi_{p}=c\, h_2^{p/2}$
for some $c\in\mathbb{R}$ if $p$ is even.
Then the change of variables
\[
\tilde \phi=\Phi^1_{-c {\tilde h}^{p/2}}  \circ\Phi^1_{\chi}
\]
also transforms $\tilde h'$ into $\tilde h$: Indeed, $h\circ \Phi^1_{-c {\tilde h}^{p/2}}  \circ\Phi^1_{\chi}=
h \circ\Phi^1_{\chi}=\tilde h'$.
It is easy to check that there is a formal Hamiltonian $ \chi'$
such that $\tilde \phi=\Phi^1_{\chi'}$ and the lowest order in $\chi'$ 
is at least  $2p+2$.

Repeating the arguments inductively we see that $\tilde h$ and $\tilde h'$
coincide at all orders. Hence the simplified normal form is unique.

\medskip\noindent{\bf Case of $n=5$.}
Let $\mathcal H^5_m$ denote the set of all real-valued polynomials
which can be represented as a sum of resonant monomials
of the order $m$. The Hamiltonian $ h=\sum_{k\ge 5} h_k $, $ h_k \in \mathcal H_k^5 .$
The leading term has the form
\[ h_5=b_0(z^5+\bar z^5).\]
The homological operator takes the form
\[
\Lambda \chi =-2i \{ h_5, \chi \}\,.
\]
It is checked directly that $\Lambda : \mathcal H^5_p\to \mathcal H^5_{p+3}$.
Then we use  induction. Suppose the Hamiltonian $h$ has the desired form
for all orders up to $p+2$ for some $p\ge3$. Then a change of variables
generated by $\chi\in\mathcal H^5_p$ gives a new Hamiltonian
$\tilde h=h\circ \Phi^1_{\chi}$ with the properties
\[
\tilde h_m=h_m \quad {\rm for \ } 5\le m\le p+2
\]
and
\begin{equation}\label{Eq:homol5}
\tilde h_{p+3}=h_{p+3}+\Lambda {\chi}\,.
\end{equation}
We can chose $\chi$ to ensure that $\tilde h_{p+3}$
is in a complement to the image of $\Lambda$.
The results of the study of $\Lambda $ are summarised in Table~\ref{Ta:n5}.
\begin{table}[ht]
\begin{center}
\begin{tabular}{|c|c|c|c|c|}
\hline
$p$ & $\dim\mathcal H^5_p$ & $\dim\mathcal H^5_{p+3}$ & $\dim\ker\Lambda $ &
co{-}dim\ Image $ \Lambda $ \\
\hline
$5k$   & $k+1$ & $k$ &   $1$ & $0$ \\
$5k+1$ & $k  $ & $k+1$ & $0$ & $1$ \\
$5k+2$ & $k+1$ & $k+2$ & $0$ & $1$   \\
$5k+3$ &  $k$ & $k+1$ & $0$ & $1$ \\
$5k+4$ & $k+1$ & $k+2$ & $0$ & $1$  \\
\hline
\end{tabular}
\end{center}
\caption{Properties of the homological operator for $n=5$.\label{Ta:n5}}
\end{table}

The  complement to the image of $\Lambda$ is empty for $p=0\pmod5$. 
Otherwise it can be chosen 
as a real multiple of $ (z \bar z)^{(p+3)/2}$ if $p$ is odd, and of $(z \bar z)^{(p-2)/2} (z^5+\bar z^5 )$
if $p$ is even.

\medskip\noindent{\bf Case of $n=6$.}
In this case, $k=l\pmod n$ implies $k+l$ is even.
Let $\mathcal H^6_m$ denote the set of all real-valued polynomials
which can be represented as a sum of resonant monomials
of the order $2m$.

The Hamiltonian $h$ can be written as $h=\sum_{k\ge 3} h_k$, where $h_k \in \mathcal H^6_k$.
In particular,
 $h_3(z, \bar z)=h_{3,3} z^3 \bar z^3 +h_{6,0}z^6+h_{0,6}\bar z^6$
 where $h_{6,0}>0$. To slightly shorten the notation let us write
\[ 
h_3 (z, \bar z )=a_0 z^3 \bar z^3 +b_0(z^6+\bar z^6),
\]
where $a_0=h_{3,3}$ and $b_0=h_{6,0}$ are real. Then 
the homological operator $\Lambda : \mathcal H^6_p\to \mathcal H^6_{p+2}$ is defined by
\[
\Lambda \chi =-2i \{ h_3, \chi \}.
\]
Then
\[
\tilde h_m=h_m \quad {\rm for \ } 3\le m\le p+1
\]
and
\begin{equation}\label{Eq:homol6}
\tilde h_{p+2}=h_{p+2}+\Lambda {\chi}\,.
\end{equation}
It is convenient to denote the resonant monomials of order $2m$ by
\begin{equation}\label{Eq:Pmj6}
Q_{m,0}=z^{m}\bar z^{m}\qquad\mbox{and}\qquad
Q_{m,j}=z^{m+3j}\bar z^{m-3j}
\end{equation}
for $0 < |j| \le \left \lfloor{\frac m3}\right\rfloor $. Then any resonant polynomial which contains only monomials of the order $2m$
has the form $\sum_{j=-\left\lfloor{\frac m3}\right\rfloor}^{\left\lfloor{\frac m3}\right\rfloor}c_{j}Q_{mj}$.
Taking into account that  $c_{-j}=c_{j}^*$ due to real-valuedness,
we conclude that the real dimension of the space 
$\dim\mathcal H^6_m=1+2\left\lfloor{\frac m3}\right\rfloor$.

We compute the action of the homological operator $\Lambda $ on monomials:
\begin{eqnarray*}
\Lambda Q_{p,0}&=&-12ib_0 \left(pQ_{p+2,1}-pQ_{p+2,-1} \right) ,\\
\Lambda Q_{p,j}&=&-12i \left( -3ja_0Q_{p+2,j}+b_0(p-3j)Q_{p+2,j+1} -b_0(p+3j)Q_{p+2,j-1} \right).
\end{eqnarray*}
The results of the study of $\Lambda $ are summarised in Table~\ref{Ta:n6}.

\begin{table}[ht]
\begin{center}
\begin{tabular}{|c|c|c|c|c|}
\hline
$p$ & $\dim\mathcal H^n_p$ & $\dim\mathcal H^n_{p+2}$ & $\dim\ker \Lambda $ &
co{-}dim\ Image $ \Lambda $ \\
\hline
$3k$   & $2k+1$ & $2k+1$ &   $1$ & $1$ \\
$3k+1$ & $2k+1  $ & $2k+3$ & $0$ & $2$ \\
$3k+2$ & $2k+1$ & $2k+3$ & $0$ & $2$   \\
\hline
\end{tabular}
\end{center}
\caption{Properties of the homological operator for $n \ge 6$.\label{Ta:n6}}
\end{table}

If $p=3k$ a complement to the image of $\Lambda $ is generated by $ (z \bar z)^{p+2} $ and,
if $p=3k+1$ or $p=3k+2$, by $ (z \bar z)^{p+2} $ and $(z \bar z)^{p-1}(z^6+\bar z^6)$.

\bigskip\noindent{\bf Case of $n\ge7$.}
It is convenient to group together terms of the same $\delta$-order.
For a monomial $z^k\bar z^l$ we define its $\delta$-order by
\begin{equation}\label{Eq:deltamon}
\delta(z^k\bar z^l)=
3\left|\frac{k-l}n\right|+\min\{\,k,l\,\}
=
\frac12(k+l)-\frac{n-6}{2n}|k-l|\,.
\end{equation}
Let $\mathcal H^n_m$ denote the set of all real-valued polynomials
which can be represented as a sum of resonant monomials
of the $\delta$-order $m$.

It is convenient to denote resonant monomials by
\begin{equation}\label{Eq:Pmj7}
Q_{m,j}=z^{m+nj-3j}\bar z^{m-3j}\qquad\mbox{and}\qquad
Q_{m,-j}=z^{m-3j}\bar z^{m+nj-3j}
\end{equation}
for $0\le j\le \left\lfloor{\frac m3}\right\rfloor$.
Then any resonant polynomial which contains only monomials of the $\delta$-order $m$
has the form $\sum_{j=- \left\lfloor{\frac m3}\right\rfloor}^{ \left\lfloor{\frac m3}\right\rfloor}c_{j}Q_{mj}$.
Taking into account that  $c_{-j}=c_{j}^*$ due to real-valuedness,
we conclude that the real dimension of the space 
$\dim\mathcal H^n_m=1+2\left\lfloor{\frac m3}\right\rfloor$. 
We can write the original Hamiltonian in the form $ h=\sum_{p \ge 3} h_p $
where $h_p \in \mathcal H_p^n $. The leading term has the form
\begin{equation}\label{Eq:h3n}
h_3 (z, \bar z)= a_0z^3\bar z^3 +b_0 (z^n + \bar z^n).
\end{equation}
The homological operator $\Lambda : \mathcal H^n_p\to \mathcal H^n_{p+2}$ for $p \ge 3$
is defined by
\[
\Lambda \chi =\left [ -2i \{ h_3, \chi \} \right ]_{p+2},
\]
where $[g ]_p$ denotes the sum of all terms of $g$ which have  $\delta$-order $p$.
Using (\ref{Eq:h3n}) we compute the action of the homological operator $\Lambda$ on monomials:
\begin{eqnarray*}
\Lambda Q_{p,0}&=&-2inb_0p \left ( Q_{p+2,1} - Q_{p+2,-1} \right ),
\\
\Lambda Q_{p,j}&=&-2in \left (-3a_0jQ_{p+2,j} +b_0(p-3j)Q_{p+2,j+1} \right ).
\end{eqnarray*}
Let $\chi \in \mathcal H^n_p $. Then it can written in the form
\[
\chi=c_{0}Q_{p,0}+\sum_{j=1}^{\left\lfloor{\frac p3}\right\rfloor}
(c_{j}Q_{p,j}+c_{j}^* Q_{p,-j})
\]
and $\Lambda \chi \in \mathcal H^n_{p+2} $ has the form
\[
\Lambda \chi =d_{0}Q_{p+2,0}+\sum_{j=1}^{\left\lfloor{\frac{p+2}{3}}\right\rfloor}
(d_{j}Q_{p+2,j}+d_{j}^* Q_{p+2,-j}).
\]
The conclusions of the study of $\Lambda$ are the same as in the case of $n=6$ and 
therefore are described by  Table~\ref{Ta:n6}.
It can be checked that, if $p=3k$ a complement to the image of $\Lambda $ is 
generated by $ (z \bar z)^{p+2} $ and,
if $p=3k+1$ or $p=3k+2$,  it is generated by  $ (z \bar z)^{p+2} $ and $ (z \bar z)^{p-1}(z^n+\bar z^n)$.

\medskip

Finally we note that the arguments used to prove the uniqueness on the simplified normal form 
in the case $n=4$ can be
easily modified to prove the uniqueness for $n\ge5$.
\end{proof}

\bigskip

\noindent{\bf Remark.} In the case of $n=6$ it is possible to transform the
Hamiltonian to  an alternative normal form, namely,
\[
\tilde h(z,\bar z)=(z\bar{z})^{3}\sum_{\substack{k\ge0\cr k\ne 5\pmod 6}} a_k {(z\bar{z})}^k
+(z^{n}+\bar z^n) \sum_{\substack{k\ge0\cr k\ne 2\pmod 6}} b_k {(z\bar{z})}^k .
\]
The coefficients of this simplified normal form are unique. This normal form
has an advantage over the form proposed in Proposition~\ref{Pro:h22is0}
as  the condition $h_{33} \ne 0$ can be dropped from the assumptions. 
The derivation of this version of the normal form is similar
to the case of $n=4$ described above.

%%%%%%%%%%%%%%%%%%%%%%%%%%%%%%%%%%%%%%%%%%%%%%%%%%%
%%%%%%%%%%%%%%%%%%%%%%%  FAMILIES  %%%%%%%%%%%%%%%%%%%%%%

\subsection{Families with a small twist\label{Se:Familiesh22}}

\begin{proposition}\label{Pro:Famh22is0}
Let
\begin{equation} \label{Eq:hfam}h(z,\bar z; \delta , \nu )=\delta z \bar z + \nu (z \bar z)^2 +
\sum_{\substack{k,l,m,j \ge 0\cr k=l\pmod n}} h_{klmj}z^k \bar z^l \delta^m \nu^j
\end{equation}
be a formal series with $h_{klmj}=h_{lkmj}^*$, $h_{11mj}=h_{22mj}=0$,  $h_{n000}\ge 0$
(and additionally $h_{3300}\ne0$ for $n\ge6$ only).
There exists a formal canonical change of variables which
transforms the Hamiltonian $h$
into
\begin{itemize}
\item
if $n=3$
\begin{eqnarray*}
\tilde h(z,\bar z; \delta , \nu )= \delta z\bar z +\nu {(z\bar z)}^2 +
(z\bar{z})^{3}\sum_{\substack{k,m,j\ge0\cr k\ne 2\pmod 3}} a_{kmj} {(z\bar{z})}^k \delta^m  \nu^j   \\
+  (z^{3}+\bar z^3) \sum_{\substack{k,m,j\ge0\cr k\ne 2\pmod 3}} b_{kmj} {(z\bar{z})}^k  \delta^m  \nu^j,
\end{eqnarray*}%
\item
if $n=4$
\begin{eqnarray*}
\tilde h(z,\bar z; \delta , \nu )=\delta z\bar z +\nu {(z\bar z)}^2 +
(z\bar{z})^{4}\sum_{\substack{k,m,j\ge0\cr k\ne 3\pmod 4}} a_{kmj} {(z\bar{z})}^k \delta^m  \nu^j
\\ +(z^{4}+\bar z^4) \sum_{\substack{k,m,j\ge0\cr k\ne 3\pmod 4}} b_{kmj} {(z\bar{z})}^k \delta^m  \nu^j ,
\end{eqnarray*}
\item
if $n=5$
\begin{eqnarray*}
\tilde h(z,\bar z; \delta , \nu )=\delta z\bar z +\nu {(z\bar z)}^2
+(z\bar{z})^{3}\sum_{\substack{k,m,j\ge0\cr k\ne 1\pmod 5}} a_{kmj} {(z\bar{z})}^k \delta^m  \nu^j
\\ +(z^{5}+\bar z^5) \sum_{\substack{k,m,j\ge0\cr k\ne 4\pmod 5}} b_{kmj} {(z\bar{z})}^k \delta^m  \nu^j ,
\end{eqnarray*}
\item
if $n\ge 6$
\begin{eqnarray*}
\tilde h(z,\bar z; \delta , \nu )=\delta z\bar z +\nu {(z\bar z)}^2 +
(z\bar{z})^{3}\sum_{k,m,j\ge0} a_{kmj} {(z\bar{z})}^k  \delta^m  \nu^j
\\ +(z^{n}+\bar z^n) \sum_{\substack{k,m,j\ge0\cr k\ne 2\pmod 3}}  b_{kmj} {(z\bar{z})}^k \delta^m  \nu^j .
\end{eqnarray*}
The coefficients of the series $a_{kmj}$ and $b_{kmj}$ are real and defined uniquely by the series $h$
provided the leading order is normalised to ensure $b_{000}>0$.
\end{itemize}
\end{proposition}

\goodbreak
\begin{proof}
\begin{itemize}
\item
$n \ge 6$

%We consider the family for which the Hamiltonian $h$ is defined by the equality (\ref{Eq:hfam}).
Let us define the $\tilde \delta$-order by
\[ \tilde \delta (z^k \bar z^l \delta^m \nu^j) = \delta(z^k \bar z^l)+3m+2j, \]
where $\delta(z^k \bar z^l)$ is defined by (\ref{Eq:deltamon}). Then the lowest $\tilde \delta$-order in (\ref{Eq:hfam}) is 3. Let
\[ h_{(3)}(z,\bar z; \delta , \nu ) = h_3(z,\bar z)=a(z\bar z)^3+b(z^n+\bar z^n),\]
where $a=h_{3300}$ and $b=|h_{n000}|$.
The term of the fourth order is
\[ h_{(4)}(z,\bar z; \delta , \nu )= \delta z\bar z +\nu {(z\bar z)}^2+a_1{(z\bar z)}^4 \]
and already has the declared form.

All higher  $\tilde \delta$-orders are transformed to the desired form inductively. 
Let $h_{(s)}$ has the desired form up to $\tilde \delta$-order $p-1$. 
The term of $\tilde \delta$-order $p$ can be written as
\[ h_{(p)}(z,\bar z; \delta , \nu )=\sum_{m=0}^{\left\lfloor{\frac {p-3}{3}}\right\rfloor} \delta^m
\sum_{j=0}^{\left\lfloor{\frac {p-3m-3}{2}}\right\rfloor} \nu^jh_{mjp}(z,\bar z) \, ,
\]
where $h_{mjp}(z,\bar z) \in \mathcal H^n_{p-3m-2j}$.

Let $\chi \in \mathcal H^n_{p-3M-2J-2}$ for some $M$ and $J$.
After the change of variables $(z, \bar z) \mapsto \Phi_\chi^{\delta^M \nu^J}(z, \bar z)$ the Hamiltonian takes the form
\[
\tilde h = h+\delta^M \nu^J L_\chi h + \sum_{l \ge 2} \frac{\delta^{Ml} \nu^{Jl}}{l!} L_\chi^l h .
\]
It's not difficult to see that
\[ \tilde h_{(s)}(z,\bar z; \delta , \nu )=h_{(s)}(z,\bar z; \delta , \nu ) \quad {\mbox {for } }s \le p-1 \]
and
\[ \tilde h_{mjp}(z,\bar z) =h_{mjp}(z,\bar z) \quad {\mbox {for } } m \ne M,j\ne J, \]
\[
\tilde h_{MJp}=h_{MJp} +\Lambda \chi,
\]
where $\Lambda$ is the homological operator from Section \ref{Se:h22is0}.
It was been shown that there exists such $\chi$ that $\tilde h_{MJp}$ has the desired form.

\item
$n=3,4,5$.

The Hamiltonian $h$ can be written as
\[
h(z, \bar z; \delta , \nu )=\delta z \bar z + \nu ( z \bar z)^2 +b(z^n +\bar z^n)  +\sum_{m+j+s > n} \delta^m \nu^j h_{mjs} (z, \bar z),
\]
where $h_{mjs} (z, \bar z)$ is the homogeneous polynomial of order $s$.
Now let $\tilde \delta$-order be
\[ \tilde \delta (z^k \bar z^l \delta^m \nu^j) = k+l+4m+2j. \]

Let $\chi$ is a homogenous resonance polynomial of order $p-4M-2J \ge 1$ for some $M$ and $J$.

After the change of variables $(z, \bar z) \mapsto \Phi_\chi^{\delta^M \nu^J}(z, \bar z)$

\[
\tilde h = h+\delta^M \nu^J L_\chi h + \sum_{l \ge 2} \frac{\delta^{Ml} \nu^{Jl}}{l!} L_\chi^l h.
\]

So
\[\tilde h_{mjs}=h_{mjs} \mbox{\  for \ } \left[ \begin{array}{l}
4m+2j+s<p \\ 4m+2j+s=p \mbox{\ if\ } m \ne M \mbox{\ and\ } j\ne J.\end{array} \right.\]

\[
\tilde h_{M,J,p+(n-2)-4M-2J}=h_{M,J,p+(n-2)-4M-2J} +\Lambda \chi,
\]
where $\Lambda$ is the homological operator from Section \ref{Se:h22is0} for $n=3,4$ or $5$ respectively.
It was been shown that there exists such $\chi$ that $\tilde h_{p+(n-2), K,M}$ has the desired form.
\end{itemize}
The uniqueness can be proved using the same type of arguments as we used
 for  individual maps in Section~\ref{Se:h22is0}.
\end{proof}

\end{document}